\documentclass{amsart}
\usepackage{amssymb, latexsym}
\usepackage{color}

\author{Tristan L\'{e}ger}
\address[Tristan L\'{e}ger]{Courant Institute of Mathematical Sciences, 251 Mercer Street, New York, NY 10012, USA }
\email{tleger@cims.nyu.edu}

\theoremstyle{plain}
\newtheorem{theorem}{Theorem}

\newtheorem{remark}[theorem]{Remark}
\newtheorem{proposition}[theorem]{Proposition}
\newtheorem{lemma}[theorem]{Lemma}
\newtheorem{corollary}[theorem]{Corollary}

\numberwithin{equation}{section} \numberwithin{theorem}{section}

\begin{document}

\title[Global existence and scattering for quadratic NLS with potential in 3D]
{Asymptotic behavior of quadratic NLS with potential}

\title[Asymptotic behavior of quadratic NLS with potential]
{Global existence and scattering for quadratic NLS with potential in 3D}

\vspace{-0.3in}
\begin{abstract}
In this article we study the asymptotic behavior of a quadratic NLS equation with small, time-dependent potential and small spatially localized initial data. We prove global existence and scattering of solutions. The two main ingredients of the proof are the space-time resonance method and the boundedness of wave operators for the linear Schr\"{o}dinger equation with potential.
\end{abstract}

\maketitle 

\tableofcontents

\section{Introduction}

\subsection{Presentation of the equation}
We consider the following initial value problem, set in $\mathbb{R}^3:$ 
\begin{equation} \label{NLSV}
\begin{cases}
i \partial_t u + \Delta u = V u + u^2 \\
u(t=1)= u_1.
\end{cases}
\end{equation} 
where $V = V(t,x)$ is a complex-valued, time-dependent potential and $u(t,x)$ is complex-valued. \\
We study the global well-posedness and asymptotic behavior of solutions to \eqref{NLSV} for small, spatially localized initial data. 

Such equations arise in physics as toy models for perturbations of the flat nonlinear Schr\"{o}dinger equation (that is when $V=0$) due for example to an exterior electric field. They also come up when studying the stability of ground states or travelling wave solutions of dispersive equations. For example the linearization of NLS around a ground state has, at the second order, the form $i \partial_t u + \Delta u = V u + Q(u,u)$ where $Q(u,u)$ is a quadratic form in $u.$ \\

Such problems have first been studied in the case where $V=0$ and more generally for power type nonlinearities. It was proved in \cite{S} by Strauss that, when the power of the nonlinearity is in a certain range (weak enough to have a good local-wellposedness theory but strictly larger than the so-called Strauss exponent), then small data solutions scatter regardless of the structure of the nonlinearity. For example in dimension 3, this theory requires that the power of the nonlinear part be strictly greater than 2. Therefore this theory does not apply to \eqref{NLSV} even for $V=0.$

To handle quadratic nonlinearities, two techniques were then developed: the normal form method of Shatah (\cite{Sh}) and the vector field method of Klainerman (\cite{Kl}). \\
They allowed for the study of the asymptotic behavior of equations that were not covered by the Strauss theory (for example quadratic Klein-Gordon equations were considered in Shatah's original work, or wave equations satisfying null conditions in Klainerman's).
The problem \eqref{NLSV} with $V=0$ and small data has been treated by N.Hayashi and P.Naumkin (\cite{HN}) using the vector field method. 
Then P. Germain, N. Masmoudi and J. Shatah developed the theory of space-time resonances, that brings together the methods of normal forms and vector fields. They applied it to \eqref{NLSV} with $V=0,$ and proved in \cite{GMS} that small spatially localized solutions exist globally and scatter. 
The method's reach goes in fact far beyond this result and applies to numerous dispersive equations. It has for example been used to prove global existence of solutions to the gravity water waves equation in 3d by the same three authors (\cite{GMSww}, \cite{GMSww2}). A. Ionescu and F. Pusateri refined the method and were able to treat the water waves system in 2d (\cite{IPu1}, \cite{IPu2}). They also studied, with Y. Deng and B. Pausader the gravity-capillary water wave system in three dimensions in \cite{DIPP}. Other  equations were also considered, such as the Euler-Poisson system in 2d in the work of A. Ionescu and B. Pausader (\cite{IP}), the Euler-Maxwell system in 3d by Y. Guo, A. Ionescu and B. Pausader (\cite{GIP}) as well as the Euler-Maxwell system for electrons in 2d by Y. Deng, A. Ionescu, B. Pausader (\cite{DIP}) among others. \\
Some approaches related to the space-time resonance method have been developed, for example for the study of the Gross-Pitaevskii equation in 3d by S. Gustafson, K. Nakanishi and T.-P. Tsai (\cite{GNT}). See also the related paper \cite{MKV} of J. Murphy, R. Killip and M. Visan where this method is revisited in the setting of a cubic-quintic Schr\"{o}dinger equation with a nonvanishing boundary condition at infinity. \\
The method of testing against wave packets of M. Ifrim and D. Tataru has also been used to study the asymptotic behavior of dispersive equations like the one-dimensional cubic Schr\"{o}dinger equation (\cite{IT1}) as well as the water waves system (\cite{IT2}) among others. 
\\

When the potential is present but the nonlinearity is not, S.Agmon proved in (\cite{A}), following earlier work of T. Kato (\cite{Ka}) and S.T. Kuroda (\cite{Ku}), that for potentials $V$ that are short-range, wave operators exist and are complete. Then K.Yajima proved in \cite{Y} that under decay and regularity assumptions on the potential these wave operators are bounded on $W^{k,p}$ in dimensions larger that three. These assumptions were weakened in the 3d case by M.Beceanu (\cite{B}) and M.Beceanu, W.Schlag (\cite{BSmain}, \cite{BSsmall}). In both \cite{Y} and \cite{BSmain} the solution is expanded into a finite sum plus a remainder. The main difference lies in the way the remainder term is handled. Yajima uses a more explicit representation, which requires more decay of the potential. Beceanu and Schlag use an abstract Wiener theorem, which allows them to relax the assumptions made on the potential. 

For time-dependent potentials, a scattering theory has been developed by H. Kitara and K. Yajima (\cite{KY}). Another challenge has been to establish good dispersive and Strichartz estimates for this kind of equations. Indeed this corresponds to a first step towards understanding the linearization of the flat NLS equation around time-dependent soliton type solutions (see for example the work of I. Rodnianski, W. Schlag and A. Soffer \cite{RSS}). 
\\

Problems where both the potential and the nonlinearity are present have also been studied. In \cite{GHW} P. Germain, Z. Hani and S. Walsh considered the following equation set in $\mathbb{R}^3:$
\begin{equation} 
i \partial_t u + \Delta u = V u + \bar{u}^2,
\end{equation} 
where $V$ does not depend on time. 
This nonlinearity is more favorable than $u^2$ from the point of view of space-time resonances since no time resonances are present. It also means that the equation is amenable to a treatment by normal form. The proof adapts the method of \cite{GMS} to the case where $V \neq 0$ using a distorted Fourier transform. However it is not entirely clear how to apply it to \eqref{NLSV}, even if the potential does not depend on time. 

Problems that cannot be treated using only the normal form method have also been considered in one space dimension. For example P. Germain, F. Pusateri and F. Rousset consider the equation $i \partial_t u - \partial_{xx} u + Vu = \vert u \vert^2 u. $ They proved, using space-time resonances and a distorted Fourier transform that small localized solutions decay at the same rate as linear solutions. Moreover they prove that solutions have a modified scattering type behavior, that is they approach solutions to the linear equation up to a logarithmic phase correction. Another similar result was obtained by Naumkin (\cite{N}), under weaker assumptions on the potential. 

\subsection{Main difficulties}
\subsubsection{The quadratic nonlinearity}
We start by explaining the difficulties related to the quadratic nonlinearity, first without the potential. This issue has been treated with the use of the space-time resonances by P.Germain, N.Masmoudi and J.Shatah in \cite{GMS}. We recall the main points here, given their relevance to the present work. 
\\

The goal is to construct global solutions to the equation $i \partial_t u + \Delta u = u^2$ in the weakly nonlinear regime, that is when the nonlinearity can be viewed as a perturbation of the linear equation. To ensure this, the initial data is typically assumed small. 

It is also assumed smooth (say in $H^{10}$), so $u$ can be shown to satisfy the energy estimate
\begin{align*}
\Vert u \Vert_{L^{\infty}_t H^{10}_x}^2 \lesssim \exp \big( \int_1 ^t \Vert u \Vert_{L^{\infty}} ds \big) \Vert u_1 \Vert_{H^{10}_x}^2.
\end{align*}
Therefore proving global existence reduces to showing that $u$ satisfies the same dispersive estimates as linear solutions. Indeed in the setting of the weakly nonlinear regime of the equation, the solution is expected to retain some features of the linear equation. The difficulty with the quadratic nonlinearity is that a straightforward bootstrap argument would fail here, hence the need to develop a new method, the space-time resonance method, which we now explain.

We consider the so-called profile of the solution $f(t) = e^{-i t \Delta} u(t).$ The problem is therefore to prove a global in time bound on $f$ in a Lebesgue space with low index. For example proving a global $L^1$ bound would directly imply that $ \sup_{t \in [1,+\infty)} t^{\frac{3}{2}} \Vert u \Vert_{L^{\infty} _x} \lesssim  1.$ Here since we represent the solution in frequency space, it is more convenient to work in weighted $L^2$ spaces, say $\langle x \rangle L^2.$ We write that
\begin{align*}
&\partial_{\xi_l} \widehat{f}(t,\xi) = \\
& \partial_{\xi_l} \widehat{f}_0 (\xi) +\frac{2}{(2\pi)^3} \int_1 ^t \int_{\mathbb{R}^3} s \xi_l e^{is(\vert \xi \vert^2 - \vert \xi - \eta \vert^2 - \vert \eta \vert^2)} \widehat{f}(s,\eta) \widehat{f}(s,\xi-\eta) d\eta ds + \lbrace \textrm{better terms} \rbrace.
\end{align*}
The obstruction is then the growth in time of the integral, due to the $s$ factor. The space-time resonance method gives two ways to improve the decay in the integral: 
\begin{itemize}
\item Use the time oscillations of the exponential, by integrating by parts in $s$. This is the analog of the averaging method for ordinary differential equations. It essentially replaces the quadratic nonlinearity by a cubic one. This is more favorable since an extra $f$ factor will mean extra decay. This method will only be applicable when the phase does not vanish, namely when there is no \textit{time resonance.}
\item Use the frequency oscillations, by integrating by parts in $\eta.$ This will gain a $1/s$ factor, and hence improve the decay. However this method can only be used if the gradient of the phase $\vert \xi \vert^2 - \vert \eta \vert^2 - \vert \xi - \eta \vert^2$ does not vanish, namely if there is no \textit{space resonance.}
\end{itemize}
Overall, it will be possible to prove the desired bound if both methods can be implemented in complementary regions of the frequency space. This means that the intersection of the time and space resonance sets, called the \textit{space-time resonance} set should be small. For example in \cite{GMS} it is reduced to a single point: the origin. 

\subsubsection{Adding the potential}
We now consider the situation treated in the present article, where the potential is present. Our approach is perturbative in nature, since the potential is allowed to depend on time. As a result, we will assume it small in some sense made precise below. We now give a caricature of the main result of this paper (see Section \ref{mainresult} for the rigorous version). Here $\mathcal{S}$ denotes the Schwartz class.
\begin{theorem} \label{caricature}
If $V$ is small enough in $L^{\infty}_t \mathcal{S}_x,$ $\partial_t V$ small enough in $L^{1}_t \mathcal{S}_x,$ and $e^{-i\Delta} u_1$ small enough in $\langle x \rangle L^2 \cap H^{10}_x,$ then \eqref{NLSV} has a global solution.
\end{theorem}
Note that this is not exactly the main result we show, since the spaces for the potential and initial data in the actual theorem are slightly more technical than in the above caricature. In particular the decay of the potential is made more precise than Schwartz class $\mathcal{S}.$ Moreover we also show that the solution scatters, see Theorem \ref{mainresultscattering} in Section \ref{mainresult}.

When the potential is present, a straightforward application of the space-time resonance theory as in the previous section will fail. For example, when dealing with the potential term, integrating by parts in time no longer produces decay. Indeed the degree of the nonlinearity is not improved (one gets an extra $Vu$ factor in the nonlinearity, which does not decay faster that $u$). 

The idea we use to handle this difficulty is borrowed from linear scattering theory for the Schr\"{o}dinger equation, that is when the nonlinearity $u^2$ is absent. We represent $\partial_{\xi_l} \widehat{f}$ as a series that converges in $L^{\infty}_t L^2 _x$ if the potential is small. This type of approach has been used for example to handle the case of small time-independent potentials by K. Yajima (\cite{Y}), as well as in the time-dependent case by I. Rodnianski and W. Schlag in \cite{RS}. We will see that this method will need to be adapted to our setting. Since we need to keep control of the space localization of the profile (measured say by its $\langle x \rangle L^2 _x$ norm), some difficulties will arise in the generation of the series representation. The procedure will be explained heuristically in the next section \ref{outline}. The difficulty is then to merge this linear strategy with the space-time resonance theory: the most challenging terms to handle will cumulate both the difficulties of the linear problem (convergence of the series representation) and the nonlinear problem (slow decay in time of oscillatory integrals). 

\subsection{Outline of the proof} \label{outline}
We now present our strategy informally. As we explained, the main point is to prove a global in time bound for the profile in $x L^2,$ or equivalently of the derivative of the Fourier transform $\partial_{\xi} \widehat{f}$ in $L^2.$ As outlined in the previous section, we do so by obtaining a series representation for this expression. To explain the idea, we start by showing in the next subsection how to bound the Fourier transform of the profile in $L^2_x.$ This part of the argument is completely classical in the linear Schr\"{o}dinger literature. We elected to present it here, since we use a slightly unusual language (we work with the profile $f(t):=e^{-it\Delta} u(t)$ and carry out the estimates in frequency space). In the following subsection, we explain how this method is adapted to bound the profile in $x L^2.$

\subsubsection{Expanding $\widehat{f}$ as a series} 
We focus on the linear term, since difficulties in the generation of the representation come from this part of the perturbation.

We start with the Duhamel formula for the profile, and we focus on the potential term
\begin{align} \label{potf}
\int_1 ^t \int_{\mathbb{R}^3} e^{is(\vert \xi \vert^2 - \vert \eta_1 \vert^2)} \widehat{f}(s,\eta_1) \widehat{V}(s, \xi-\eta_1) d\eta_1 ds. 
\end{align}
Integrating by parts in time, we obtain
\begin{align*}
\eqref{potf} &=-i \int_{\mathbb{R}^3} e^{it(\vert \xi \vert^2 - \vert \eta_1 \vert^2)} \frac{\widehat{f}(t,\eta_1) \widehat{V}(t,\xi-\eta_1)}{\vert \xi \vert^2 - \vert \eta_1 \vert^2} d\eta_1 \\
             &+i \int_1^t \int_{\mathbb{R}^3} e^{is(\vert \xi \vert^2 - \vert \eta_1 \vert^2)} \frac{\partial_s \widehat{f}(s,\eta_1) \widehat{V}(s,\xi-\eta_1)}{\vert \xi \vert^2 - \vert \eta_1 \vert^2} d\eta_1 ds \\
             & + \lbrace \textrm{better or similar terms} \rbrace \\
             &= -i \int_{\mathbb{R}^3} e^{it(\vert \xi \vert^2 - \vert \eta_1 \vert^2)} \frac{\widehat{f}(t,\eta_1) \widehat{V}(t,\xi-\eta_1)}{\vert \xi \vert^2 - \vert \eta_1 \vert^2} d\eta_1 \\
             &+ \frac{1}{(2\pi)^3} \int_1 ^t \int_{\mathbb{R}^3} \frac{\widehat{V}(s,\xi-\eta_1)}{\vert \xi \vert^2 - \vert \eta_1 \vert ^2} \int_{\mathbb{R}^3} e^{is(\vert \xi \vert^2 - \vert \eta_2 \vert ^2)} \widehat{f}(s,\eta_2) \widehat{V}(s,\eta_1-\eta_2)  d\eta_2 d\eta_1 ds  \\
             & + \lbrace \textrm{better or similar terms} \rbrace.
\end{align*}
Here better terms refer for example to the case where the time derivative falls on the potential. Due to the time integrability of the derivative of the potential, it can essentially be considered as a boundary term. We also ignored the bilinear terms, since, as mentioned above, the need for the series representation comes from the potential part. 

Next we note that the integrated term has a similar structure as the term we started with. Therefore the idea is to keep integrating by parts in time, and we ultimately write $\widehat{f}$ as the sum of the boundary terms that remain. They are of the form
\begin{align} \label{multiV}
\int_{\big(\mathbb{R}^3 \big)^{n-1}} \prod_{l=1}^{n-1} \frac{\widehat{V}(t,\eta_{l-1}-\eta_l)}{\vert \xi \vert^2 - \vert \eta_l \vert^2} d\eta_1 ... d\eta_{n-2} \int_{\mathbb{R}^3} e^{it(\vert \xi \vert^2 - \vert \eta_{n} \vert^2)} \widehat{f}(t,\eta_n) \frac{\widehat{V}(t,\eta_{n-1} - \eta_n)}{\vert \xi \vert^2 - \vert \eta_n \vert^2}   d\eta_n d\eta_{n-1}.
\end{align}
Then there remains to prove that the series above converges. We can rely on estimates similar to those proved by Yajima \cite{Y}, and show that there exists a constant $C>0$ such that for every $n \in \mathbb{N},$
\begin{align*}
\Bigg \Vert \mathcal{F}^{-1}_{\xi} \int_{\mathbb{R}^{3n}}  \widehat{f}(t,\eta_n)  \prod_{l=1}^{n} \frac{\widehat{V}(t,\eta_{l-1}-\eta_l)}{\vert \xi \vert^2 - \vert \eta_l \vert^2} d\eta_1 ... d\eta_{n} \Bigg \Vert_{W^{k,p}} \leqslant C^n \Vert V \Vert_{Y}^n \Vert f \Vert_{W^{k,p}},
\end{align*}
where $\Vert V \Vert_Y$ denotes some norm that we do not give explicitely here. In other words, the boundary terms that arise from repeated integrations by parts define continuous linear operators (in $f$) on $W^{k,p}$ spaces (usual Sobolev spaces), and their operator norms are geometric in some norm of the potential. The smallness of the potential then allows us to conclude. It is important to note that the theorem of K.Yajima assumes the potential to be time independent. Therefore since we are interested in time-dependent potentials in this paper, we will have to use a different approach which extends his result.

\subsubsection{Expanding $\partial_{\xi} \widehat{f}$ as a series}

Now we move to the more difficult task of finding a series representation for $\partial_{\xi} \widehat{f}.$ Starting as before from the Duhamel formula for this quantity, we focus on the case where the $\partial_{\xi}$ falls on the phase, since this is the worse term. If we attempt to apply the same exact strategy as in the previous subsection, we obtain
\begin{align*}
\notag&\int_1 ^t 2 i s \xi_l e^{is \vert \xi \vert^2} \int_{\mathbb{R}^3} \widehat{V}(s,\xi - \eta_1) e^{-is \vert \eta_1 \vert ^2} \widehat{f}(s,\eta_1) d\eta_1 ds \\ \notag&=\int_{\mathbb{R}^3} 2it \xi_l e^{it (\vert \xi \vert^2 - \vert \eta_1 \vert^2)} \frac{\widehat{V}(t,\xi-\eta_1)}{\vert \xi \vert^2 - \vert \eta_1 \vert ^2} \widehat{f}(t,\eta_1) d\eta_1 \\
\notag & -\int_1 ^t 2 i s \xi_l e^{is \vert \xi \vert^2} \int_{\mathbb{R}^3} \frac{\widehat{V}(s,\xi-\eta_1)}{\vert \xi \vert^2 - \vert \eta_1 \vert ^2} e^{-is \vert \eta_1 \vert ^2} \partial_s \widehat{f}(s,\eta_1) d\eta_1 ds \\
\notag& + \lbrace  \textrm{better terms} \rbrace \\
&= \int_{\mathbb{R}^3} 2it \xi_l e^{it (\vert \xi \vert^2 - \vert \eta_1 \vert^2)}\frac{\widehat{V}(t,\xi-\eta_1)}{\vert \xi \vert^2 - \vert \eta_1 \vert ^2} \widehat{f}(t,\eta_1) d\eta_1 \\ 
&-\frac{1}{(2\pi)^3}\int_1 ^t \int_{\mathbb{R}^3} 2 i s \xi_l \frac{\widehat{V}(s,\xi-\eta_1)}{\vert \xi \vert^2 - \vert \eta_1 \vert ^2} \int_{\mathbb{R}^3} e^{is(\vert \xi \vert^2 - \vert \eta_2 \vert ^2)} \widehat{f}(s,\eta_2) \widehat{V}(s,\eta_1-\eta_2)  d\eta_2 d\eta_1 ds  \\
&-\frac{1}{(2\pi)^3}\int_1 ^t \int_{\mathbb{R}^3} 2 i s \xi_l \frac{\widehat{V}(s,\xi-\eta_1)}{\vert \xi \vert^2 - \vert \eta_1 \vert ^2} \int_{\mathbb{R}^3} e^{is(\vert \xi \vert^2 - \vert \eta_2 \vert ^2 - \vert \eta_1 - \eta_2 \vert ^2)} \widehat{f}(s,\eta_2) \widehat{f}(s,\eta_1-\eta_2)  d\eta_2 d\eta_1 ds \\
\notag&+ \lbrace  \textrm{better terms} \rbrace .
\end{align*}
We immediately notice that the main boundary term obtained 
\begin{align*}
\int_{\mathbb{R}^3} t \xi_l e^{it (\vert \xi \vert^2 - \vert \eta_1 \vert^2)}\frac{\widehat{V}(t,\xi-\eta_1)}{\vert \xi \vert^2 - \vert \eta_1 \vert ^2} \widehat{f}(t,\eta_1) d\eta_1
\end{align*}
cannot be bounded uniformly in time on $L^2 _x$ by using the Yajima bounds from the previous section, due to the presence of the $t$ factor. However the strategy can be salvaged away from the singularity of the denominator $\vert \xi \vert^2 - \vert \eta_1 \vert^2,$ since in this case the operator behaves like a standard Fourier multiplier. To deal with the difficulty near the singularity, we can instead integrate by parts in $\eta_n$ first to get rid of the $s$ factor in the integral, and then integrate by parts in time. In this case the main boundary term becomes
\begin{align*}
\int_{\mathbb{R}^3} \xi_l \eta_j e^{it (\vert \xi \vert^2 - \vert \eta_1 \vert^2)}\frac{\widehat{V}(t,\xi-\eta_1)}{\vert \eta_1 \vert^2 (\vert \xi \vert^2 - \vert \eta_1 \vert ^2)} \widehat{f}(t,\eta_1) d\eta_1,
\end{align*}
which we can bound uniformly on $L^2$ using the Yajima bound. Note however that we must be close to the singular set of $\vert \xi \vert^2 - \vert \eta_1 \vert^2,$ so that the multiplier $\frac{\xi_l \eta_j}{\vert \eta_1 \vert^2}$ is not singular. Then we can, as in the previous section, iterate the procedure by differentiating whether we are close to the singular set or not, and choose the appropriate strategy to generate the next terms.

\subsubsection{The bilinear terms}
The difficulty that then arises is that of dealing with the bilinear terms (in the profile) that appear due to the presence of the quadratic part of the nonlinearity. Indeed the time derivative of the profile $f$ has a potential part and a nonlinear part. Typically, these terms present both difficulties of being multilinear in $V$ and bilinear in $f.$ To deal with them, we can adapt the space-time resonance strategy. We must also adapt the Yajima bounds, and here it turns out to be more convenient to rely on the more recent result of M.Beceanu and W.Schlag \cite{BSmain}, which provides a more precise representation in physical space of operators of the type \eqref{multiV}.

\subsection{Organization of the paper}
The paper is organized as follows: in Section \ref{series-repr} we detail the computation of the expansion, presented here informally. Then in Section \ref{firstit} we treat the terms that we discarded here as better terms at the first step of the iteration. Then we move on to estimating the $n-th$ iterates. First, as mentioned earlier, we must prove analogs of bilinear estimates that are tailored to our setting. This is done in Section \ref{multilinanalysis}. Then we apply these result to bound the terms of the series in $L^{\infty}_t L^2 _x$ in Section \ref{mainproppf}. The most challenging part is estimating the $n-$th iterates of the bilinear term, since we must use the space-time resonance method. Finally we prove an energy bound for the solution in Section \ref{energy}. For this part we will also expand the potential term into a series, but the reasoning will be easier since integrations by parts in time alone will be enough: the terms that appear can be controlled using essentially the boundedness of wave operators on $H^{10}.$ Finally we show in Section \ref{scattering} how the estimates proved imply scattering of the solution.

\section{Main result} \label{mainresult}
\subsection{Notations}
Before stating our main theorem we need to introduce some notations:
\\
For the Fourier transform we take the following convention:
\begin{align*}
\widehat{f}(\xi) = \mathcal{F}f (\xi)= \int_{\mathbb{R}^3} e^{-i x \cdot \xi} f(x) dx,
\end{align*}
hence the following definition for the inverse Fourier transform:
\begin{align*}
\check{f}(x) = [\mathcal{F}^{-1} f ](x) = \frac{1}{(2 \pi)^3} \int_{\mathbb{R}^3} e^{i x \cdot \xi} f(\xi) d\xi.
\end{align*}
To define Littlewood-Paley projections, we consider $\phi$ a smooth radial function supported on the annulus $\mathcal{C} = \lbrace \xi \in \mathbb{R}^3 ; \frac{1}{1.04} \leqslant \vert \xi \vert \leqslant 1.04 \times 1.1 \rbrace $ such that
\begin{align*}
\forall \xi \in \mathbb{R}^3 \setminus \lbrace 0 \rbrace, ~~~ \sum_{j \in \mathbb{Z}} \phi\big( 1.1^{-j} \xi \big) =1.
\end{align*}
We will see later (in Lemma \ref{estimateR9}) why we chose to localize at frequency $1.1^j$ and not $2^j$.
\\
Notice that if $j-j'>1$ then $1.1^j \mathcal{C} \cap 1.1^{j'} \mathcal{C} = \emptyset.$ Indeed for this intersection to be nonempty we would need to have 
\begin{align*}
& 1.1^j \frac{1}{1.04} \leqslant 1.1^{j'+1} 1.04  \\
&\Longrightarrow  1.1^{j-j'} \leqslant 1.1 (1.04)^2 = 1.1 \times 1.0816 \\
&\Longrightarrow  j-j' \leqslant 1.
\end{align*}
$P_k (\xi) := \phi(1.1^{-k}\xi)$ will denote the Littlewood-Paley projection at frequency $1.1^k .$ \\
Similarly $P_{\leqslant k} (\xi)$ will denote the Littlewood-Paley projection at frequencies less than $1.1^k.$ \\
We will also sometimes use the notation $\widehat{f_k}(\xi) =P_k (\xi) \widehat{f}(\xi)$ 
\\
\\
Let's define the main norms used in the paper: first to control the profile of the solution we need the following:
\begin{eqnarray*}
\Vert f \Vert_X = \sup_{k \in \mathbb{Z}} \Vert \nabla_{\xi} \widehat{f_k} \Vert_2 .
\end{eqnarray*}
For the potential we introduce the following norms:
\begin{align*}
\Vert V \Vert_{B_x} &= \Vert V \Vert_{\langle x \rangle^{-2} L^2 _x} = \Vert \langle x \rangle^2 V(x) \Vert_{L^2 _x}, \\
\Vert V \Vert_{B'_x} &= \Vert \langle x \rangle V_{\leqslant 1} \Vert_{B_x} + \Bigg( \sum_{k \geqslant 0} 1.1^{20 k} \Vert \langle x \rangle V_{k} \Vert_{B_x}^2 \Bigg)^{1/2}.
\end{align*}

\subsection{Main results}
We can now state the main theorem proved in this paper:
\begin{theorem} \label{mainthm}
Let $V=V(t,x)$ be such that $ V \in L^{\infty}_t B'_x,$ and $ \partial_t V \in L^{1}_t B'_x. $ \\ 
There exists $\varepsilon>0$ such that if $\varepsilon_0, \delta < \varepsilon$ and if $u_{1}$ and $V$ satisfy
\begin{eqnarray*}
\Vert  V \Vert_{L^{\infty}_t B' _x} + \Vert \partial_t V \Vert_{L^1 _t B'_x} & \leqslant & \delta ,\\
\Vert e^{-i \Delta} u_{1} \Vert_{H^{10}} + \Vert e^{-i \Delta} u_{1} \Vert_{X} & \leqslant & \varepsilon_0,
\end{eqnarray*}
then \eqref{NLSV} has a unique global solution. Moreover it satisfies the estimate
\begin{align} \label{mainestimate}
\sup_{t \in [1;\infty)} \Vert u(t) \Vert_{H^{10}} + \Vert e^{-it \Delta} u(t) \Vert_{X} + \sup_{k \in \mathbb{Z}} t \Vert u_k (t) \Vert_{L^6} \lesssim \varepsilon_0.
\end{align}
\end{theorem}
The choice to give the initial data at $t=1$ was made for convenience (we avoid the non integrability at 0 of the decay factor $1/t$ from dispersive estimates). The proof can be adapted without major changes if the initial data is prescribed at $t=0.$ \\ 
\\
Let's comment briefly on the assumptions made here: note that we did not strive for the optimal conditions on the potential or the initial data. It is very likely that they can be weakened without major changes in the proofs. \\
\\
Regarding the initial data, it is standard when implementing the space-time resonances method to require it to be spatially localized near the origin. This is what the $X-$norm encodes. We also need some Sobolev regularity to carry out energy estimates, hence the $H^{10}$ regularity assumption. \\
Regarding the potential, since our proof uses ideas from linear scattering theory, it is natural to have the same kind of hypotheses. To prove $L^p$ boundedness of wave operators, it is standard to require some spatial localization. That is why we impose a weighted Lebesgue condition. Note that the high exponent was chosen here for convenience (since such a weighted $L^2$ norm controls $L^p$ norms for $1 \leqslant p \leqslant 2$). \\ 
In dealing with the $H^{10}$ norm of the solution, we will also essentially need boundedness of wave operators on Sobolev spaces. This is where the $B'_x$ norm condition comes from. It can be thought of as an analog of the condition Yajima requires in his proof of $W^{1,p}$ boundedness in \cite{Y}, namely that $\mathcal{F}(\langle x \rangle^{\sigma} \nabla V) \in L^{3/2}$ for $\sigma > 4/3.$ 
\\
\\
In fact the estimates proved while showing Theorem \ref{mainthm} allow us to describe the asymptotic behavior of the solution:
\begin{theorem} \label{mainresultscattering}
The solution constructed in Theorem \ref{mainthm} scatters in the following sense: \\
There exists a linear operator $W_V (t)$ bounded on $L^{\infty}_t H^{10}_x$ such that $e^{-it\Delta} W_V(t) u(t)$ has a limit in $H^{10}$ as $t \to +\infty.$
\end{theorem}
\begin{remark}
This statement can be seen as a perturbation in $V$ of the usual scattering statements from space-time resonances, where $W_V (t)$ is a correction. Indeed when $V=0,$ then $W_V (t)= identity,$ and we recover the scattering of $f$ in $H^{10}$. This is standard, and proved in Section \ref{scattering}, Theorem \ref{usualscattering}.
\end{remark}
\begin{remark}
A similar scattering statement can be proved for the $X$ norm.
\end{remark}
\subsection*{Acknowledgments} 
The author wishes to thank his PhD advisor Pr. Pierre Germain for suggesting this problem to him as well as for the many very helpful discussions that ensued. He also thanks him for his patience while reading through earlier versions of this paper, and for his valuable comments and suggestions. \\
The author wishes to thank the anonymous referee for their careful reading of the article, as well as an important correction to the scattering statement of an earlier version of the paper. The presentation and readability were also improved thanks to their valuable input.  

\section{Preliminaries} \label{prelim+boot}
In this section, we start by collecting some technical results that will be needed in the argument in the first Subsection \ref{techn}. In the second Subsection \ref{boot}, we initiate the proof of Theorem \ref{mainthm} by reducing it to proving bootstrap estimates.

\subsection{Technical results} \label{techn}
In this section we recall standard technical results that will be used in our proofs: \\
The first one is a standard bilinear estimate. The proof given here will be used later on in the paper. 
\begin{lemma}[Bilinear estimate] \label{bilin}
The following inequality holds
\begin{equation}
\Big \Vert \mathcal{F}^{-1} \int_{\mathbb{R}^3} m(\xi,\eta) \widehat{f}(\xi-\eta) \widehat{g}(\eta) d\eta \Big \Vert_{L^r} \lesssim \Vert \mathcal{F}^{-1}(m(\xi-\eta,\eta)) \Vert_{L^1} \Vert f \Vert_{L^p} \Vert g \Vert_{L^q} ,
\end{equation}
where $1/r = 1/p+1/q.$
\end{lemma}
\begin{proof}
We write that
\begin{align*}
& \mathcal{F}^{-1} \int_{\mathbb{R}^3} m(\xi,\eta) \widehat{f}(\xi-\eta) \widehat{g}(\eta) d\eta \\
&= \frac{1}{(2\pi)^3} \int_{\mathbb{R}^{12}} e^{i x \xi} e^{-i y (\xi-\eta)} e^{-iz \eta} f(y) g(z) m(\xi, \eta) dy dz d\eta d\xi \\
&=\frac{1}{(2\pi)^3} \int_{\mathbb{R}^6} f(y) g(z) \int_{\mathbb{R}^6} e^{i (x-y) \xi} e^{i \eta (y-z)} m(\xi, \eta) d\xi d\eta dy dz \\
&=(2\pi)^3 \int_{\mathbb{R}^6} f(y) g(z) \check{m}(x-y,y-z) dydz \\
& = (2\pi)^3 \int_{\mathbb{R}^6} f(x-y) g(x-y-z) \check{m}(y,z) dydz.
\end{align*}
Now we prove the estimate by duality: let $h \in L^{r'}.$ We consider
\begin{align*}
\int_{\mathbb{R}^6} \check{m}(y,z) \int_{\mathbb{R}^3} h(x) f(x-y) g(x-y-z)dx dydz,
\end{align*}
then we use H\"{o}lder's inequality in $x$ to find that it is bounded by
\begin{align*}
\Vert \check{m} \Vert_{L^1} \Vert h \Vert_{L^{r'}} \Vert f \Vert_{L^p} \Vert g \Vert_{L^q}
\end{align*}
and the result follows.
\end{proof}
We recall the following obvious but useful bound on multipliers:
\begin{lemma}\label{symbol}
Let $m(\xi,\eta)$ and $m'(\xi, \eta)$ be two multipliers such that $\check{m}, \check{m'} \in L^1.$ 
\\ Then 
\begin{align*}
\Vert \big[ \mathcal{F}^{-1} (mm') \big] \Vert_{L^1} \lesssim \Vert \check{m} \Vert_{L^1} \Vert \check{m'} \Vert_{L^1}
\end{align*}
\end{lemma}
For symbols we also have the following bound that will be used several times in the paper:
\begin{lemma} \label{symbolbis}
Let $k,k_1 \in \mathbb{Z}$ be such that $k_1 -k >1.$ Let $m$ be the multiplier defined as  
\begin{align*}
m(\eta, \xi_1) = \frac{P_k (\eta) P_{k_1}(\xi_1 + \eta)}{\vert \eta + \xi_1 \vert^2 - \vert \eta \vert^2}.
\end{align*}
Then the following bound holds
\begin{align*}
\Vert \check{m} \Vert_{L^1} \lesssim 1.1^{-2k_1}.
\end{align*}
\end{lemma}
\begin{proof}
Given the properties of our Littlewood-Paley projections recalled in the notation part of the paper, we have 
\begin{align*}
\vert \xi + \eta_1 \vert & \geqslant \frac{1.1^{k+2}}{1.04} \\
                         & \geqslant \frac{1.1}{1.04} \times 1.1^{k+1} \\
                         & \geqslant 1.057 \times 1.1^{k+1} \\
                         & \geqslant \frac{1.057}{1.04} \vert \eta \vert.
\end{align*}
Therefore 
\begin{align*}
\Vert \check{m} \Vert_{L^1} \lesssim 1.1^{-2k_1}.
\end{align*}
\end{proof}
Another useful symbol bound is provided by the following lemma:
\begin{lemma} \label{symbolbisbis}
Let $k_1,k_2$ and $k_3$ be three integers such that $k_3 \leqslant k_1 -10. $
\\
Let
$$
m(\xi,\eta_1) = \frac{\eta_{1,l} P_{k_1}(\eta_1)P_{k_2}(\xi-\eta_1) P_{k_3}(\xi-2\eta_1)}{\eta_1 \cdot (\xi-\eta_1)}.
$$
Then 
\begin{align*}
\Vert \check{m} \Vert_{L^1} \lesssim 1.1^{-k_1}.
\end{align*}
\end{lemma}
\begin{proof}
Let's first notice that by Lemma \ref{symbol} we have 
\begin{align*}
\Vert \check{m} \Vert_{L^1} & \lesssim  \Bigg \Vert \frac{\eta_{1,l} P_{k_1}(\eta_1) P_{k_3}(\xi-2\eta_1)}{\eta_1 \cdot (\xi-\eta_1)} \Bigg \Vert_{L^1} \big \Vert P_{k_2}(\xi-\eta_1) P_{k_1}(\eta_1) \big \Vert_{L^1} \\
& \lesssim \Bigg \Vert \frac{\eta_{1,l} P_{k_1}(\eta_1) P_{k_3}(\xi-2\eta_1)}{\eta_1 \cdot (\xi-\eta_1)} \Bigg \Vert_{L^1}.
\end{align*}
To estimate this term we write that, using the properties of the Littlewood-Paley localization described in the notation section of the paper, 
\begin{align*}
\vert \eta_1 \cdot (\xi- \eta_1) \vert & \geqslant \vert \eta_1 \vert^2 - \vert \eta_1 \vert \vert \xi-2\eta_1 \vert \\
& \geqslant \vert \eta_1 \vert^2 - \frac{1}{1.1^7} \vert \eta_1 \vert^2 \\
& \gtrsim \vert \eta_1 \vert^2,
\end{align*}
hence the result.
\end{proof}
We will need the following dispersive estimate:
\begin{lemma}\label{dispersive}  
\begin{equation*}
\Vert e^{it \Delta} f_k \Vert_{L^6} \lesssim \frac{\varepsilon_1}{t} .
\end{equation*}
\end{lemma}
\begin{proof}
This is direct consequence of Corollary 2.5.4 in \cite{cazenave:book}. Indeed it gives
\begin{align*}
\Vert e^{it \Delta} f \Vert_{L^6} \lesssim \frac{1}{t} \big(\Vert f \Vert_{L^2} + \Vert \nabla_{\xi} \widehat{f} \Vert_{L^2} \big),
\end{align*}
and the result follows.
\end{proof}
Using this lemma we prove two decay bounds that will be useful in the paper.

\begin{lemma}[Decay of the solution]\label{decay}
Let $u(t) = e^{it\Delta} f(t).$ Assume that $t \simeq 1.1^m. $\\ 
We have
\begin{align*}
\Vert u \Vert_{L^6 _x} &  \lesssim 1.1^{-0.99 m} \varepsilon_1 ,\\
\Vert u \Vert_{L^4 _x} &  \lesssim 1.1^{-0.745 m} \varepsilon_1,
\end{align*}
where $\varepsilon_1$ has been defined in \eqref{bootstrapassumption2}.
\end{lemma}
\begin{proof}
Let's start by splitting frequencies dyadically and distinguishing three types of terms:
\begin{align*}
\Vert u \Vert_{L^6 _x} & \lesssim \sum_{k_1 < -10 m } \Vert u_{k_1} \Vert_{L^6 _x} + \sum_{-10 m \leqslant k_1 \leqslant m} \Vert u_{k_1} \Vert_{L^6 _x} + \sum_{k_1 >m} \Vert u_{k_1} \Vert_{L^6 _x}.
\end{align*}
In the first sum we use Bernstein's inequality and the isometry property of the Schr\"{o}dinger group on $L^2 _x$ to write that
\begin{align*}
\Vert u_{k_1} \Vert_{L^6 _x} \lesssim 1.1^{k_1} \Vert u_{k_1} \Vert_{L^2 _x} \lesssim 1.1^{-9m} 1.1^{k_1/10} \varepsilon_1.
\end{align*}
For the middle sum we use Lemma \ref{dispersive} and since there are $O(m)$ terms, we can write that
\begin{align*}
\sum_{-10 m \leqslant k_1 \leqslant m} \Vert u_{k_1} \Vert_{L^6 _x} \lesssim 1.1^{-0.99 m} \varepsilon_1,
\end{align*}
and finally for the third sum we use Bernstein's inequality, the isometry property of the group as well as the energy bound:
\begin{align*}
\Vert u_{k_1} \Vert_{L^6 _x} \lesssim 1.1^{k_1} \Vert u_{k_1} \Vert_{L^2 _x} \lesssim 1.1^{-9 k_1} \varepsilon_1 \lesssim 1.1^{-8m} 1.1^{-k_1} \varepsilon_1.
\end{align*}
Therefore 
\begin{align*}
\sum_{k_1 > m} \Vert u_{k_1} \Vert_{L^6 _x} \lesssim 1.1^{-0.99m} \varepsilon_1
\end{align*}
and the first inequality follows. \\
The second inequality is proved similarly therefore we skip the proof.
\end{proof}
Now we state the second useful decay bound, whose point is to provide summability in $k_1$ when needed. 
\begin{lemma} \label{disp2}
We have that
\begin{align*}
 \Vert e^{it \Delta} f_{k_1} \Vert_{L^{4/3}_t L^6 _x} & \lesssim \min  \lbrace 1.1^{-5 k_1 / 6} ; 1.1^{k_1 /8} \rbrace \varepsilon_1 .
\end{align*}
\end{lemma}
\begin{proof}
We use dispersive estimates as well as Bernstein's inequality and the energy bound to write that
\begin{align*}
\Vert e^{it \Delta} f_{k_1} \Vert_{L^6 _x} & = \bigg( \int_{\mathbb{R}^3} \vert e^{it\Delta} f_{k_1} \vert^{11/2}  \vert e^{it\Delta} f_{k_1} \vert^{1/2}  dx  \bigg)^{1/6} \\
& \leqslant \Vert e^{it \Delta} f_{k_1} \Vert_{L^{\infty} _x} ^{1/12}  \bigg( \int_{\mathbb{R}^3} \vert e^{it\Delta} f_{k_1} \vert^{11/2} dx  \bigg)^{1/6} \\
& \leqslant \Vert e^{it \Delta} f_{k_1} \Vert_{L^{\infty} _x} ^{1/12} \Vert e^{it \Delta} f_{k_1} \Vert_{L^{11/2} _x} ^{11/12} \\
& \lesssim \min  \lbrace 1.1^{-10 k_1 / 12} ; 1.1^{3 k_1 /24} \rbrace \frac{1}{t^{21/24}} \varepsilon_1,
\end{align*}
and the result follows.
\end{proof}
Finally we control a norm that appears naturally in the estimates. 
\begin{lemma} \label{X'}
Define the $X'-$norm as \begin{align*}
\Vert f \Vert_{X'} = \sup_{k \in \mathbb{Z}} \Vert \big( \nabla_{\xi} \widehat{f} \big) P_k (\xi) \Vert_{L^2}. 
\end{align*}
Then 
\begin{align*}
\Vert f \Vert_{X'} \lesssim \Vert f \Vert_X.
\end{align*}
\end{lemma}
\begin{proof}
We write that for $k \in \mathbb{Z}$ and $l \in \lbrace 1;2;3 \rbrace$ we have
\begin{align*}
\big(\partial_{\xi_l} \widehat{f} \big)P_k (\xi) = \partial_{\xi_l}(\widehat{f_k}) - 1.1^{-k} \widehat{f}(\xi) \phi'(1.1^{-k} \xi) \frac{\xi_l}{\vert \xi \vert}.
\end{align*}
It it enough to control the second term. Given the properties of the Littlewood-Paley localization given in the notation part of the paper, we can write that 
\begin{align*}
\big \Vert \widehat{f}(\xi) \phi'(1.1^{-k} \xi) \frac{\xi_l}{\vert \xi \vert} \big \Vert_{L^2} & = \big  \Vert \widehat{f}(\xi) \phi'(1.1^{-k} \xi) \frac{\xi_l}{\vert \xi \vert} \big(P_{k-1} + P_k + P_{k+1} \big)(\xi) \big \Vert_{L^2} \\
& \lesssim \Vert \widehat{f_{k-1}} \Vert_{L^2} +   \Vert \widehat{f_{k}} \Vert_{L^2} +  \Vert \widehat{f_{k+1}} \Vert_{L^2} .
\end{align*}
Now for each of the three terms that appear, we have that $\vert \xi \vert \leqslant 1.1^{k+10}$ hence $1.1^{-k} \lesssim \frac{1}{\vert \xi \vert}$ and therefore
\begin{align*}
1.1^{-k} \Vert \widehat{f_k} \Vert_{L^2} \lesssim \Bigg \Vert \frac{\widehat{f_k}(\xi)}{\vert \xi \vert} \Bigg \Vert_{L^2},
\end{align*}
and we conclude using Hardy's inequality. The other two terms are treated similarly.
\end{proof}

\subsection{The bootstrap argument} \label{boot}
We can now start the proof of Theorem \ref{mainthm}. The main difficulty is to prove global existence, as well as the dispersive estimate \eqref{mainestimate}.
It is based on a bootstrap argument which we now detail. 

As it is usually the case with space-time resonances, we will work with the profile $f$ defined as $ \widehat{f}(t, \xi) = e^{it \vert \xi \vert^2} \widehat{u}(t,\xi). $ We consider a potential $V(t,x)$ such that
\begin{eqnarray*}
\Vert  V \Vert_{L^{\infty}_t B '_x} +\Vert \partial_t V \Vert_{L^{1}_t B ' _x}  & \leqslant & \delta.
\end{eqnarray*}
Keeping in mind that we treat the potential part perturbatively, we consider strong solutions $u(t,x)$ to \eqref{NLSV} that are given by Duhamel's formula:
\begin{align} \label{Duhamel}
\widehat{f}(t,\xi) &= e^{i \vert \xi \vert^2} \widehat{u_1}(\xi) - \frac{i}{(2\pi)^3} \int_1 ^t e^{is \vert \xi \vert ^2} \int_{\mathbb{R}^3} e^{-is \vert \xi-\eta_1 \vert^2} e^{-i s \vert \eta_1 \vert^2} \widehat{f}(s,\eta_1) \widehat{f}(s,\xi-\eta_1) d\eta_1 ds \\
\notag&-\frac{i}{(2\pi)^3}\int_1 ^t e^{is \vert \xi \vert^2} \int_{\mathbb{R}^3} \widehat{V}(\xi - \eta_1) e^{-is \vert \eta_1 \vert ^2} \widehat{f}(s,\eta_1) d\eta_1 ds.
\end{align}
This first double integral will sometimes be referred to as the bilinear part, and the second as the potential part. 
\\
Now we give a local well-posedness statement for \eqref{NLSV} and the type of initial data prescribed in our main theorem. We write it at an initial time $T$ (and not 1) since we also use this lemma to continue the bootstrap assumptions.
\begin{lemma}[Local Well-posedness] \label{LWP}
Let $T \geqslant 1.$ \\
Let $u(T)$ be such that $e^{-iT \Delta} u(T) \in H^{10} \cap X$ and $\Vert e^{-iT \Delta} u(T) \Vert_{H^{10}} + \Vert e^{-iT \Delta} u(T) \Vert_X \sim \varepsilon.$
\\
Then there exists $T'(\varepsilon)>T$ such that \eqref{NLSV} has a unique solution $u(t)$ such that $e^{-it\Delta}u(t) \in \mathcal{C}([T,T'];H^{10} \cap X).$ 
\end{lemma}
\begin{proof}
We only sketch the proof since it is standard: let $f(t)=e^{-it\Delta} u(t).$ \\
We consider the following map
\begin{align*}
\Phi (f)(t) = f(T) - i \int_T ^t e^{-is \Delta} \bigg( \big( e^{is \Delta} f \big) \big(e^{is\Delta} f \big) \bigg) ds - i \int_T ^t e^{-is \Delta} \bigg( \big( e^{is \Delta} f \big) V  \bigg) ds.
\end{align*}
We take a Fourier transform, take a derivative in $\xi$ and localize, and find that we have the following estimate:
\begin{align*}
\Vert \Phi(f) \Vert_{L^{\infty}_t X} & \leqslant \Vert f(T) \Vert_X + C (t-T)^2 \Vert f \Vert_{L^{\infty}_t H^{10}_x}^2  + C (t-T) \Vert f \Vert_{L^{\infty}_t X}  \Vert f \Vert_{L^{\infty}_t H^{10}_x} \\
&+ C (t-T) \Vert V \Vert_{L^{\infty}_t B_x}  \Vert f \Vert_{L^{\infty}_t H^{10}_x} + C (t-T)^2  \Vert V \Vert_{L^{\infty}_t B_x}  \Vert f \Vert_{L^{\infty}_t H^{10}_x}
\end{align*}
for some constant $C.$ \\
We also have
\begin{align*}
\Vert \Phi(f) \Vert_{L^{\infty}_t H^{10}_x} & \leqslant \Vert f(T) \Vert_{H^{10}_x} + C'(t- T) \Vert f \Vert_{L^{\infty}_t H^{10}_x}^2 + C' (t-T) \Vert V \Vert_{L^{\infty}_t B'_x} \Vert f \Vert_{L^{\infty}_t H^{10}_x}.
\end{align*}
These two estimates prove that $\Phi$ maps a suitably chosen ball (centered at 0 of radius say twice $\Vert f(T) \Vert_X + \Vert f(T) \Vert_{H^{10}}$) to itself if $t-T$ is small enough. \\
Similarly we can prove that this map is a contraction on that same ball (provided $t-T$ is small enough). We conclude by standard arguments that the solution exists and is unique in $ \mathcal{C}([T,T'];H^{10} \cap X)$ for some $T'>T.$   
\end{proof}
We can now set up the bootstrap argument: \\
Let $\varepsilon_1 := A \varepsilon_0$ where $A$ is a large number. \\
Let $T>1.$ We make the following assumptions:
\begin{align} 
\label{bootstrapassumption1} \sup_{t \in [1;T]} \Vert f(t) \Vert_{H^{10}} & \leqslant  \varepsilon_1, \\
\label{bootstrapassumption2} \sup_{t \in [1;T]} \Vert f(t) \Vert_{X} & \leqslant  \varepsilon_1 .
\end{align}

Note that the local-wellposedness result implies that this is true for some $T>1$.

We prove that under these assumptions we have the stronger conclusions
\begin{align}
\label{bootstrapconcl1} \sup_{t \in [1;T]} \Vert f(t) \Vert_{H^{10}}& \leqslant \frac{\varepsilon_1}{2} ,\\
\label{bootstrapconcl2} \sup_{t \in [1;T]} \Vert f(t) \Vert_{X} & \leqslant \frac{\varepsilon_1}{2} .
\end{align}

Then using the local-wellposedness result \ref{LWP} we can extend the solution after $T,$ and conclude that the bootstrap assumptions in fact hold for all times. This gives us the estimate \ref{mainestimate} in the theorem. 
The most challenging inequality turns out to be \eqref{bootstrapconcl2}. It will occupy most of the paper (all following sections, except the last one). First in Section \ref{series-repr} we reduce the proof to proving bounds on certain oscillatory integrals (see Propositions \ref{step1:estimates} and \ref{stepn:estimates}). Then Proposition \ref{step1:estimates} is proved in Section \ref{firstit} and Proposition \ref{stepn:estimates} is proved in Sections \ref{multilinanalysis} and \ref{mainproppf}. 

Finally we finish the proofs of the main theorems \ref{mainthm} and \ref{mainresultscattering} in Section \ref{end}. First we prove the other bootstrap conclusion \eqref{bootstrapconcl1} for the $H^{10}$ norm in a first subsection \ref{energy}. As explained above, this yields Theorem \ref{mainthm}. Then in the second subsection \ref{scattering} we prove Theorem \ref{mainresultscattering}. We explain how it essentially follows from estimates that have already been proved for the bootstrap argument.

\section{The series representation} \label{series-repr}
This section is the first step towards the proof of \eqref{bootstrapconcl2}. As explained in the introduction, we write the derivative of the Fourier transform of the profile as an infinite series using integrations by parts in time in the potential part.

\subsection{First expansion}
In this subsection we obtain the first term of the expansion. \\ 
We start by localizing and taking a derivative in $\xi_l$ of \eqref{Duhamel}:
\begin{align*}
\notag  \partial_{\xi_l} \widehat{f_k}(t,\xi)&= \partial_{\xi_l} \big(e^{i \vert \xi \vert^2} \widehat{u_{1,k}}(\xi) \big) \\
& - \frac{i}{(2\pi)^3} P_k(\xi) \int_{1}^t \int_{\mathbb{R}^3}  e^{is (\vert \xi \vert^2 -\vert \eta_1 \vert^2)} \widehat{f}(s,\eta_1) \partial_{\xi_l} \widehat{V}(s,\xi-\eta_1) d\eta_1 ds \\
& - \frac{i}{(2\pi)^3} 1.1^{-k} \phi'(1.1^{-k} \xi) \frac{\xi_l}{\vert \xi \vert} \bigg[ \int_{1}^t \int_{\mathbb{R}^3}  e^{is (\vert \xi \vert^2 -\vert \eta_1 \vert^2)} \widehat{f}(s,\eta_1) \widehat{V}(s,\xi-\eta_1) d\eta_1 ds \\ 
& +  \int_{1}^t \int_{\mathbb{R}^3}  e^{is (\vert \xi \vert^2 -\vert \eta_1 \vert^2 - \vert \xi-\eta_1 \vert^2)} \widehat{f}(s,\eta_1) \widehat{f}(s,\xi-\eta_1) d\eta_1 ds \bigg] \\ 
& -\frac{i}{(2\pi)^3} P_k(\xi) \int_{1}^t  \int_{\mathbb{R}^3} 2 i s \xi_l e^{is (\vert \xi \vert^2 -\vert \eta_1 \vert^2)} \widehat{f}(s,\eta_1) \widehat{V}(s,\xi-\eta_1) d\eta_1 ds  \\
 & -\frac{i}{(2\pi)^3} P_k(\xi) \int_{1}^t  \int_{\mathbb{R}^3} 2 i s \eta_{1,l} e^{is (\vert \xi \vert^2 -\vert \xi -\eta_1 \vert^2 - \vert \eta_1 \vert^2)} \widehat{f}(s,\eta_1) \widehat{f}(s,\xi-\eta_1) d\eta_1 ds \\
 & -\frac{i}{(2\pi)^3} P_k(\xi) \int_{1}^t  \int_{\mathbb{R}^3} e^{is (\vert \xi \vert^2 -\vert \xi -\eta_1 \vert^2 - \vert \eta_1 \vert^2)} \widehat{f}(s,\eta_1) \partial_{\xi_l} \widehat{f}(s,\xi-\eta_1) d\eta_1 ds .
\end{align*}
Now we split the $\eta_1$ frequency dyadically. Let's denote $k_1$ the corresponding index. \\
We obtain
\begin{align}
\notag \partial_{\xi_l} \widehat{f_k}(t,\xi)&=\partial_{\xi_l} \big(e^{i \vert \xi \vert^2} \widehat{u_{1,k}}(\xi) \big) + \sum_{k_1 \in \mathbb{Z}} \Bigg( \\
\label{R1}&- \frac{i}{(2\pi)^3} P_k(\xi) \int_{1}^t \int_{\mathbb{R}^3}  e^{is (\vert \xi \vert^2 -\vert \eta_1 \vert^2)} \widehat{f_{k_1}}(s,\eta_1) \partial_{\xi_l} \widehat{V}(s,\xi-\eta_1) d\eta_1 ds \\
\label{R1bis}& - \frac{i}{(2\pi)^3} 1.1^{-k} \phi'(1.1^{-k} \xi) \frac{\xi_l}{\vert \xi \vert} \bigg[ \int_{1}^t \int_{\mathbb{R}^3}  e^{is (\vert \xi \vert^2 -\vert \eta_1 \vert^2)} \widehat{f_{k_1}}(s,\eta_1) \widehat{V}(s,\xi-\eta_1) d\eta_1 ds \\ 
\label{R1bisbis}& + \int_{1}^t \int_{\mathbb{R}^3}  e^{is (\vert \xi \vert^2 -\vert \eta_1 \vert^2-\vert \xi-\eta_1 \vert^2)} \widehat{f_{k_1}}(s,\eta_1) \widehat{f}(s,\xi-\eta_1) d\eta_1 ds \bigg] \\ 
\label{B1} & -\frac{i}{(2\pi)^3} P_k(\xi) \int_{1}^t  \int_{\mathbb{R}^3} 2 i s \eta_{1,l} e^{is (\vert \xi \vert^2 -\vert \xi -\eta_1 \vert^2 - \vert \eta_1 \vert^2)} \widehat{f_{k_1}}(s,\eta_1) \widehat{f}(s,\xi-\eta_1) d\eta_1 ds \\
\label{B2} & -\frac{i}{(2\pi)^3} P_k(\xi) \int_{1}^t  \int_{\mathbb{R}^3} e^{is (\vert \xi \vert^2 -\vert \xi -\eta_1 \vert^2 - \vert \eta_1 \vert^2)}\widehat{f_{k_1}}(s,\eta_1) \partial_{\xi_l} \widehat{f}(s,\xi-\eta_1) d\eta_1 ds  \\
\notag &  \underbrace{-\frac{i}{(2\pi)^3} P_k(\xi) \int_{1}^t  \int_{\mathbb{R}^3} 2 i s \xi_l e^{is (\vert \xi \vert^2 -\vert \eta_1 \vert^2)} \widehat{f_{k_1}}(s,\eta_1) \widehat{V}(s,\xi-\eta_1) d\eta_1 ds}_{:={\mathcal{T}}_{k_1}^1} \Bigg) .
\end{align}
\indent The iteration argument developed in this section will allow us to bound the potential term ${\mathcal{T}}_{k_1}^1. $ \\
In the remainder of this section, we explain how to expand this term as an infinite series.

There are different cases to consider, depending on whether the denominator $\displaystyle \frac{1}{\vert \xi \vert^2 - \vert \eta_1 \vert^2}$ is singular or not:\\
\\
\underline{Case 1: $ \vert k_1-k \vert >1$} \\
\\
We integrate by parts in time, and get that: 
\begin{align}
\notag {\mathcal{T}}_{k_1}^1&= \\
\label{R2}&-\frac{2i}{(2\pi)^3} P_k (\xi) \int_{\mathbb{R}^3} \frac{it \xi_l \widehat{V}(t,\xi-\eta_1)}{\vert \xi \vert^2 - \vert \eta_1 \vert^2} e^{it (\vert \xi \vert^2 - \vert \eta_1 \vert^2)} \widehat{f_{k_1}} (t,\eta_1) d\eta_1 \\
\label{R2prime} &+\frac{2i}{(2\pi)^3} P_k (\xi) \int_{\mathbb{R}^3} \frac{i \xi_l \widehat{V}(1,\xi-\eta_1)}{\vert \xi \vert^2 - \vert \eta_1 \vert^2} e^{i (\vert \xi \vert^2 - \vert \eta_1 \vert^2)} \widehat{f_{k_1}} (1,\eta_1) d\eta_1 \\
\label{R2bis} &+\frac{2i}{(2\pi)^3} P_k (\xi) \int_1 ^t \int_{\mathbb{R}^3}  \frac{is \xi_l \partial_s \widehat{V}(s,\xi-\eta_1)}{\vert \xi \vert^2 - \vert \eta_1 \vert^2} e^{is (\vert \xi \vert^2 - \vert \eta_1 \vert^2)} \widehat{f_{k_1}} (s,\eta_1) d\eta_1 ds \\
\label{R6prime} &+ \frac{2i}{(2\pi)^3}P_k (\xi) \int_1 ^t \int_{\mathbb{R}^3}  \frac{i \xi_l \widehat{V}(s,\xi-\eta_1)}{\vert \xi \vert^2 - \vert \eta_1 \vert^2} e^{is (\vert \xi \vert^2 - \vert \eta_1 \vert^2)} \widehat{f_{k_1}} (s,\eta_1) d\eta_1 ds \\
\notag &+ \underbrace{\frac{2i}{(2\pi)^3} P_k (\xi) \int_1 ^t \int_{\mathbb{R}^3}  \frac{is \xi_l \widehat{V}(s,\xi-\eta_1)}{\vert \xi \vert^2 - \vert \eta_1 \vert^2} e^{is (\vert \xi \vert^2 - \vert \eta_1 \vert^2)} \partial_s \widehat{f_{k_1}} (s,\eta_1) d\eta_1 ds}_{:={\mathcal{T}}_{k_1}^{1'}} .
\end{align}
Expanding further the last term, we obtain:
\begin{align}
\notag {\mathcal{T}}_{k_1}^{1'} &= \\ 
\label{B3}& \frac{2i}{(2\pi)^6}  P_k (\xi) \int_1 ^t \int_{\mathbb{R}^3}  \frac{\widehat{V}(s,\xi-\eta_1)}{\vert \xi \vert^2 - \vert \eta_1 \vert^2} P_{k_1}(\eta_1) \\
\notag& \times\int_{ \mathbb{R}^3 }  s \xi_l  \widehat{f}(s,\eta_1-\eta_2) e^{is (\vert \xi \vert^2 - \vert \eta_2 \vert^2 - \vert \eta_1 - \eta_2 \vert^2)}  \widehat{f} (s,\eta_2) d\eta_2 ds d\eta_1 \\
\notag &+ \frac{2i}{(2\pi)^6}  P_k (\xi) \int_1 ^t \int_{\mathbb{R}^3}  \frac{\widehat{V}(s,\xi-\eta_1)}{\vert \xi \vert^2 - \vert \eta_1 \vert^2} P_{k_1}(\eta_1) \\
\notag & \underbrace{ \times \int_{ \mathbb{R}^3 } s \xi_l  \widehat{V}(s,\eta_1-\eta_2) e^{is (\vert \xi \vert^2 - \vert \eta_2 \vert^2)}  \widehat{f} (s,\eta_2) d\eta_2 ds d\eta_1}_{:={\mathcal{T}}_{k_1}^2 } .
\end{align}
We conclude that in this case, we have the expansion
\begin{align*}
{\mathcal{T}}_{k_1} ^1 &= {\mathcal{T}}_{k_1} ^2 + \eqref{R2} + \eqref{R2prime} + \eqref{R2bis} + \eqref{R6prime} + \eqref{B3}. 
\end{align*}
\underline{Case 2: When $\vert k-k_1 \vert \leqslant 1:$} \\
In this case we start by doing an integration by parts in $\eta_1:$ We obtain
\begin{align} 
& \notag  {\mathcal{T}}_{k_1} ^1= \\
\label{R3} &\frac{i}{(2\pi)^3} P_k (\xi) \int_1 ^t \int_{\mathbb{R}^3} \frac{\xi_l \eta_{1,j}}{\vert \eta_1 \vert ^2} \partial_{\eta_{1,j}} \widehat{V}(s,\xi-\eta_1) e^{is( \vert \xi \vert^2 - \vert \eta_1 \vert^2)} \widehat{f_{k_1}}(s,\eta_1) d\eta_1 ds \\
\label{R4} &-\frac{i}{(2\pi)^3} P_k (\xi) \int_1 ^t \int_{\mathbb{R}^3 } \partial_{\eta_{1,j}} \bigg( \frac{\xi_l \eta_{1,j}}{\vert \eta_1 \vert ^2} \bigg) \widehat{V}(s,\xi-\eta_1) e^{is( \vert \xi \vert^2 - \vert \eta_1 \vert ^2)} \widehat{f_{k_1}}(s,\eta_1) d\eta_1 ds \\
\label{R4bis} &-\frac{i}{(2\pi)^3} P_k (\xi) \int_1 ^t \int_{\mathbb{R}^3 } 1.1^{-k_1}  \frac{\xi_l \eta_{1,j}}{\vert \eta_1 \vert ^2} \widehat{V}(s,\xi-\eta_1) \\
\notag & \times e^{is( \vert \xi \vert^2 - \vert \eta_1 \vert ^2)} \widehat{f_{k_1}}(s,\eta_1) \phi' \big(\frac{\vert \eta_1 \vert}{1.1^{k_1}} \big) \frac{\eta_{1,j}}{\vert \eta_1 \vert} d\eta_1 ds \\
\notag & \underbrace{-\frac{i}{(2\pi)^3} P_k (\xi) \int_1 ^{t} \int_{ \mathbb{R}^3} \frac{\xi_l \eta_{1,j}}{\vert \eta_1 \vert ^2} \widehat{V}(s,\xi-\eta_1) e^{is( \vert \xi \vert^2 - \vert \eta_1 \vert^2)} \partial_{\eta_{1,j}} \widehat{f}(s,\eta_1)P_{k_1}(\eta_1) d\eta_1 ds}_{:={\mathcal{T}}_{k_1}^{1'}},
\end{align}
where \eqref{R4bis} comes from the derivative hitting the localization $P_{k_1}.$ \\
Notice that there is an implicit sum on $j$ in the expressions above. \\
Now we introduce a regularization of ${\mathcal{T}}_{k_1}^{1'}$, namely
\begin{align*}
\mathcal{T}_{k_1}^{1', \beta} :=-\frac{i}{(2\pi)^3} P_k(\xi) \int_1 ^t \int_{\mathbb{R}^3 } \frac{\xi_l \eta_{1,j}}{\vert \eta_1 \vert ^2} \widehat{V}(s,\xi-\eta_1) e^{i s( \vert \xi \vert^2 - \vert \eta_1 \vert^2)} e^{-s \beta} \partial_{\eta_{1,j}} \widehat{f}(s,\eta_1)P_{k_1}(\eta_1) d\eta_1 ds
\end{align*}
for $\beta>0.$ \\
We integrate by parts in time in that regularized integral and obtain: 
\begin{align}
\notag {\mathcal{T}}_{k_1}^{1', \beta}&= \\
\label{R5} &\frac{1}{(2\pi)^3} P_k (\xi) \int_{ \mathbb{R}^3 } \frac{\xi_l \eta_{1,j}}{\vert \eta_1 \vert ^2} \frac{\widehat{V}(t,\xi-\eta_1)}{\vert \xi \vert^2 - \vert \eta_1 \vert^2 +i \beta} e^{it( \vert \xi \vert^2 - \vert \eta_1 \vert^2+ i\beta)} \partial_{\eta_{1,j}} \widehat{f}(t,\eta_1)P_{k_1}(\eta_1) d\eta_1 \\
\label{R5prime} &-\frac{1}{(2\pi)^3} i P_k (\xi) \int_{ \mathbb{R}^3 } \frac{\xi_l \eta_{1,j}}{\vert \eta_1 \vert ^2} \frac{\widehat{V}(1,\xi-\eta_1)}{\vert \xi \vert^2 - \vert \eta_1 \vert^2 +i \beta} e^{i( \vert \xi \vert^2 - \vert \eta_1 \vert^2+i\beta)} \partial_{\eta_{1,j}} \widehat{f}(1,\eta_1)P_{k_1}(\eta_1)  d\eta_1 \\
\label{R5bis}&-\frac{1}{(2\pi)^3} i P_k (\xi) \int_1 ^t \int_{\mathbb{R}^3} \frac{\xi_l \eta_{1,j}}{\vert \eta_1 \vert ^2} \frac{\partial_s \widehat{V}(s,\xi-\eta_1)}{\vert \xi \vert^2 - \vert \eta_1 \vert^2 + i \beta} e^{i s( \vert \xi \vert^2 - \vert \eta_1 \vert^2+i\beta)} \partial_{\eta_{1,j}} \widehat{f}(s,\eta_1) P_{k_1}(\eta_1)  d\eta_1 ds \\
\notag &\underbrace{-\frac{1}{(2\pi)^3} P_k (\xi) \int_1 ^t \int_{\mathbb{R}^3 } \frac{\xi_l \eta_{1,j}}{\vert \eta_1 \vert ^2} \frac{\widehat{V}(s,\xi-\eta_1)}{\vert \xi \vert^2 - \vert \eta_1 \vert^2 + i \beta} e^{i s( \vert \xi \vert^2 - \vert \eta_1 \vert^2+i\beta)} \partial_s \partial_{\eta_{1,j}} \widehat{f}(s,\eta_1)P_{k_1}(\eta_1)  d\eta_1 ds}_{:={\mathcal{T}}_{k_1}^{1'',\beta}} . 
\end{align}
Now we use \eqref{Duhamel} to find that 
\begin{align*}
\partial_t \partial_{\xi_j} \widehat{f}(t,\xi) &=\frac{2}{(2\pi)^3} t \xi_j e^{it \vert \xi \vert^2} \int_{\mathbb{R}^3} \widehat{V}(t,\xi - \eta_1) e^{-it \vert \eta_1 \vert ^2} \widehat{f}(t,\eta_1) d\eta_1 \\
\notag&-\frac{i}{(2\pi)^3} e^{it \vert \xi \vert^2} \int_{\mathbb{R}^3} \widehat{x_j V}(t,\xi - \eta_1) e^{-it \vert \eta_1 \vert ^2} \widehat{f}(s,\eta_1) d\eta_1 \\
\notag&+\frac{2}{(2\pi)^3}  \int_{\mathbb{R}^3} t \eta_{1,j} e^{it (\vert \xi \vert^2 - \vert \xi-\eta_1 \vert^2 - \vert \eta_1 \vert^2)} \widehat{f}(t,\eta_1) \widehat{f}(t,\xi- \eta_1) d\eta_1 \\
\notag&-\frac{i}{(2\pi)^3} \int_{\mathbb{R}^3} e^{it (\vert \xi \vert^2 - \vert \xi-\eta_1 \vert^2 - \vert \eta_1 \vert^2)} \widehat{f}(t,\eta_1) \partial_{\xi_j} \widehat{f}(t,\xi- \eta_1) d\eta_1 ,
\end{align*}
and we plug that expression back into ${\mathcal{T}}_{k_1}^{1'',\beta}$ and get that
\begin{align}
\notag \mathcal{T}_{k_1}^{1'',\beta} &= \\             
                            &\label{R8} \frac{i}{(2\pi)^6} \int_1 ^t e^{-\beta s} \int_{ \mathbb{R}^3 } \frac{\xi_l \eta_{1,j}}{\vert \eta_1 \vert ^2} \frac{P_k (\xi) P_{k_1} (\eta_1) \widehat{V}(s,\xi-\eta_1)}{\vert \xi \vert^2 - \vert \eta_1 \vert^2 + i \beta} e^{i s( \vert \xi \vert^2 - \vert \eta_1 \vert^2)} \\
                             & \notag \times \int_{\mathbb{R}^3} \widehat{x_j V}(\eta_1 - \eta_2) \widehat{f}(s,\eta_2) d\eta_1 d\eta_2 ds \\
                             &\label{R9} +\frac{2i}{(2\pi)^6} \int_1 ^t e^{-\beta s} \int_{\mathbb{R}^3 } \frac{\xi_l \eta_{1,j}}{\vert \eta_1 \vert ^2} \frac{P_k (\xi) P_{k_1} (\eta_1) \widehat{V}(s,\xi-\eta_1)}{\vert \xi \vert^2 - \vert \eta_1 \vert^2 + i \beta} \\
                             &\notag ~~~~~~~~~ \times \int_{\mathbb{R}^3} is \eta_{2,j} e^{is (\vert \xi \vert^2 - \vert \eta_1-\eta_2 \vert^2 - \vert \eta_2 \vert^2)} \widehat{f}(s,\eta_2) \widehat{f}(s,\eta_1- \eta_2) d\eta_2 d\eta_1 ds \\
                             &\label{R10} + \frac{i}{(2\pi)^6}  \int_1 ^t e^{-\beta s} \int_{\mathbb{R}^3 } \frac{\xi_l \eta_{1,j}}{\vert \eta_1 \vert ^2} \frac{P_k (\xi) P_{k_1} (\eta_1) \widehat{V}(s,\xi-\eta_1)}{\vert \xi \vert^2 - \vert \eta_1 \vert^2 + i \beta} 
                             \end{align}
                             \begin{align*}
                             & \notag ~~~~~~~~ \times \int_{\mathbb{R}^3} e^{is (\vert \xi \vert^2 - \vert \eta_1-\eta_2 \vert^2 - \vert \eta_2 \vert^2)} \widehat{f}(s,\eta_2) \partial_{\eta_{1,j}} \widehat{f}(s,\eta_1- \eta_2) d\eta_2 d\eta_1 ds  \\
                              &\notag + \frac{2i}{(2\pi)^6} \int_1 ^t \int_{\mathbb{R}^3 } \frac{P_k (\xi) P_{k_1} (\eta_1) \widehat{V}(s,\xi-\eta_1)}{\vert \xi \vert^2 - \vert \eta_1 \vert^2 + i \beta} e^{-\beta s} \\
                            & \notag \underbrace{\times \int_{\mathbb{R}^3} i s \xi_l \widehat{V}(\eta_1 - \eta_2) e^{i s( \vert \xi \vert^2 - \vert \eta_2 \vert^2)} \widehat{f}(s,\eta_2) d\eta_2 d\eta_1 ds}_{:={\mathcal{T}}_{k_1}^{2,\beta}} .
\end{align*}
\\
(Note that the factor $\displaystyle \frac{\xi_l \eta_{1,j}}{\vert \eta_1 \vert^2}$ is not present in ${\mathcal{T}}_{k_1}^{2,\beta}$ due to the summation in $j$). 
The manipulations we did in this case yield in conclusion
\begin{align*}
{\mathcal{T}}_{k_1}^{1} = \lim_{\beta \to 0, \beta >0} \bigg( {\mathcal{T}}_{k_1}^{2,\beta} + \eqref{R3} + \eqref{R4} + \eqref{R4bis} +  \eqref{R5} + \eqref{R5prime} + \eqref{R5bis} + \eqref{R8} + \eqref{R9} + \eqref{R10} \bigg).
\end{align*}
\begin{remark}[Desingularization process] \label{desing}
To estimate $\mathcal{T}_{k_1}^{1},$ we will prove uniforms bounds (in $\beta$) for the regularized terms. Indeed by lower semi-continuity of the norm, 
\begin{align*}
\Vert \mathcal{T}_{k_1}^{1'} \Vert_{L^{\infty}_t L^2 _x}  = \Big \Vert \lim_{\beta \to 0, \beta >0} \mathcal{T}_{k_1}^{1',\beta} \Big \Vert_{L^{\infty}_t L^2 _x}  \leqslant \liminf_{\beta \to 0, \beta >0} \big \Vert \mathcal{T}_{k_1}^{1',\beta} \big \Vert_{L^{\infty}_t L^2 _x}.
\end{align*}
In the sequel every estimate on regularized terms will be uniform on $\beta$. Therefore to make our expressions more legible we will drop the regularizing factors in the rest of the paper.
\end{remark}
\noindent \underline{Conclusion for the first expansion:} \\
Putting together the manipulations from both cases, we proved the following lemma:
\begin{lemma}
We have
\begin{align*}
P_k (\xi) \partial_{\xi_l} \widehat{f}(t,\xi) &= \sum_{k_1 \in \mathbb{Z}} \eqref{R1} + \eqref{R1bis} + \eqref{R1bisbis} + \eqref{B1} + \eqref{B2} \\
& + \sum_{\vert k-k_1 \vert >1}  \eqref{R2} + \eqref{R2prime} + \eqref{R2bis} + \eqref{R6prime} + \eqref{B3} \\
&+ \sum_{\vert k-k_1 \vert \leqslant 1} \Bigg(  \eqref{R3} + \eqref{R4} + \eqref{R4bis} \\
&+  \eqref{R5} + \eqref{R5prime} + \eqref{R5bis} + \eqref{R8} + \eqref{R9} + \eqref{R10}\Bigg) \\
&+ \sum_{k_1 \in \mathbb{Z}} \mathcal{T}_{k_1}^2 . 
\end{align*}
\end{lemma}
We will show in Section 5 that the following estimates hold:
\begin{proposition} \label{step1:estimates} We have the following estimates:
\begin{align*}
\Bigg \Vert \sum_{k_1 \in \mathbb{Z}} \eqref{R1} + \eqref{R1bis} + \eqref{R1bisbis} + \sum_{\vert k-k_1 \vert >1} \eqref{R2}+\eqref{R2prime} + \eqref{R2bis} + \eqref{R6prime}  \\
 +\sum_{\vert k- k_1 \vert \leqslant 1} \eqref{R3}+\eqref{R4} +\eqref{R4bis} + \eqref{R5} + \eqref{R5prime}+ \eqref{R5bis} \Bigg \Vert_{L^{\infty}_t L^2 _x} & \lesssim  \varepsilon_1 \delta. 
\end{align*}
Moreover
\begin{align*}
\Bigg \Vert \sum_{k_1 \in \mathbb{Z}} \eqref{B1} +\eqref{B2} \Bigg \Vert_{L^{\infty}_t L^2 _x}  \lesssim  \varepsilon_1 ^2, \ \ \ 
\Bigg \Vert \sum_{\vert k-k_1 \vert >1} \eqref{B3} \Bigg \Vert_{L^{\infty}_t L^2 _x}  \lesssim  \varepsilon_1 ^2 \delta .
\end{align*}
\end{proposition}
\begin{remark}
In this proposition we included two terms that are not strictly speaking part of the iteration, but that appear when taking a derivative of $\widehat{f}(\xi)P_k (\xi),$ namely \eqref{R1}, \eqref{R1bis} and \eqref{R1bisbis}.
\end{remark}
The proof of this proposition will be the object of Section \ref{firstit}. \\
We will also see that we have the following as a consequence of Section \ref{mainproppf}:
\begin{proposition} \label{step2:estimates} We have the following estimates:
\begin{align*}
\Bigg \Vert \sum_{\vert k- k_1 \vert \leqslant 1} \eqref{R8} \Bigg \Vert_{L^{\infty}_t L^2 _x} & \lesssim  \varepsilon_1 \delta^2 , \ \ \ 
\Bigg \Vert \sum_{\vert k-k_1 \vert \leqslant 1} \eqref{R9} +\eqref{R10} \Bigg \Vert_{L^{\infty}_t L^2 _x} & \lesssim  \varepsilon_1 ^2 \delta .
\end{align*}
\end{proposition}
Now there remains to handle $\mathcal{T}_{k_1}^2$. It will not be estimated in its current form. Informally speaking, notice that they look very similar to the integral term we started with (namely $ {\mathcal{T}}_{k_1}^1 $), but with a $V$ factor on the left. Therefore to treat this term, we will repeat the procedure we just carried out for ${\mathcal{T}}_{k_1}^1 $ on the integral in $\eta_1$ in these terms. This is what we set out to do in the remainder of this section. 

\subsection{Further expansions}
Here we explain how to deal with ${\mathcal{T}}_{k_1}^2$. As mentioned previously, this will be done by repeating the procedure used on $ {\mathcal{T}}_{k_1}^1$. Ultimately we will obtain an expression of $\partial_{\xi_l} \widehat{f}$ as an infinite series. \\
\\
The first step is to recombine part of the sum over $k_1.$ More precisely we write 
\begin{align*}
\sum_{\vert k-k_1 \vert >1} \mathcal{T}_{k_1}^2 &= \frac{2i}{(2\pi)^6} \sum_{\vert k-k_1 \vert >1} \int_1 ^t \int_{\mathbb{R}^3} \frac{P_k (\xi) P_{k_1}(\eta_1) \widehat{V}(s,\xi-\eta_1)}{\vert \xi \vert^2 - \vert \eta_1 \vert^2}  \\
& \times \int_{\mathbb{R}^3} is \xi_l \widehat{V}(s,\eta_1-\eta_2) e^{is (\vert \xi \vert^2 - \vert \eta_2 \vert^2)}  \widehat{f} (s,\eta_2) d\eta_2 d\eta_1 ds \\
&= \frac{2i}{(2\pi)^6}\int_1 ^t \int_{\mathbb{R}^3} \frac{P_k (\xi) P_{\leqslant k-2}(\eta_1) \widehat{V}(s,\xi-\eta_1)}{\vert \xi \vert^2 - \vert \eta_1 \vert^2} \\
& \times \int_{\mathbb{R}^3 } is \xi_l \widehat{V}(s,\eta_1-\eta_2) e^{is (\vert \xi \vert^2 - \vert \eta_2 \vert^2)}  \widehat{f} (s,\eta_2) d\eta_2 d\eta_1 ds \\
& + \frac{2i}{(2\pi)^6}\sum_{k_1 >k+1} \int_1 ^t \int_{\mathbb{R}^3} \frac{P_k (\xi) P_{k_1}(\eta_1) \widehat{V}(s,\xi-\eta_1)}{\vert \xi \vert^2 - \vert \eta_1 \vert^2} \\
& \times \int_{ \mathbb{R}^3 } is \xi_l \widehat{V}(s,\eta_1-\eta_2) e^{is (\vert \xi \vert^2 - \vert \eta_2 \vert^2)}  \widehat{f} (s,\eta_2) d\eta_2 d\eta_1 ds.
\end{align*}
Therefore we are left with three terms that were not estimated:
\begin{align*}
\notag \sum_{k_1 \in \mathbb{Z}}  {\mathcal{T}}_{k_1}^2  & = \\
&\frac{2i}{(2\pi)^6}\int_1 ^t \int_{\mathbb{R}^3} \frac{P_k (\xi) P_{\leqslant k-2}(\eta_1) \widehat{V}(s,\xi-\eta_1)}{\vert \xi \vert^2 - \vert \eta_1 \vert^2} \\
&\notag \times \int_{\mathbb{R}^3 } is \xi_l \widehat{V}(s,\eta_1-\eta_2) e^{is (\vert \xi \vert^2 - \vert \eta_2 \vert^2)}  \widehat{f} (s,\eta_2) d\eta_2 d\eta_1 ds\\
&+\frac{2i}{(2\pi)^6} \sum_{k_1 >k+1}  \int_1 ^t \int_{\mathbb{R}^3} \frac{P_k (\xi) P_{k_1}(\eta_1) \widehat{V}(s,\xi-\eta_1)}{\vert \xi \vert^2 - \vert \eta_1 \vert^2} \\
& \notag \times \int_{\mathbb{R}^3 } is \xi_l \widehat{V}(s,\eta_1-\eta_2) e^{is (\vert \xi \vert^2 - \vert \eta_2 \vert^2)}  \widehat{f} (s,\eta_2) d\eta_2 d\eta_1 ds \\
& + \frac{2i}{(2\pi)^6} \sum_{\vert k-k_1 \vert \leqslant 1}  \int_1 ^t  \int_{\mathbb{R}^3} \frac{P_k (\xi) P_{k_1}(\eta_1) \widehat{V}(s,\xi-\eta_1)}{\vert \xi \vert^2 - \vert \eta_1 \vert^2 } \\
& \notag \times \int_{ \mathbb{R}^3 } is \xi_l \widehat{V}(s,\eta_1-\eta_2) e^{is (\vert \xi \vert^2 - \vert \eta_2 \vert^2)}  \widehat{f} (s,\eta_2) d\eta_2 d\eta_1 ds \\
& := \mathcal{T}_{k_1 \leqslant k-2} ^2 + \sum_{k_1 \geqslant k-1} \mathcal{T}_{k_1} ^2. 
\end{align*}
We carry out the following operations on these terms: \\
We start by splitting dyadically the frequency $\eta_2$ (and denote $k_2$ the corresponding index). We denote
\begin{align*}
 \mathcal{T}_{k_1 \leqslant k-2} ^2 + \sum_{k_1 \geqslant k-1} \mathcal{T}_{k_1} ^2 := \sum_{k_2 \in \mathbb{Z}} \mathcal{T}_{k_1 \leqslant k-2, k_2} ^2 +\sum_{k_1 \geqslant  k-1, k_2 \in \mathbb{Z}} {\mathcal{T}}_{k_1,k_2} ^2.
\end{align*}
We now reproduce the reasoning done for ${\mathcal{T}}_{k_1}^1$ in the previous section, this time for ${\mathcal{T}}_{k_1,k_2}^2, {\mathcal{T}}_{k_1 \leqslant k-2,k_2}^2 .$ 
We distinguish two cases depending on whether $\vert k-k_2 \vert >1$ or $\vert k-k_2 \vert \leqslant 1.$ \\ We show computations in the case $k_1>k+1$, since the other cases are similar. \\
\\
\underline{Case 1: $\vert k-k_2 \vert >1$} \\
We integrate by parts in time. We obtain
\begin{align}
\notag  {\mathcal{T}}_{k_1,k_2}^2  &= \\
&\label{2R2} -\frac{2}{(2\pi)^6} \int_{\mathbb{R}^3 } \frac{P_k (\xi) P_{k_1} (\eta_1) \widehat{V}(t,\xi-\eta_1)}{\vert \xi \vert^2 - \vert \eta_1 \vert^2} \\
& \notag \times \int_{\mathbb{R}^3 } it \xi_l e^{it(\vert \xi \vert^2 - \vert \eta_2 \vert^2)} \frac{\widehat{V}(t,\eta_1-\eta_2)}{\vert \xi \vert^2 - \vert \eta_2 \vert^2} \widehat{f_{k_2}}(t,\eta_2) d\eta_2 d\eta_1 \\
\notag-& \frac{2i}{(2\pi)^9} \int_1^t \int_{ \mathbb{R}^3 } \frac{P_k (\xi) P_{k_1} (\eta_1) \widehat{V}(s,\xi-\eta_1)}{\vert \xi \vert^2 - \vert \eta_1 \vert^2} \int_{\mathbb{R}^3 } is \xi_l e^{is(\vert \xi \vert^2 - \vert \eta_2 \vert^2)} \frac{\widehat{V}(s,\eta_1-\eta_2)}{\vert \xi \vert^2 - \vert \eta_2 \vert^2} \\
& \notag \underbrace{\times P_{k_2}(\eta_2) e^{is\vert \eta_2 \vert^2} \int_{\mathbb{R}^3 } e^{-is \vert \eta_3 \vert^2} \widehat{f}(s,\eta_3) \widehat{V}(s,\eta_2-\eta_3) d\eta_3 d\eta_2 ds d\eta_1}_{:={\mathcal{T}}_{k_1,k_2}^3}  \\
& \notag + \lbrace \textrm{similar terms 1} \rbrace.
\end{align}
Note that there is an implicit summation in $j$ in the terms above. \\
By similar terms we mean the other boundary term, the terms obtained when the time derivative hits the $V$ factors or the $s$, as well as the bilinear part of the term obtained when it hits the profile. We will see in Section 7 that the method used to estimate \eqref{2R2} can be straightforwardly adapted to estimate all these better terms. For completeness we included the precise expressions of these terms in the appendix.
\\
\\
The first term \eqref{2R2} is the main boundary term. It is similar to the term \eqref{R2} from the first expansion. The term ${\mathcal{T}}_{k_1,k_2}^3$ is the potential part of the term obtained when the time derivative hits the profile. It is similar to the term we started with (namely ${\mathcal{T}}_{k_1}^2 $), the main difference being the extra $V$ factor on the left. To get the full series representation, we will expand this term further at the next stage, following a similar strategy.  
\\
\\
\underline{Case 2: $\vert k-k_2 \vert \leqslant 1$} \\
We start by integrating by parts in $\eta_2$ and obtain 
\begin{align}
\notag {\mathcal{T}}_{k_1,k_2}^2 &= \\
&\label{2R3} \frac{i}{(2\pi)^6} \int_1 ^t \int_{\mathbb{R}^3 } \frac{P_k (\xi) P_{k_1} (\eta_1) \widehat{V}(s,\xi-\eta_1)}{\vert \xi \vert^2 - \vert \eta_1 \vert^2} \\
& \notag \times \int_{\mathbb{R}^3 } e^{is(\vert \xi \vert^2 - \vert \eta_2 \vert^2)} \frac{\xi_l \eta_{2,j}}{\vert \eta_2 \vert^2} \partial_{\eta_{2,j}} \widehat{V}(s,\eta_1-\eta_2)  \widehat{f_{k_2}}(s,\eta_2) d\eta_2 d\eta_1 ds \\
&\notag +\frac{i}{(2\pi)^6} \int_1 ^t \int_{\mathbb{R}^3 } \frac{P_k (\xi) P_{k_1} (\eta_1) \widehat{V}(s,\xi-\eta_1)}{\vert \xi \vert^2 - \vert \eta_1 \vert^2} \\
&\notag \underbrace{\times \int_{\mathbb{R}^3 } e^{is(\vert \xi \vert^2 - \vert \eta_2 \vert^2)} \frac{\xi_l \eta_{2,j}}{\vert \eta_2 \vert^2} \widehat{V}(s,\eta_1-\eta_2) P_{k_2} (\eta_2) \partial_{\eta_{2,j}} \widehat{f}(s,\eta_2) d\eta_2 d\eta_1 ds}_{:={\mathcal{T}}_{k_1,k_2}^{2'}} \\
& \notag + \lbrace \textrm{similar terms 2} \rbrace.
\end{align}
Here by similar terms we mean the term obtained when the $\eta_2$ derivative falls on the localization or on the kernel. They are written explicitely in the appendix. \\
Again, we will see that the method used to estimate \eqref{2R3} can be adapted to these easier terms. \\
Note also that \eqref{2R3} is similar to \eqref{R3}, the difference being that there is a $V$ factor on the left.  
\\
\\
Now, by analogy with the previous Subsection 4.1, we introduce a regularization of ${\mathcal{T}}_{k_1,k_2}^{2'}  $ and then integrate by parts in time. Using again the fact that
\begin{align*}
\partial_t \partial_{\xi_j} \widehat{f}(t,\xi) &=\frac{2}{(2\pi)^3} t \xi_j e^{it \vert \xi \vert^2} \int_{\mathbb{R}^3} \widehat{V}(t,\xi - \eta_1) e^{-it \vert \eta_1 \vert ^2} \widehat{f}(t,\eta_1) d\eta_1 \\
\notag&-\frac{i}{(2\pi)^3} e^{it \vert \xi \vert^2} \int_{\mathbb{R}^3} \widehat{x_j V}(t,\xi - \eta_1) e^{-it \vert \eta_1 \vert ^2} \widehat{f}(s,\eta_1) d\eta_1 \\
\notag&+\frac{2}{(2\pi)^3}  \int_{\mathbb{R}^3} t \eta_{1,j} e^{it (\vert \xi \vert^2 - \vert \xi-\eta_1 \vert^2 - \vert \eta_1 \vert^2)} \widehat{f}(t,\eta_1) \widehat{f}(t,\xi- \eta_1) d\eta_1 \\
\notag&-\frac{i}{(2\pi)^3} \int_{\mathbb{R}^3} e^{it (\vert \xi \vert^2 - \vert \xi-\eta_1 \vert^2 - \vert \eta_1 \vert^2)} \widehat{f}(t,\eta_1) \partial_{\xi_j} \widehat{f}(t,\xi- \eta_1) d\eta_1 ,
\end{align*}
we obtain
\begin{align}
\notag {\mathcal{T}}_{k_1,k_2}^{2'}  &= \\
                            &\label{2R5} \frac{1}{(2\pi)^6} \int_{ \mathbb{R}^3} \frac{P_k (\xi) P_{k_1} (\eta_1) \widehat{V}(t,\xi-\eta_1)}{\vert \xi \vert^2 - \vert \eta_1 \vert^2}\\
                            &\notag \times \int_{ \mathbb{R}^3 } \frac{\xi_l \eta_{2,j}}{\vert \eta_2 \vert ^2} \frac{P_k (\xi) P_{k_2} (\eta_2) \widehat{V}(t,\eta_1-\eta_2)}{\vert \xi \vert^2 - \vert \eta_2 \vert^2 } e^{it(\vert \xi \vert^2 - \vert \eta_2 \vert^2)} \partial_{\eta_{2,j}} \widehat{f}(t,\eta_2) P_{k_2}(\eta_2)  d\eta_2 d\eta_1\\
                            &+\frac{i}{(2\pi)^9}\label{2R8} \int_1 ^t \int_{\mathbb{R}^3 } \frac{P_k (\xi) P_{k_1} (\eta_1) \widehat{V}(s,\xi-\eta_1)}{\vert \xi \vert^2 - \vert \eta_1 \vert^2}  \int_{\mathbb{R}^3 } \frac{\xi_l \eta_{2,j}}{\vert \eta_2 \vert ^2} \frac{P_k (\xi) P_{k_2} (\eta_2) \widehat{V}(s,\eta_1-\eta_2)}{\vert \xi \vert^2 - \vert \eta_2 \vert^2 }   \\
                            & \notag \times \int_{\mathbb{R}^3} \widehat{x_j V}(s,\eta_2 - \eta_3) e^{i s( \vert \xi \vert^2 - \vert \eta_3 \vert^2)} \widehat{f}(s,\eta_3) d\eta_3 d\eta_2 d\eta_1 ds 
                            \\
                             &+\frac{2i}{(2\pi)^9}\label{2R9} \int_1 ^t \int_{\mathbb{R}^3} \frac{P_k (\xi) P_{k_1} (\eta_1) \widehat{V}(s,\xi-\eta_1)}{\vert \xi \vert^2 - \vert \eta_1 \vert^2}\int_{ \mathbb{R}^3 } \frac{\xi_l \eta_{2,j}}{\vert \eta_2 \vert ^2} \frac{P_k (\xi) P_{k_1} (\eta_2) \widehat{V}(s,\eta_1-\eta_2)}{\vert \xi \vert^2 - \vert \eta_2 \vert^2} \\
                             & \notag \times \int_{\mathbb{R}^3} is \eta_{3,j} e^{is (\vert \xi \vert^2 - \vert \eta_2-\eta_3 \vert^2 - \vert \eta_3 \vert^2)} \widehat{f}(s,\eta_3) \widehat{f}(s,\eta_2- \eta_3) d\eta_3 d\eta_2 d\eta_1 ds 
                             \end{align}
                            \begin{align}
                             &+\frac{i}{(2\pi)^9}\label{2R10} \int_1 ^t \int_{\mathbb{R}^3 } \frac{P_k (\xi) P_{k_1} (\eta_1) \widehat{V}(s,\xi-\eta_1)}{\vert \xi \vert^2 - \vert \eta_1 \vert^2}\int_{ \mathbb{R}^3 } \frac{\xi_l \eta_{2,j}}{\vert \eta_2 \vert ^2} \frac{P_k (\xi) P_{k_2} (\eta_2) \widehat{V}(s,\eta_1-\eta_2)}{\vert \xi \vert^2 - \vert \eta_1 \vert^2} \\
                             & \notag \times \int_{\mathbb{R}^3} e^{is (\vert \xi \vert^2 - \vert \eta_2-\eta_3 \vert^2 - \vert \eta_3 \vert^2)} \widehat{f}(s,\eta_3) \partial_{\eta_{2,j}} \widehat{f}(s,\eta_2- \eta_3) d\eta_3 d\eta_2 d\eta_1 ds    \\  
                             &\notag +\frac{2i}{(2\pi)^9}\int_1 ^t \int_{\mathbb{R}^3 } \frac{P_k (\xi) P_{k_1} (\eta_1) \widehat{V}(s,\xi-\eta_1)}{\vert \xi \vert^2 - \vert \eta_1 \vert^2} \int_{ \mathbb{R}^3 } \frac{P_k (\xi) P_{k_2} (\eta_2) \widehat{V}(s,\eta_1-\eta_2)}{\vert \xi \vert^2 - \vert \eta_2 \vert^2 } \\
                            &\notag \underbrace{\times \int_{\mathbb{R}^3} i s \xi_l \widehat{V}(s,\eta_2 - \eta_3) e^{i s( \vert \xi \vert^2 - \vert \eta_3 \vert^2)} \widehat{f}(s,\eta_3) d\eta_3 d\eta_2 d\eta_1 ds}_{{\mathcal{T}}_{k_1,k_2}^3}\\
                            & \notag + \lbrace \textrm{similar terms 3} \rbrace .
\end{align}                            
Note that in ${\mathcal{T}}_{k_1,k_2}^3$ the factor $\frac{\xi_l \eta_{2,j}}{\vert \eta_2 \vert ^2}$ is absent due to the summation on $j.$ \\
Here we elected to write the localizations in $\xi$ and in $\eta_2$ several times. This will be useful when obtaining estimates. Indeed it allows us to keep track of the localization when the $V$ term and the profile are considered separately. \\
The better terms will be bounded similarly to \eqref{2R5}. Their explicit expressions are written in the appendix. 
\\
\\
\underline{Conclusion for the second expansion:} \\
We conclude from cases 1 and 2 above that
\begin{align*}
\mathcal{T}_{k_1}^2 &=  \sum_{\vert k-k_2 \vert > 1} \eqref{2R2} + \lbrace \text{similar terms 1} \rbrace \\
& + \sum_{\vert k-k_2 \vert \leqslant 1} \eqref{2R3} + \eqref{2R5} + \eqref{2R8} + \eqref{2R9} + \eqref{2R10} + \lbrace \text{similar terms 2 and 3} \rbrace \\
&  +\sum_{k_2 \in \mathbb{Z}} \mathcal{T}_{k_1,k_2}^3.
\end{align*}
\\
\\
At this point, by analogy with the previous subsection, we expect to be able to estimate all the terms  effectively except the analogs of $\mathcal{T}_{k_1}^2 $ (namely ${\mathcal{T}}_{k_1,k_2}^3 $). Since they are similar to the term we started with (namely ${\mathcal{T}}_{k_1}^1$) we use the same procedure on them: we start by splitting the $\eta_3$ frequency dyadically ($k_3$ denotes the corresponding exponent). \\ If $\vert k - k_3 \vert >1 $ then we integrate by parts in time in these terms. \\ If $\vert k-k_3 \vert \leqslant 1$ then we first integrate by parts in $\eta_2$ and then in time. 
\\
\\
Now if we iterate this process, at the $n-$th step of the iteration we obtain:
\begin{itemize}
\item Analogs of ${\mathcal{T}}_{k_1,k_2}^3$ that will be integrated by parts in time at the next step.
\item Analogs of the remainder terms (namely \eqref{2R2}, \eqref{2R3}, \eqref{2R5},\eqref{2R8}, \eqref{2R9} and \eqref{2R10}), but with extra $V$ factors on the left.
\end{itemize}
\bigskip
More precisely at the $n-$th step of the iteration the following remainder terms appear: \\
first the $n-$th iterate of \eqref{2R3}:
\begin{align*}
\mathcal{F} I_1 ^n f(t,\xi) &:= \displaystyle \int_1 ^t \int_{\big(\mathbb{R}^3 \big)^{n-1} } \prod_{\gamma=1}^{n-1} \frac{\widehat{V}(s, \eta_{\gamma-1}-\eta_{\gamma}) P_{k_{\gamma}}(\eta_{\gamma}) P_{k}(\xi)}{\vert \xi \vert^2 - \vert \eta_{\gamma} \vert^2} d\eta_{1} ... d\eta_{n-2} \\
& \times \int_{\mathbb{R}^3 }  \frac{\xi_l \eta_{n,j}}{\vert \eta _n \vert^2} \partial_{n,j} \widehat{V}(s, \eta_{n-1}-\eta_n) e^{is(\vert \xi \vert^2 - \vert \eta_n \vert ^2)} \widehat{f_{k_n}}(s,\eta_n) d\eta_n d\eta_{n-1} ds .
\end{align*}
The $n-$th iterate of \eqref{2R8}
\begin{align*}
\mathcal{F} I_2 ^n f(t,\xi) & := \int_1 ^t \int_{\big(\mathbb{R}^3 \big)^{n-1}} \prod_{l=1}^{n-1} \frac{\widehat{V}(s,\eta_{l-1}-\eta_l) P_{k_l}(\eta_l) P_{k}(\xi)}{\vert \xi \vert^2 - \vert \eta_l \vert^2} d\eta_{1} ... d\eta_{n-2} \frac{\xi_l \eta_{n-1,j}}{\vert \eta_{n-1,j} \vert^2} \\
& \times \int_{\mathbb{R}^3}  \widehat{x_j V}(s,\eta_{n-1}-\eta_n) e^{is(\vert \xi \vert^2- \vert \eta_n \vert^2)} \widehat{f}(s,\eta_n) d\eta_n d\eta_{n-1} ds .
\end{align*}
The $n-$th iterate of \eqref{2R2}:
\begin{align*}
\mathcal{F} I_3 ^n f(t,\xi) &:= \displaystyle \int_{\big(\mathbb{R}^3 \big)^{n-1}} \prod_{\gamma=1}^{n-1} \frac{\widehat{V}(t, \eta_{\gamma-1}-\eta_{\gamma}) P_{k_{\gamma}}(\eta_{\gamma}) P_{k}(\xi)}{\vert \xi \vert^2 - \vert \eta_{\gamma} \vert^2} d\eta_{1} ... d\eta_{n-2} \\
& \times \int_{\mathbb{R}^3}  \xi_l \frac{\widehat{V}(t,\eta_{n-1} - \eta_n)}{\vert \xi \vert^2 - \vert \eta_n \vert^2} e^{it(\vert \xi \vert^2 - \vert \eta_n \vert ^2)} \widehat{t f_{k_n}}(t,\eta_n) d\eta_n d\eta_{n-1} .
\end{align*}
The $n-$th iterate of \eqref{2R5}:
\begin{align*}
\mathcal{F} I_4 ^n f(t,\xi) &:= \displaystyle \int_{ \big(\mathbb{R}^3 \big)^{n-1} } \prod_{\gamma=1}^{n-1} \frac{\widehat{V}(t, \eta_{\gamma-1}-\eta_{\gamma}) P_{k_\gamma}(\eta_\gamma) P_{k}(\xi)}{\vert \xi \vert^2 - \vert \eta_{\gamma} \vert^2} d\eta_{1} ... d\eta_{n-2} \\
& \times  \int_{\mathbb{R}^3}  \frac{\xi_l \eta_{n,j}}{\vert \eta_n \vert^2} \frac{P_k (\xi) \widehat{V}(t,\eta_{n-1} - \eta_n)}{\vert \xi \vert^2 - \vert \eta_n \vert^2} e^{it(\vert \xi \vert^2 - \vert \eta_n \vert ^2)} \partial_{\eta_{n,j}} \widehat{f}(t,\eta_n)P_{k_n}(\eta_n)  d\eta_n d\eta_{n-1} .
\end{align*}
The $n-$th iterate of \eqref{2R10}:
\begin{align*}
\mathcal{F} I_5 ^n f(t,\xi) & :=  \int_1 ^t \int_{\big(\mathbb{R}^3 \big)^{n-1}} \prod_{l=1}^{n-1} \frac{\widehat{V}(s,\eta_{l-1}-\eta_l) P_{k_l}(\eta_l) P_{k}(\xi)}{\vert \xi \vert^2 - \vert \eta_l \vert^2} d\eta_{1} ... d\eta_{n-2} \frac{\xi_l \eta_{n-1,j}}{\vert \eta_{n-1} \vert^2} \\
& \times \int_{\mathbb{R}^3} \partial_{\eta_{n,j}}  \widehat{f}(s,\eta_{n-1}-\eta_n) e^{is(\vert \xi \vert^2- \vert \eta_n - \eta_{n-1} \vert^2 - \vert \eta_n \vert^2)} \widehat{f}(s,\eta_n) d\eta_n d\eta_{n-1} ds ,
\end{align*}
and finally the $n-$th iterate of \eqref{2R9}:
\begin{align*}
\mathcal{F} I_6 ^n f(t,\xi) & := \displaystyle \int_1 ^t \int_{\big(\mathbb{R}^3 \big)^{n-1}} \prod_{l=1}^{n-1} \frac{\widehat{V}(s,\eta_{l-1}-\eta_l) P_{k_l}(\eta_l) P_{k}(\xi)}{\vert \xi \vert^2 - \vert \eta_l \vert^2} d\eta_{1} ... d\eta_{n-2} \frac{\xi_l \eta_{n-1,j}}{\vert \eta_{n-1} \vert^2} \\
& \times \int_{\mathbb{R}^3} i s \eta_{n,j} \widehat{f}(s,\eta_{n-1}-\eta_n) e^{is(\vert \xi \vert^2 - \vert \eta_{n-1} - \eta_n \vert ^2 - \vert \eta_n \vert^2)} \widehat{f}(s,\eta_n) d\eta_n  d\eta_{n-1} ds .
\end{align*}
Note that the $n-$th iterates differ from the corresponding term of the first expansion by a factor of $V$ of the form
\begin{align*}
\displaystyle \int_{\big(\mathbb{R}^3 \big)^{n-2}} \prod_{l=1}^{n-1} \frac{\widehat{V}(s,\eta_{l-1} - \eta_{l}) P_{k_l}(\eta_l) P_{k}(\xi)}{\vert \xi \vert^2 - \vert \eta_l \vert^2} d\eta_{1} ... d\eta_{n-2}.
\end{align*}
We used the convention $\eta_0 = \xi.$\\
Note also that there is a slight abuse of notation here since when the $i-$th factor is such that $k_i - k < -1$ then the localization reads $P_{\leqslant k-2}(\eta_i)$ and not $P_{k_i}(\eta_i)$.  We also dropped the regularizations for better legibility, which we will always do when this $n-$th kernel appears.
\\
\\
Heuristically, we expect the analog of Proposition \ref{step1:estimates} to hold for these remainder terms, but with extra $\delta$ factors due to the additional $V$ factors. \\
More precisely we will prove the following proposition in Section 7:
\begin{proposition}[Estimating $n-$th iterates] \label{stepn:estimates}
Let $J(n-1) := \lbrace i \in \lbrace 1; ...; n-1 \rbrace , k \leqslant k_i -1 \rbrace. $ \\
There exists a constant $C_0$ independent of $n$ such that
\begin{align*}
\sum_{k_j \in J(n-1)} \Vert I_1 ^n (t) \Vert_{L^{\infty}_t L^2 _x} & \lesssim C_0^{n-2} \varepsilon_1 \delta^n.
\end{align*}
Similarly we have
\begin{align*}
\sum_{k_j \in J(n)} \Vert I_2 ^n (t) \Vert_{L^{\infty}_t L^2 _x} , \sum_{k_j \in J(n)} \Vert I_3 ^n (t) \Vert_{L^{\infty}_t L^2 _x} , \sum_{k_j \in J(n-1)} \Vert I_4 ^n (t) \Vert_{L^{\infty}_t L^2 _x}& \lesssim C_0^{n-2} \varepsilon_1 \delta^n, \\
\sum_{k_j \in J(n-1)} \Vert I_5 ^n (t) \Vert_{L^{\infty}_t L^2 _x} ,\sum_{k_j \in J(n-1)} \Vert I_6 ^n (t) \Vert_{L^{\infty}_t L^2 _x} & \lesssim C_0^{n-2} \varepsilon_1 ^2 \delta^{n-1}. \\
\end{align*}
All the implicit constants are independent of $n.$ 
\end{proposition}
We now see how this proposition implies the global existence and dispersive estimate from Theorem \ref{mainthm}:
\begin{proof}[Proof of \eqref{bootstrapconcl2}]
Now we can conclude that the stronger bound \eqref{bootstrapconcl2} holds. \\
From our discussion above, we know that at the $n-$th step of the iteration $O(n)$ terms appear, and that they can be estimated in one of six ways (like one $I_j ^n$ for $j=1,...6$).
\\
By Proposition \ref{stepn:estimates}, we can therefore conclude that they can be bounded in $L^{\infty}_t L^2 _x$ by a constant times $ C_0^{n-2} \delta^n \varepsilon_1.$ Hence if $C'$ is a large enough constant (independent of $n$) then
\begin{align*}
\Vert f \Vert_{X} \leqslant \varepsilon_0+ C' \varepsilon_1 \sum_{n=1} ^{\infty}  n  C_0 ^{n-2} \delta^n  \leqslant \frac{\varepsilon_1}{2}
\end{align*}
for $\delta$ small enough, which yields the desired result.
\end{proof}
\noindent \textbf{Organization of the rest of the paper:}\\
There remains to prove Proposition \ref{stepn:estimates} to show \eqref{bootstrapconcl2}, from which the global existence and dispersive estimate \eqref{mainestimate} follow. We first treat the base case (terms arising in the first expansion) and prove Proposition \ref{step1:estimates} in Section \ref{firstit}. Then we prove Proposition \ref{stepn:estimates}, where iterates of terms from the first expansion are treated. Note that Proposition \ref{step2:estimates} is a corollary of this proposition, since all the terms present there are themselves iterates of terms from Proposition \ref{step1:estimates}. More precisely, \eqref{R8} is of the type $I_2 ^n f(t), $ \eqref{R9} is of the type $I_6 ^n f(t)$ and \eqref{R10} is of the type $I_5 ^n f(t).$

Before starting the proof of Proposition \ref{stepn:estimates}, we must develop tools to handle the multilinearity in $V$ of the expressions that appear in the above expansion. Indeed we have to prove that the series representation obtained above converges. This is carried out in Section \ref{multilinanalysis}. Then we come to the proof of Proposition \ref{stepn:estimates} in Section \ref{mainproppf}. 

To close the proof of Theorem \ref{mainthm}, we show that the other bootstrap assumption \eqref{bootstrapconcl1} holds. The proof is carried out in the first subsection of the last Section \ref{end} of the paper. 

Finally we prove Theorem \ref{mainresultscattering} in Section \ref{end} (more precisely in the last subsection \ref{scattering}). We will see that the result essentially follows from estimates already proved in the above bootstrap argument.
\section{Estimating first iterates} \label{firstit}
In this section we prove Proposition \ref{step1:estimates}. This will also serve as a blueprint for the estimates of higher order of Proposition \ref{stepn:estimates}.
\\
We split the proof in three subsections: in the first one, we handle remainder terms that are not part of the iteration. In the second one, we treat the potential terms that are linear in the profile $f.$ These terms correspond to the first estimate with $\varepsilon_1 \delta$ in Proposition \ref{step1:estimates}. In the third subsection we treat bilinear terms. They correspond to the two estimates left in Proposition \ref{step1:estimates}. 
\subsection{Remainder terms}
In this first subsection we bound terms that are not present in the iteration but that arise when bounding the $X-$norm, namely \eqref{R1}, \eqref{R1bis} and \eqref{R1bisbis}. 

\begin{lemma}[Estimating \eqref{R1}] We have
\begin{eqnarray*}
\sum_{k_1 \in \mathbb{Z}}\Bigg \Vert P_k (\xi) \int_1 ^t \int_{\mathbb{R}^3} e^{is (\vert \xi \vert^2 - \vert \eta_1 \vert^2)} \widehat{f}_{k_1}(\eta_1) \partial_{\xi_l} \widehat{V}(s,\xi-\eta_1) d\eta_1 ds \Bigg \Vert_{L^{\infty}_t L^2 _x} \lesssim \varepsilon_1 \delta.
\end{eqnarray*}
\end{lemma}
\begin{proof}
We use Strichartz estimates to obtain
\begin{eqnarray*}
&& \Bigg \Vert P_k (\xi) \int_1 ^t \int_{\mathbb{R}^3} e^{is (\vert \xi \vert^2 - \vert \eta_1 \vert^2)} \widehat{f}_{k_1}(s,\eta_1) \partial_{\xi_l} \widehat{V}(s,\xi-\eta_1) d\eta_1 ds \Bigg \Vert_{L^{\infty}_t L^2 _x} \\
&=& \Bigg \Vert \mathcal{F}^{-1} P_k (\xi) \int_1 ^t \int_{\mathbb{R}^3} e^{is (\vert \xi \vert^2 - \vert \eta_1 \vert^2)} \widehat{f}_{k_1}(s,\eta_1) \partial_{\xi_l} \widehat{V}(s,\xi-\eta_1) d\eta_1 ds \Bigg \Vert_{L^{\infty}_t L^2 _x} \\
& \leqslant & \Bigg \Vert \int_1 ^t e^{-is \Delta} \big( e^{is \Delta} f_{k_1} (xV) \big) ds \Bigg \Vert_{L^{\infty}_t L^2 _x} \\
& \lesssim & \Vert e^{it \Delta} f_{k_1} (x_l V) \Vert_{L^{4/3}_t L^{3/2}_x} \\
& \lesssim & \Vert e^{it \Delta} f_{k_1} \Vert_{L^{4/3}_t L^6 _x} \Vert x_l V \Vert_{L^{\infty} _t L^2 _x} \\
& \lesssim & \Vert e^{it \Delta} f_{k_1} \Vert_{L^{4/3}_t L^{6} _x} \Vert V \Vert_{L^{\infty}_t B_x} \\
& \lesssim & \min \lbrace 1.1^{-5k_1/6} , 1.1^{k_1 /8} \rbrace \varepsilon_1 \delta,
\end{eqnarray*}
where we used Lemma \ref{disp2} for the last line. \\
The desired result follows.
\end{proof}
Now we move on to the more involved terms \eqref{R1bis} and \eqref{R1bisbis}:
\begin{lemma}[Estimating \eqref{R1bis}, \eqref{R1bisbis}]
We have 
\begin{align*}
&\sum_{k_1 \in \mathbb{Z}} \Bigg \Vert  1.1^{-k} \phi'(1.1^{-k} \xi) \frac{\xi_l}{\vert \xi \vert} \int_{1}^t \int_{\mathbb{R}^3}  e^{is (\vert \xi \vert^2 -\vert \eta_1 \vert^2)} \widehat{f_{k_1}}(s,\eta_1) \widehat{V}(s,\xi-\eta_1) d\eta_1 ds \Bigg \Vert_{L^2}, \\
&\sum_{k_1 \in \mathbb{Z}} \Bigg \Vert  1.1^{-k} \phi'(1.1^{-k} \xi) \frac{\xi_l}{\vert \xi \vert} \int_{1}^t \int_{\mathbb{R}^3}  e^{is (\vert \xi \vert^2 -\vert \eta_1 \vert^2)} \widehat{f_{k_1}}(s,\eta_1) \widehat{f}(s,\xi-\eta_1) d\eta_1 ds \Bigg \Vert_{L^2}
\\
& \lesssim \varepsilon_1 \delta.
\end{align*}
\end{lemma}
\begin{remark}
Notice that the factor in front allows us to sum over $k_1 \in \mathbb{Z}$ and therefore obtain the corresponding estimate in Proposition \ref{step1:estimates}.
\end{remark}
\begin{proof} We split the proof into several cases: \\
\underline{Case 1: $k>0$} \\
In this case we use Strichartz estimates to write that
\begin{align*}
& \Bigg \Vert  1.1^{-k} \phi'(1.1^{-k} \xi) \frac{\xi_l}{\vert \xi \vert} \int_{1}^t \int_{\mathbb{R}^3}  e^{is (\vert \xi \vert^2 -\vert \eta_1 \vert^2)} \widehat{f_{k_1}}(s,\eta_1) \widehat{V}(s,\xi-\eta_1) d\eta_1 ds \Bigg \Vert_{L^{\infty}_t L^2 _x} \\
& \lesssim \Bigg \Vert \mathcal{F}^{-1} \int_1 ^t \int_{\mathbb{R}^3} e^{is (\vert \xi \vert^2 - \vert \eta_1 \vert^2)} \widehat{f}_{k_1}(s,\eta_1)  \widehat{V}(s,\xi-\eta_1) d\eta_1 ds \Bigg \Vert_{L^{\infty}_t L^2 _x} \\
& \lesssim \Vert \big( e^{it \Delta} f_{k_1} \big)V(t) \Vert_{L^{4/3}_t L^{3/2}_x} \\
& \lesssim \Vert e^{it \Delta} f_{k_1} \Vert_{L^{4/3}_t L^{6}_x} \Vert V \Vert_{L^{\infty}_t L^{2}_x} \\
& \lesssim \Vert e^{it \Delta} f_{k_1} \Vert_{L^{4/3}_t L^{6}_x} \delta ,
\end{align*}
and the result follows using Lemma \ref{disp2}.
\\
\\
\underline{Case 2: $k \leqslant 0$}
\\
We distinguish three subcases:\\
\underline{Case 2.1: $k>k_1+1$} \\
Then from the proof of Lemma \ref{symbolbis} we see that the frequency $\vert \xi- \eta_1 \vert$ is localized at $1.1^{k},$ therefore we can write, using the same reasoning as in the previous case, that 
\begin{align*}
& \Bigg \Vert  1.1^{-k} \phi'(1.1^{-k} \xi) \frac{\xi_l}{\vert \xi \vert} \int_{1}^t \int_{\mathbb{R}^3}  e^{is (\vert \xi \vert^2 -\vert \eta_1 \vert^2)} \widehat{f_{k_1}}(s,\eta_1) \widehat{V}(s,\xi-\eta_1) d\eta_1 ds \Bigg \Vert_{L^{\infty}_t L^2 _x} \\
& = \Bigg \Vert  1.1^{-k} \phi'(1.1^{-k} \xi) \frac{\xi_l}{\vert \xi \vert} \int_{1}^t \int_{\mathbb{R}^3}  e^{is (\vert \xi \vert^2 -\vert \eta_1 \vert^2)} \widehat{f_{k_1}}(s,\eta_1) \widehat{V_{k}}(s,\xi-\eta_1) d\eta_1 ds \Bigg \Vert_{L^{\infty}_t L^2 _x} \\
& \lesssim 1.1^{-k} \Vert V_{k} \Vert_{L^{\infty}_t L^2 _x}  \min  \lbrace 1.1^{-5 k_1 / 6} ; 1.1^{k_1 /8} \rbrace \varepsilon_1 ,
\end{align*}
and now we use Bernstein's inequality to write that
\begin{align*}
1.1^{-k} \Vert V_{k} \Vert_{L^{\infty}_t L^2 _x} \lesssim  \Vert V_{k} \Vert_{L^{\infty}_t L^{6/5} _x} \lesssim \delta.
\end{align*}
\underline{Case 2.2: $\vert k-k_1 \vert \leqslant 1$}\\
Then we split the $\xi-\eta_1$ frequency dyadically and denote $k_2$ the corresponding exponent. \\
Note that $\vert \xi-\eta_1 \vert \leqslant \vert \xi \vert + \vert \eta_1 \vert \leqslant 1.1^{k+1} + 1.1^{k+2} \leqslant 1.1^{k+10}.$ \\
 As a result $k_2 \leqslant k +10.$ \\
Now we can write, using Strichartz estimates, Bernstein's inequality and Lemma \ref{disp2}:
\begin{align*}
 & \Bigg \Vert  1.1^{-k} \phi'(1.1^{-k} \xi) \frac{\xi_l}{\vert \xi \vert} \int_{1}^t \int_{\mathbb{R}^3}  e^{is (\vert \xi \vert^2 -\vert \eta_1 \vert^2)} \widehat{f_{k_1}}(s,\eta_1) \widehat{V_{k_2}}(s,\xi-\eta_1) d\eta_1 ds \Bigg \Vert_{L^{\infty}_t L^2 _x} \\
 & \lesssim 1.1^{-k} \Vert V_{k_2} \Vert_{L^{\infty}_t L^{2}_x}  \min  \lbrace 1.1^{-5 k_1 / 6} ; 1.1^{k_1 /8} \rbrace \varepsilon_1 \\
 & \lesssim 1.1^{k_2-k} 1.1^{k_2 /2} \Vert V \Vert_{L^{\infty}_t L^1 _x} \min  \lbrace 1.1^{-5 k_1 / 6} ; 1.1^{k_1 /8} \rbrace \varepsilon_1 \\
 & \lesssim 1.1^{k_2 /2} \delta \min  \lbrace 1.1^{-5 k_1 / 6} ; 1.1^{k_1 /8} \rbrace \varepsilon_1.
\end{align*}
Now since $k \leqslant 0, k_2 \leqslant 10$ and the factor in front allows us to sum over $k_2.$ The result follows.
\\
\underline{Case 2.3: $k_1 > k+1$}\\
In this case we split the time variable dyadically. Let's denote $m$ the corresponding exponent. \\
We must estimate
\begin{align*}
I_{m,k_1,k} :=  1.1^{-k} \phi'(1.1^{-k} \xi) \frac{\xi_l}{\vert \xi \vert} \int_{1.1^m}^{1.1^{m+1}} \int_{\mathbb{R}^3}  e^{is (\vert \xi \vert^2 -\vert \eta_1 \vert^2)} \widehat{f_{k_1}}(s,\eta_1) \widehat{V_{k_1}}(s,\xi-\eta_1) d\eta_1 ds ,
\end{align*}
where the extra localization can be placed on $V$ since $\xi - \eta_1$ has magnitude roughly $1.1^{k_1}$ given $k_1>k+1.$ \\
\\
\underline{Subcase 2.3.1: $k \leqslant -\frac{m}{100}$.} \\
Then we write, using Bernstein's inequality, H\"{o}lder's inequality and Lemma \ref{dispersive}, that
\begin{align*}
\Vert I_{m,k_1,k} \Vert_{L^{\infty}_t L^2 _x} & \lesssim 1.1^{-k} 1.1^{m} \sup_{t \simeq 1.1^m} \Vert \big( e^{it \Delta} f_{k_1} V_{k_1}(t) \big)_k \Vert_{L^2 _x} \\
& \lesssim 1.1^{m} 1.1^{k/2} \sup_{t \simeq 1.1^m} \Vert e^{it \Delta} f_{k_1} V_{k_1}(t)  \Vert_{L^{1}_x} \\
                                              & \lesssim 1.1^{m} 1.1^{k/2} \sup_{t \simeq 1.1^m} \Vert e^{it \Delta} f_{k_1} \Vert_{L^6 _x} \Vert V_{k_1}(t)  \Vert_{L^{\infty}_t L^{6/5}_x} \\
& \lesssim 1.1^m 1.1^{-m} 1.1^{k/2} \varepsilon_1 \Vert V_{k_1} \Vert_{L^{\infty}_t L^{6/5}_x} \\
& \lesssim 1.1^{k/2} \varepsilon_1 \min \lbrace 1.1^{-10 k_1}; 1.1^{k_1 /2} \rbrace \delta \\
& \lesssim 1.1^{-m/200} \varepsilon_1 \min \lbrace 1.1^{-10 k_1}; 1.1^{k_1 /2} \rbrace \delta .
\end{align*}
For the second to last line we used the definition of the $B_x'$ norm for the high frequencies and Bernstein's inequality for the low frequencies. \\
We can sum over $k_1$ and $m$ given the factors that appear.
\\
\\
\underline{Subcase 2.3.2: $-\frac{m}{100} < k \leqslant 0$} \\
In this case we use Strichartz estimates as well as Lemma \ref{dispersive}:
\begin{align*}
\Vert I_{m,k_1,k} \Vert_{L^{\infty}_t L^2 _x} & \lesssim 1.1^{\frac{m}{100}} \Bigg \Vert \int_{1.1^m}  ^{1.1^{m+1}} e^{-is \Delta} \bigg( \big( e^{is \Delta} f_{k_1} \big) V_{k_1}(s) \bigg) ds \Bigg \Vert_{L^2 _x} \\
& \lesssim 1.1^{\frac{m}{100}} \Vert \textbf{1}_{t \simeq 1.1^m} \big( e^{it \Delta} f_{k_1} \big) V_{k_1}(t) \Vert_{L^{4/3}_t L^{3/2}_x} \\
& \lesssim 1.1^{\frac{m}{100}} \Vert \textbf{1}_{t \simeq 1.1^m} \big( e^{it \Delta} f_{k_1} \big) \Vert_{L^{4/3}_t L^6 _x} \Vert V_{k_1} \Vert_{L^{\infty}_t L^{2}_x} \\
& \lesssim 1.1^{\frac{m}{100}} 1.1^{-m/4} \varepsilon_1 \min \lbrace 1.1^{-10k_1} ; 1.1^{3k_1/2} \rbrace \delta .
\end{align*}
Now notice that we can sum over $m$ and $k_1$ because of the factors in front.  \\
The term \eqref{R1bisbis} can be bounded similarly, therefore we omit the details.
\end{proof}

\subsection{Potential terms}
Now we consider the terms that are linear in the profile. We first treat the terms with nonsingular denominator, and then the terms with a singular multiplier. 
\subsubsection{Easier iterate terms} 
Now we prove the estimates of Proposition \ref{step1:estimates} for the easier terms \eqref{R2}, \eqref{R2prime}, \eqref{R2bis}, \eqref{R6prime}, \eqref{R3}, \eqref{R4}, \eqref{R4bis}.
\\
\\
Let's start with \eqref{R6prime}:
\begin{lemma} Assume that $\vert k-k_1 \vert >1.$ Then 
\begin{align*}
\Bigg \Vert P_k(\xi) \int_1 ^t \int_{\mathbb{R}^3} \frac{i \xi_l \widehat{V}(s,\xi-\eta_1)}{\vert \xi \vert^2 - \vert \eta_1 \vert^2} e^{is(\vert \xi \vert^2 - \vert \eta_1 \vert^2)} \widehat{f_{k_1}}(s,\eta_1) d\eta_1 ds \Bigg \Vert_{L^{\infty}_t L^2 _x} \lesssim \min \lbrace 1.1^{k-k_1} ; 1.1^{(k_1-k)/8} \rbrace  \delta \varepsilon_1.
\end{align*}
\end{lemma}
\begin{proof}
Assume first that $k_1>k+1.$ \\
We use Strichartz estimates, Lemma \ref{bilin} and Bernstein's inequality to write that
\begin{align*}
& \Bigg \Vert P_k(\xi) \int_1 ^t \int_{\mathbb{R}^3} \frac{i \xi_l \widehat{V}(s,\xi-\eta_1)}{\vert \xi \vert^2 - \vert \eta_1 \vert^2} e^{is(\vert \xi \vert^2 - \vert \eta_1 \vert^2)} \widehat{f_{k_1}}(s,\eta_1) d\eta_1 ds \Bigg \Vert_{L^{\infty}_t L^2 _x} \\
& \lesssim \Bigg \Vert \int_1 ^t e^{-is\Delta} \Bigg( \mathcal{F}^{-1} \int_{\mathbb{R}^3} \frac{i \xi_l \widehat{V}(s,\xi-\eta_1)}{\vert \xi \vert^2 - \vert \eta_1 \vert^2} P_k(\xi) e^{-is\vert \eta_1 \vert^2} \widehat{f_{k_1}}(s,\eta_1) d\eta_1 \Bigg) ds \Bigg \Vert_{L^{\infty}_t L^2 _x} \\
& \lesssim \Bigg \Vert \mathcal{F}^{-1}  \int_{\mathbb{R}^3} \frac{\xi_l \widehat{V_{\leqslant k_1 +10}}(t,\xi-\eta_1)}{\vert \xi \vert^2 - \vert \eta_1 \vert^2} e^{-it \vert \eta_1 \vert^2} \widehat{f_{k_1}}(t,\eta_1) d\eta_1 \Bigg \Vert_{L^{4/3}_t L^{3/2} _x} 
\\
& \lesssim 1.1^{k-2 k_1} \Vert e^{it\Delta} f_{k_1} \Vert_{L^{4/3}_t L^6 _x} \Vert V_{\leqslant k_1+10} \Vert_{L^{\infty}_t L^2 _x} \\
& \lesssim 1.1^{k-k_1} \Vert e^{it\Delta} f_{k_1} \Vert_{L^{4/3}_t L^6 _x} \Vert V \Vert_{L^{\infty}_t B_x '} \\
& \lesssim  1.1^{k-k_1} \delta \varepsilon_1 ,
\end{align*}
where we used the dispersive estimate from Lemma \ref{dispersive} for the last line. \\
In the other case (if $k>k_1 +1$) only the last lines are different:
\begin{align*}
& \Bigg \Vert P_k(\xi) \int_1 ^t \int_{\mathbb{R}^3} \frac{i \xi_l \widehat{V}(s,\xi-\eta_1)}{\vert \xi \vert^2 - \vert \eta_1 \vert^2} e^{is(\vert \xi \vert^2 - \vert \eta_1 \vert^2)} \widehat{f_{k_1}}(s,\eta_1) d\eta_1 ds \Bigg \Vert_{L^{\infty}_t L^2 _x} \\
& \lesssim 1.1^{-k} \Vert V_{\leqslant k+10} \Vert_{L^{\infty}_t L^2 _x} \Vert e^{it \Delta} f_{k_1} \Vert_{L^{4/3}_t L^6 _x} \\
& \lesssim 1.1^{(k_1-k) /8} \varepsilon_1 \delta,
\end{align*}
where we used Lemma \ref{disp2}.
\end{proof}
\begin{remark}
Note that the factors in $k,k_1$ allow us to sum over $k_1$ such that $\vert k_1 - k \vert >1$ which proves the corresponding estimate in Proposition \ref{step1:estimates}.
\end{remark}
Now we estimate \eqref{R3}, \eqref{R4} and \eqref{R4bis}. They appeared in Section 4 when $\vert k-k_1 \vert \leqslant 1.$
\begin{lemma}[Estimating \eqref{R3}, \eqref{R4}, \eqref{R4bis}] Assume that $\vert k-k_1 \vert \leqslant 1.$ The following estimates hold:
\begin{eqnarray*}
\Bigg \Vert P_k (\xi) \int_{1} ^{t} \int_{ \mathbb{R}^3 } \frac{\xi_l \eta_{1,j}}{\vert \eta_1 \vert ^2} \partial_{\eta_{1,j}} \widehat{V}(s,\xi-\eta_1) e^{is( \vert \xi \vert^2 - \vert \eta_1 \vert^2)} \widehat{f_{k_1}}(s,\eta_1) d\eta_1 ds \Bigg \Vert_{L^{\infty}_t L^2 _x} , \\
\Bigg \Vert P_k (\xi) \int_1 ^t \int_{ \mathbb{R}^3 } \partial_{\eta_{1,j}} \bigg( \frac{\xi_l \eta_{1,j}}{\vert \eta_1 \vert ^2} \bigg) \widehat{V}(s,\xi-\eta_1) e^{is( \vert \xi \vert^2 - \vert \eta_1 \vert ^2)} \widehat{f_{k_1}}(s,\eta_1) d\eta_1 ds \Bigg \Vert_{L^{\infty}_t L^2 _x} \lesssim \varepsilon_1 \delta, \\
\Bigg \Vert P_k (\xi) \int_1 ^t \int_{\mathbb{R}^3 } 1.1^{-k_1} \frac{\xi_l \eta_{1,j}}{\vert \eta_1 \vert ^2} \widehat{V}(s,\xi-\eta_1) e^{is( \vert \xi \vert^2 - \vert \eta_1 \vert ^2)} \widehat{f_{k_1}}(s,\eta_1) \phi'(1.1^{-k_1} \vert \eta_1 \vert) \frac{\eta_{1,j}}{\vert \eta_1 \vert} d\eta_1 ds \Bigg \Vert_{L^{\infty}_t L^2 _x} \lesssim \varepsilon_1 \delta.
\end{eqnarray*}
\end{lemma}
\begin{proof}
We use Strichartz estimates, as well as the bilinear estimate from Lemma \ref{bilin} to obtain:
\begin{align*}
&\Bigg \Vert P_k (\xi) \int_{1} ^{t} \int_{ \mathbb{R}^3} \frac{\xi_l \eta_{1,j}}{\vert \eta_1 \vert ^2} \partial_{\eta_{1,j}} \widehat{V}(s,\xi-\eta_1) e^{is( \vert \xi \vert^2 - \vert \eta_1 \vert^2)} \widehat{f_{k_1}}(s,\eta_1) d\eta_1 ds \Bigg \Vert_{L^{\infty}_t L^2 _x} \\ 
& \lesssim  \Bigg \Vert \mathcal{F}^{-1} \int_{\mathbb{R}^3} \frac{\xi_l \eta_{1,j}}{\vert \eta_1 \vert^2} \partial_{\eta_{1,j}} \widehat{V}(s,\xi-\eta_1) e^{-it \vert \eta_1 \vert^2} \widehat{f_{k_1}}(t,\eta_1) ~d\eta_1 \Bigg \Vert_{L^{4/3}_t L^{3/2}_x} \\
& \lesssim  \Vert e^{it\Delta} f_{k_1} \Vert_{L^{4/3}_t L^{6}_x} \Vert x V \Vert_{L^{\infty} _t L^2 _x} \\
& \lesssim \varepsilon_1 \delta .
\end{align*}
For the second inequality, we write similarly that
\begin{align*}
&\Bigg \Vert P_k (\xi) \int_{1} ^{t} \int_{ \mathbb{R}^3} \partial_{\eta_{1,j}} \bigg( \frac{\xi_l \eta_{1,j}}{\vert \eta_1 \vert ^2} \bigg) \widehat{V_{\leqslant k+10}}(s,\xi-\eta_1) e^{is( \vert \xi \vert^2 - \vert \eta_1 \vert^2)} \widehat{f_{k_1}}(s,\eta_1) d\eta_1 ds \Bigg \Vert_{L^{\infty}_t L^2 _x} \\ 
& \lesssim  \Bigg \Vert \mathcal{F}^{-1}_{\xi} \int_{\mathbb{R}^3} \partial_{\eta_{1,j}} \bigg( \frac{\xi_l \eta_{1,j}}{\vert \eta_1 \vert ^2} \bigg) \widehat{V_{\leqslant k+10}}(s,\xi-\eta_1) e^{-it \vert \eta_1 \vert^2} \widehat{f_{k_1}}(t,\eta_1) ~d\eta_1 \Bigg \Vert_{L^{4/3}_t L^{3/2}_x} \\
& \lesssim \Vert e^{it\Delta} f_{k_1} \Vert_{L^{4/3}_t L^{6}_x} 1.1^{-k} \Vert \widehat{V_{\leqslant k+10}} \Vert_{L^{\infty}_t L^2 _x} .
\end{align*}
Now since $\vert \xi \vert \leqslant 1.1^{k+10} \Rightarrow 1.1^{-k} \leqslant \frac{2}{\vert \xi \vert} $ 
\\
As a result we can write, using this fact and Hardy's inequality: 
\begin{align*}
1.1^{-k} \Vert \widehat{V}_{\leqslant k+10} \Vert_{L^{\infty} _t L^2 _x} & \lesssim \bigg \Vert \frac{\widehat{V}_{\leqslant k+10}}{\vert \xi \vert} \bigg \Vert_{L^{\infty} _t L^2 _x} \\
& \lesssim \delta
\end{align*}
and therefore 
\begin{align*}
\Vert \eqref{R4} \Vert_{L^{\infty}_t L^2_x} \lesssim \varepsilon_1 \delta .
\end{align*}
The estimate for \eqref{R4bis} is the exact same as the one for \eqref{R4}, therefore we skip it.
\end{proof}
Now we move on to \eqref{R2}. These terms are present when $\vert k-k_1 \vert >1.$
\begin{lemma}[Estimating \eqref{R2}, \eqref{R2prime}] \label{Estimate:R2}
Assume that $ k_1-k >1.$ We have 
\begin{eqnarray*}
\Bigg \Vert P_k (\xi) \int_{\mathbb{R}^3} \frac{it \xi_l \widehat{V}(t,\xi-\eta)}{\vert \xi \vert^2 - \vert \eta \vert^2} e^{it (\vert \xi \vert^2 - \vert \eta \vert^2)} \widehat{f_{k_1}} (t,\eta) d\eta \Bigg \Vert_{L^{\infty}_t L^2 _x} \lesssim 1.1^{k-k_1} \varepsilon_1 \delta .
\end{eqnarray*}
\end{lemma}
\begin{proof} 
First notice that given that $k_1>k+1,$ we have
\begin{align*}
& P_k(\xi) \int_{\mathbb{R}^3} \frac{t \xi_l \widehat{V}(t,\xi-\eta_1)}{\vert \xi \vert^2 - \vert \eta_1 \vert^2} e^{it(\vert \xi \vert^2 - \vert \eta_{1} \vert^2)} \widehat{f_{k_1}}(t,\eta_1) d\eta_1 \\
 &= e^{it\vert \xi \vert^2} \int_{\mathbb{R}^3} \frac{t \xi_l P_k(\xi) \widehat{V_{\leqslant k_1+10}}(t,\xi-\eta_1)}{\vert \xi \vert^2 - \vert \eta_1 \vert^2} e^{-it \vert \eta_{1} \vert^2} \widehat{f_{k_1}}(t,\eta_1) d\eta_1 .
\end{align*}
Therefore we can use Lemma \ref{bilin}, Lemma \ref{symbolbis}, Bernstein's inequality as well as Lemma \ref{dispersive} to write that
\begin{align*}
\Vert \eqref{R2} \Vert_{L^{\infty}_t L^2_x} & \lesssim 1.1^{k-2k_1} \Vert V_{\leqslant k_1+10} \Vert_{L^{\infty}_t L^{3}_x} \Vert t e^{it\Delta} f_{k_1} \Vert_{L^{\infty}_t L^6_x} \\ 
                                            & \lesssim 1.1^{k-k_1}  \Vert V_{\leqslant k_1+10} \Vert_{L^{\infty}_t L^{3/2}_x} \varepsilon_1 \\
                                            & \lesssim 1.1^{k-k_1} \delta \varepsilon_1                 
\end{align*}
which yields the desired result.
\end{proof}
The bound above allows us to sum over $k_1>k+1.$
\\
\\
Similarly, we have the bound
\begin{lemma}[Estimating \eqref{R2}, \eqref{R2prime}]
Assume that $k>k_1+1.$ Then we have the bound
\begin{eqnarray*}
\Bigg \Vert P_k (\xi) \int_{\mathbb{R}^3} \frac{it \xi_l \widehat{V}(t,\xi-\eta_1)}{\vert \xi \vert^2 - \vert \eta_1 \vert^2} e^{it (\vert \xi \vert^2 - \vert \eta_1 \vert^2)} \widehat{f_{k_1}} (t,\eta_1) d\eta_1 \Bigg \Vert_{L^{\infty}_t L^2 _x} \lesssim 1.1^{(k_1-k)/2} \delta \varepsilon_1 . 
\end{eqnarray*}
\end{lemma}
\begin{proof}
Using Lemma \ref{bilin}, Bernstein's inequality and Lemma \ref{symbolbis} we can write:
\begin{align*}
&\Bigg \Vert P_k (\xi) \int_{\mathbb{R}^3} \frac{it \xi_l \widehat{V}(t,\xi-\eta_1)}{\vert \xi \vert^2 - \vert \eta_1 \vert^2} e^{it (\vert \xi \vert^2 - \vert \eta_1 \vert^2)} \widehat{f_{k_1}} (t,\eta_1) d\eta_1 \Bigg \Vert_{L^{\infty}_t L^2 _x} \\
& \lesssim \Vert t e^{it\Delta} f_{k_1}(t, \cdot) \Vert_{L^{\infty}_t L^\infty _x} \Vert V_k \Vert_{L^{\infty}_t L^2 _x} 1.1^k 1.1^{-2k} \\
& \lesssim 1.1^{k_1 /2} \Vert  t e^{it\Delta} f_{k_1} (t, \cdot) \Vert_{L^{\infty}_t L^6 _x} 1.1^{k/2} \Vert V \Vert_{L^{\infty} _t B_x} \\
& \lesssim 1.1^{(k_1-k)/2} \varepsilon_1 \delta .
\end{align*}
\end{proof}
We can sum for $k>k_1+1$ because of the factor in front. \\
We can also estimate \eqref{R2bis} in a very similar way: 
\begin{lemma}
Assume that $\vert k-k_1 \vert>1.$ Then
\begin{align*}
& \Bigg \Vert P_k (\xi) \int_1 ^t \int_{\mathbb{R}^3}  \frac{is \xi_l \partial_s \widehat{V}(s,\xi-\eta_1)}{\vert \xi \vert^2 - \vert \eta_1 \vert^2} e^{is (\vert \xi \vert^2 - \vert \eta_1 \vert^2)} \widehat{f_{k_1}} (s,\eta_1) d\eta_1 ds \Bigg \Vert_{L^{\infty}_t L^2 _x} \\
& \lesssim \min \lbrace 1.1^{(k_1-k)/2} , 1.1^{k-k_1} \rbrace \varepsilon_1 \Vert \partial_t V \Vert_{L^1 _t B_x} \\
& \lesssim \min \lbrace 1.1^{(k_1-k)/2} , 1.1^{k-k_1} \rbrace \varepsilon_1 \delta .
\end{align*}
\end{lemma}
\bigskip
\subsubsection{Estimating \eqref{R5}, \eqref{R5prime}, \eqref{R5bis}}
The proof used here is essentially taken from the work of M.Beceanu and W.Schlag (\cite{BSmain}). However since their proof will need to be adapted later on in the paper, we recall the details. 
\\
\\
Let's start with the following useful result:
\begin{lemma}[See lemma 5.3 and equality (5.11), \cite{BSmain}] \label{BSmain1}
Let $T_3 (x_0,x,y)$ be defined as 
\begin{align*}
\big(\mathcal{F}^{-1}_{x_0} \mathcal{F}_{x_1,y} T_3 \big)(\xi_0, \xi_1,\eta) := \frac{\widehat{V}(\xi_1 - \xi_0)}{\vert \xi_1 + \eta \vert^2 - \vert \eta \vert ^2} .
\end{align*}
Then we have
\begin{align*}
\int_{\mathbb{R}^3} T_3 (x_0,x,y) dx_{0} &=\mathcal{F}^{-1} _{\xi,\eta} \bigg( \frac{\widehat{V}(\xi)}{\vert \xi + \eta \vert^2 - \vert \eta \vert^2 } \bigg)(x,y) \\
& = C \vert y \vert^{-2} \int_0 ^{\infty} e^{-i s \hat{y} \cdot (x-y/2)} \widehat{V}(-s \hat{z}) s ds \\
& := C \vert y \vert^{-2} L(\vert y \vert - 2 x \cdot \hat{y},\hat{y})
\end{align*}
with $C$ a constant and $\hat{z} = \frac{z}{\vert z \vert}.$ 
\end{lemma}
We will sometimes write $L_{V}$ to emphasize the dependence on the potential. \\ 
\\
The following lemma provides estimates on $L:$
\begin{lemma}[Proposition 6.1 of \cite{BSmain}] \label{BSmain2}
\begin{align*}
\bigg(\int_{\mathbb{S}^2} \int_0 ^{\infty} \vert L(r,\omega) \vert ^2 dr d\omega \bigg)^{1/2} &  \lesssim \Vert V \Vert_{L^2} \\
\int_{\mathbb{S}^2} \int_0 ^{\infty} \vert L(r,\omega) \vert dr d\omega \lesssim \sum_{k \in \mathbb{Z}} 1.1^{k/2} \bigg( \int_{\mathbb{S}^2} \int_{r \sim 1.1^k} \vert L(r,\omega) \vert^2 dr d\omega \bigg)^{1/2} & \lesssim \Vert V \Vert_{\dot{B}^{1/2}} 
\end{align*}
where we denoted
\begin{align*}
\Vert V \Vert_{\dot{B}^{1/2}} = \sum_{k=-\infty}^{\infty} 1.1^{k/2} \Vert V(x) \textbf{1}_{1.1^k \leqslant \vert x \vert \leqslant 1.1^{k+1}} \Vert_{L^2 _x} .
\end{align*}
\end{lemma}
Notice that this norm is weaker than our controlling norm $B_x,$ which can be seen by applying the Cauchy-Schwarz inequality in the definition of $\dot{B}^{1/2}.$ 
\\
\\
With this result we are ready to bound \eqref{R5}, \eqref{R5prime}:
\begin{lemma}[Estimating \eqref{R5},\eqref{R5prime}]
We have, when $\vert k-k_1 \vert \leqslant 1,$
\begin{eqnarray*}
\Bigg \Vert P_k(\xi) \int_{ \mathbb{R}^3} \frac{\xi_l \eta_{1,j}}{\vert \eta_1 \vert ^2} \frac{\widehat{V}(t,\xi-\eta_1)}{\vert \xi \vert^2 - \vert \eta_1 \vert^2} e^{it( \vert \xi \vert^2 - \vert \eta_1 \vert^2)} \partial_{\eta_{1,j}} \widehat{f}(t,\eta_1)P_{k_1}(\eta_1) d\eta_1 \Bigg \Vert_{L^{\infty}_t L^2 _x}&  \lesssim & \varepsilon \delta .
\end{eqnarray*}
\end{lemma} 
\begin{proof}
We prove the estimate by duality. Let $g \in L^2 _x.$
\begin{eqnarray*}
&&\langle P_k (\xi)\int_{\mathbb{R}^3} \frac{\xi_l \eta_{1,j}}{\vert \eta_1 \vert ^2} \frac{\widehat{V}(t,\xi-\eta_1)}{\vert \xi \vert^2 - \vert \eta_1 \vert^2} e^{it (\vert \xi \vert^2 - \vert \eta_1 \vert^2)} \partial_{\eta_{1,j}} \widehat{f} (t,\eta_1)P_{k_1}(\eta_1) d\eta_1 , \hat{g}(\xi) \rangle \\
&=& \int_{\mathbb{R}^6} P_k (\xi) \frac{\xi_l \eta_{1,j}}{\vert \eta_1 \vert ^2} \frac{\widehat{V}(t,\xi-\eta_1)}{\vert \xi \vert^2 - \vert \eta_1 \vert^2} e^{it (\vert \xi \vert^2 - \vert \eta_1 \vert^2)} \partial_{\eta_{1,j}} \widehat{f} (t,\eta_1) P_{k_1}(\eta_1) d\eta_1 \overline{\hat{g}}(\xi) d\xi  \\
&=&- \int_{\mathbb{R}^6} \frac{\widehat{V'}(t,\xi_1)}{\vert \eta + \xi_1 \vert^2 - \vert \eta \vert ^2} e^{-it \vert \eta +\xi_1 \vert^2} \frac{(\xi_1+\eta)_j}{\vert \xi_1 + \eta \vert^2} \\
& \times & \partial_{\xi_{1,j}} \widehat{f}(t,\xi_1 +\eta)P_{k_1}(\xi_1+\eta) \eta_j P_k(\eta) e^{it \vert \eta \vert^2} \bar{\widehat{g}}_k (\eta)d\eta d\xi_1 ,
\end{eqnarray*}
where to go from the second to the third line we use the change of variables $(\xi,\eta_1) \to (\eta,\xi_1+\eta)$ and denoted $\widehat{V'}(t,\xi))=\widehat{V}(t,-\xi).$ \\
Now we use Plancherel's theorem to write the term above as
\begin{eqnarray*}
\langle \eqref{R5},\hat{g} \rangle =(2 \pi)^3 \int_{\mathbb{R}^6} K_1(x,y) \overline{g_1}(t,y-x) \overline{f_1} (t,-x)~ dx dy ,
\end{eqnarray*}
where (we drop the time dependence of $K_1$ for convenience)
\begin{align*}
K_1 (x,z) &= \int_{\mathbb{R}^6} \frac{e^{i x \xi} \widehat{V'}(t,\xi)e^{iz \eta}}{\vert \eta + \xi \vert^2 - \vert \eta \vert^2} d\xi  d\eta \\
f_1 (t,x) &= \mathcal{F}^{-1} \big( e^{-it \vert \xi_1 \vert^2} \partial_{\xi_{1,j}} \widehat{f}(\xi_1) \frac{\xi_{1,j}}{\vert \xi_1 \vert^2} P_{k_1}(\xi_1) \big)\\
g_1 (t,x) &= \mathcal{F}^{-1} \big( \eta_l P_k (\eta) e^{-it \vert \eta \vert^2} \hat{g}(\eta)  \big) .
\end{align*}
Now using Lemma \ref{BSmain1} we can replace $K_1$ by its expression in terms of $L$ and switch to polar coordinates for the variable $y:$ 
\begin{align*}
\langle \eqref{R5},\hat{g} \rangle &= C \int_{\mathbb{R}^6} \frac{1}{\vert y \vert^2} L_{V'}(\vert y \vert-2\hat{y}\cdot x, \hat{y} ) \overline{g_1}(t,y-x) \overline{f_1} (t,-x) dx dy \\
                             &= C \int_{\mathbb{R}^3} \int_{\mathbb{S}^2} \int_{0} ^{\infty} L_{V'}(\rho - 2 \omega \cdot x,\omega) \overline{g_1}(t,\rho \omega-x) \overline{f_1} (t,-x) d\rho d\omega dx \\
                             &= C \int_{\mathbb{S}^2} \int_{0} ^{\infty} L_{V'}(\rho,\omega) \int_{\mathbb{R}^3} \textbf{1}_{\rho>-2\omega \cdot x} \overline{g_1}(t,2(\omega \cdot x)-x +\rho \omega) \overline{f_1} (t,-x) dx d\rho d\omega  .
\end{align*}
Now using the Cauchy-Schwarz inequality in $x$ we can write that
\begin{align*}
\vert \langle \eqref{R5},g \rangle \vert & \lesssim \int_{\mathbb{S}^2} \int_0 ^{\infty} \vert L_{V'} (\rho,\omega) \vert d\rho d\omega \Vert f_1 \Vert_{L^{\infty}_t L^2 _x} \Vert g_1 \Vert_{L^{\infty}_t L^2 _x}
\end{align*}
and we can conclude using the first estimate in Lemma \ref{BSmain2} as well as Lemma \ref{X'}:
\begin{align*}
\vert \langle \eqref{R5},g \rangle \vert & \lesssim \Vert V' \Vert_{L^{\infty}_t B_x} \Vert f_1 \Vert_{L^{\infty}_t L^2 _x} \Vert g_1 \Vert_{L^{\infty}_t L^2 _x} \\ 
                                         & \lesssim \Vert g \Vert_{L^2} \delta \varepsilon_1,
\end{align*}
and the desired result follows.
\end{proof}

As seen in Remark \ref{desing}, this yields the desired bound on \eqref{R5}. The very similar term \eqref{R5bis} can be estimated in the same way: 
\begin{lemma}[Estimating \eqref{R5bis}]
For $\vert k-k_1 \vert \leqslant 1$ we have the following bound:
\begin{align*}
& \Bigg \Vert P_k(\xi) \int_1 ^t \int_{\mathbb{R}^3} \frac{\xi_l \eta_{1,j}}{\vert \eta_1 \vert ^2} \frac{\partial_s \widehat{V}(s,\xi-\eta_1)}{\vert \xi \vert^2 - \vert \eta_1 \vert^2 } e^{i s( \vert \xi \vert^2 - \vert \eta_1 \vert^2)} \partial_{\eta_{1,j}} \widehat{f}(s,\eta_1)P_{k_1}(\eta_1) d\eta_1 ds \Bigg \Vert_{L^{\infty} _t L^2 _x} \\
& \lesssim \varepsilon_1 \Vert \partial_t V \Vert_{L^1 _t B_x} \\
& \lesssim \varepsilon_1 \delta .
\end{align*}
\end{lemma}
\bigskip
\subsection{The bilinear part}
In this subsection we prove the estimates of Proposition \ref{step1:estimates} for the so-called bilinear terms, namely \eqref{B1}, \eqref{B2} and \eqref{B3}. 
\begin{lemma} Assume that $\vert k-k_1 \vert >1$. \\
The following bound holds:
\begin{align*}
&\sum_{\vert k-k_1 \vert >1} \Bigg \Vert  P_k (\xi) \int_1 ^t \int_{\mathbb{R}^3}  \frac{\widehat{V}(s,\xi-\eta_1)}{\vert \xi \vert^2 - \vert \eta_1 \vert^2} P_{k_1}(\eta_1) \int_{\mathbb{R}^3 } is \xi_l  \widehat{f}(s,\eta_1-\eta_2) e^{is (\vert \xi \vert^2 - \vert \eta_2 \vert^2 - \vert \eta_1 - \eta_2 \vert^2)}  \widehat{f} (s,\eta_2) d\eta_2 ds d\eta_1 \Bigg \Vert_{L^{\infty}_t L^2 _x} \\
&\lesssim \varepsilon_1 ^2 \delta .
\end{align*}
\end{lemma}
\begin{proof}
Let's do the proof when $k < k_1-1$. The other case is similar. \\
First we discretize in the variable $\eta_2$ and $\eta_1 - \eta_2.$ We denote $k_2$ and $k_3$ the corresponding exponents. \\
To bound the term, we use Strichartz estimates, Lemma \ref{bilin}, Lemma \ref{symbolbis} and dispersive estimates:
\begin{align*}
\sum_{k_1 > k+1 } \Vert \eqref{B3} \Vert_{L^{\infty}_t L^2_x} & \lesssim \sum_{k_1 > k+1 } \Bigg \Vert \mathcal{F}^{-1}_{\xi} \int_{ \mathbb{R}^3 } is \xi_l \frac{P_k (\xi) P_{k_1} (\eta_1) \widehat{V}(t,\xi-\eta_1)}{\vert \xi \vert^2 - \vert \eta_1 \vert^2}  \widehat{u_{k_2} u_{k_3}}(\eta_1) \Bigg \Vert_{L^{4/3}_t L^{3/2}_x} \\
& \lesssim \sum_{k_1 > k+1 } 1.1^{k} 1.1^{-2k_1} \Vert V_{k_1} \Vert_{L^{\infty}_t L^{3}} \Vert t u_{k_2} u_{k_3} \Vert_{L^{4/3}_t L^{3}_x} \\
& \lesssim \sum_{k_1 > k+1 } 1.1^{k-k_1} \delta \bigg \Vert s \Vert u_{k_2} \Vert_{L^{6}_x}^{9/10} \Vert u_{k_3} \Vert_{L^{6}_x}^{9/10} \Vert u_{k_2} \Vert_{L^{6}_x}^{1/10} \Vert u_{k_3} \Vert_{L^{6}_x}^{1/10} \bigg \Vert_{L^{4/3}_t} \\
& \lesssim \sum_{k_1 > k+1 } 1.1^{k-k_1} \delta \bigg \Vert s \Vert u_{k_2} \Vert_{L^{6}}^{9/10} \Vert u_{k_3} \Vert_{L^{6}}^{9/10} \bigg \Vert_{L^{4/3}_t} \bigg \Vert \Vert u_{k_2} \Vert_{L^{6}}^{1/10} \Vert u_{k_3} \Vert_{L^{6}}^{1/10} \bigg \Vert_{L^{\infty}_t} \\
& \lesssim \sum_{k_1 > k+1 } 1.1^{k-k_1} \delta \Bigg( \int_1 ^t \bigg( s \frac{\varepsilon_1^{9/10}}{s^{9/10}} \frac{\varepsilon_1 ^{9/10}}{s^{9/10}} \bigg)^{4/3} ds \Bigg)^{3/4} \varepsilon_1 ^{2/10} \\ & \times \min \lbrace  1.1^{-k_2} , 1.1^{k_2 /10} \rbrace \min \lbrace 1.1^{-k_3} , 1.1^{k_3 /10} \rbrace \\
& \lesssim  \min \lbrace  1.1^{-k_2} , 1.1^{k_2 /10} \rbrace \min \lbrace 1.1^{-k_3} , 1.1^{k_3 /10} \rbrace  \varepsilon_1 ^2 \delta 
\end{align*}
and that last factor can be summed on $k_2, k_3$ which yields the desired result.
\end{proof}
Now we prove the desired estimate on \eqref{B2}.
\begin{lemma}[Estimating \eqref{B2}] \label{estimateB2}
We have
\begin{align*}
\sum_{k_1 \in \mathbb{Z}} \Bigg \Vert \mathcal{F}^{-1}_{\xi} P_k (\xi) \int_{1} ^{t} \int_{\mathbb{R}^3} e^{is (\vert \xi \vert^2 - \vert \eta_1 \vert^2 - \vert \xi-\eta_1 \vert^2)} \widehat{f_{k_1}}(s,\eta_1) \partial_{\xi} \widehat{f}(s,\xi-\eta_1) d\eta_1 ds \Bigg \Vert_{L^{\infty}_t L^2 _x} \lesssim  \varepsilon_1 ^2 .
\end{align*}
\end{lemma}
\begin{proof}
We start by splitting the frequency $\xi-\eta_1$ dyadically. The corresponding exponent is denoted $k_2$. We also split the time integral dyadically, and denote $m$ the exponent: 
\begin{align*}
&  P_k (\xi) \int_{1.1^m} ^{1.1^{m+1}} \int_{\mathbb{R}^3} e^{is (\vert \xi \vert^2 - \vert \eta_1 \vert^2 - \vert \xi-\eta_1 \vert^2)} \widehat{f_{k_1}}(s,\eta_1) \partial_{\xi} \widehat{f}(s,\xi-\eta_1) d\eta_1 ds  =: \sum_{k_2 \in \mathbb{Z}} I_{m,k_1,k_2}.
\end{align*}
\underline{Case 1: $\max \lbrace k_1 , k_2 \rbrace \geqslant m $} \\
Then we can use the bilinear estimates from Lemma \ref{bilin}: we place the high frequency term in $L^2$ and use the energy bound. The lower frequency term is placed in $L^{\infty}$ and then we use Bernstein's inequality or the energy bound.
\begin{align*}
\Vert I_{m,k_1,k_2} \Vert_{L^{\infty} _t, L^2 _x} & \lesssim 1.1^{m} 1.1^{-10 \max \lbrace k_1 , k_2 \rbrace } \min \lbrace 1.1^{-10 \min \lbrace k_1,k_2 \rbrace} , 1.1^{3 \min \lbrace k_1,k_2 \rbrace /2} \rbrace \varepsilon_1 ^2 \\
& \lesssim 1.1^{- 8 m} 1.1^{-\max \lbrace k_1,k_2 \rbrace} \min \lbrace 1.1^{-10 \min \lbrace k_1,k_2 \rbrace} , 1.1^{3 \min \lbrace k_1,k_2 \rbrace /2} \rbrace \varepsilon_1 ^2 ,
\end{align*}
and the last term can be summed over $k_1, k_2.$
\\
\underline{Case 2: $\min \lbrace k_1,k_2 \rbrace \leqslant -2m$} \\
Then using again the bilinear estimates from Lemma \ref{bilin}, we place the low frequency term in $L^{\infty}$ and the high frequency term in $L^2.$ Then we use Bernstein's inequality for the low frequency term and either Bernstein's inequality or the energy bound for the high frequency term.  
\begin{align*}
\Vert I_{m,k_1,k_2} \Vert_{L^{\infty}_t L^2 _x} & \lesssim 1.1^{m} 1.1^{3 \min \lbrace k_1,k_2 \rbrace /2}   \min \lbrace 1.1^{-10 \max \lbrace k_1,k_2 \rbrace} , 1.1^{3 \max \lbrace k_1,k_2 \rbrace /4} \rbrace \varepsilon_1 ^2 \\
& \lesssim 1.1^{-m} 1.1^{\min \lbrace k_1,k_2 \rbrace /2} \min \lbrace 1.1^{-10 \max \lbrace k_1,k_2 \rbrace} , 1.1^{3 \max \lbrace k_1,k_2 \rbrace /4} \rbrace \varepsilon_1 ^2  
\end{align*}
which can be summed.
\\
\underline{Case 3: $ -2m \leqslant k_1 ,k_2 \leqslant m $}
\\
Then we have $O(m^2)$ terms in the sums on $k_1,k_2$ therefore a factor like $1.1^{-\alpha m}$ ($\alpha >0$) will be enough to ensure convergence of the sums. \\
Now we can bound the main term using Strichartz estimates and Lemma \ref{X'}:
\begin{align*}
\Vert I_{m,k_1,k_2 } \Vert_{L^{\infty}_t L^2 _x} & = \Bigg \Vert \int_{1.1^m} ^{1.1^{m+1}} e^{-is \Delta} \bigg( \big(e^{is \Delta} f_{k_1} \big) \big( e^{is \Delta} (x f)_{k_2} \big) \bigg) ds \Bigg \Vert_{L^{\infty}_t L^2 _x} \\
& \lesssim \bigg \Vert \big(e^{is \Delta} f_{k_1} \big) \big( e^{is \Delta} (x f)_{k_2} \big)  \bigg \Vert_{L^{4/3}_t L^{3/2} _x} \\
& \lesssim \Vert e^{it \Delta} f_{k_1} \Vert_{L^{4/3} _t L^{6} _x} \Vert (x f)_{k_2}  \Vert_{L^{\infty} _t L^{2}_x} \\
& \lesssim 1.1^{-m/4} \varepsilon_1 ^2 .
\end{align*}
\end{proof}
Now we move on to the most challenging term \eqref{B1}. 
\begin{lemma}[Estimating \eqref{B1}] \label{bilinhard}
We have:
\begin{eqnarray*}
\sum_{k_1 \in \mathbb{Z}} \Bigg \Vert \mathcal{F}^{-1}_{\xi} P_k(\xi) \int_1 ^t \int_{\mathbb{R}^3} is \eta_{1,l} e^{is(\vert \xi \vert^2-\vert \eta_1 \vert^2 - \vert \xi -\eta_1 \vert^2)} \widehat{f_{k_1}}(s,\eta_1) \widehat{f}(s,\xi-\eta_1) d\eta_1 ds \Bigg \Vert_{L^{\infty}_t L^2 _x} \lesssim \varepsilon_1 ^2 .
\end{eqnarray*}
\end{lemma}
\begin{proof}
We split the frequency $\xi-\eta_1$ dyadically ($k_2$ denotes the corresponding exponent) as well as time: 
\begin{eqnarray*}
P_k (\xi) \int_{1} ^{t} is \eta_{1,l} e^{is (\vert \xi \vert^2 - \vert \eta_1 \vert ^2 - \vert \xi - \eta_1 \vert ^2)} \widehat{f_{k_1}}(s,\eta_1) \widehat{f}(s,\xi-\eta_1) d\eta_1 ds =: \sum_{m=0}^{\ln t} \sum_{k_2 \in \mathbb{Z}} J_{m,k_1,k_2}.
\end{eqnarray*}
\underline{Case 1: $\max \lbrace k_1 ; k_2 \rbrace \geqslant m $}
\\
We use the bilinear estimates of Lemma \ref{bilin} by placing the largest frequency in $L^2$ and the other one in $L^{\infty}:$ then we use the bootstrap assumption on the high frequency term and either Bernstein's inequality or the energy bound for the low frequency term:
\begin{eqnarray*}
\Vert I_{m,k_1,k_2} \Vert_{L^{\infty}_t L^2 _x} & \lesssim & 1.1^{2m} 1.1^{\max \lbrace k_1 ; k_2 \rbrace} 1.1^{-10 \max \lbrace k_1 ; k_2 \rbrace} \min \big \lbrace 1.1^{-10 \min \lbrace k_1 ; k_2 \rbrace}; 1.1^{3 \min \lbrace k_1 ; k_2 \rbrace /2} \big \rbrace \varepsilon_1 ^2 \\
                                       & \lesssim &1.1^{-6m} 1.1^{\max \lbrace k_1 ; k_2 \rbrace} \min \big \lbrace 1.1^{-10 \min \lbrace k_1 ; k_2 \rbrace}; 1.1^{3 \min \lbrace k_1 ; k_2 \rbrace /2} \big \rbrace  \varepsilon_1 ^2 ,
\end{eqnarray*}
which we can sum. \\
\underline{Case 2: $\min \lbrace k_1 ; k_2 \rbrace \leqslant -2m $} \\
In this case we use Lemma \ref{bilin} by placing the low frequency in $L^{\infty}$ and the high frequency term in $L^2.$ Then we use Bernstein's inequality to obtain
\begin{eqnarray*}
\Vert I_{m,k_1,k_2} \Vert_{L^{\infty}_t L^2 _x} & \lesssim & 1.1^{2m} 1.1^{\max \lbrace k_1 ; k_2 \rbrace} 1.1^{3 \min \lbrace k_1 ; k_2 \rbrace /2} \min \big \lbrace 1.1^{-10 \max \lbrace k_1 ; k_2 \rbrace}; 1.1^{\max \lbrace k_1 ; k_2 \rbrace /3} \big \rbrace \varepsilon _1 ^2 \\
                                       & \lesssim & 1.1^{-0.5 m} 1.1^{0.25 \min \lbrace k_1 ; k_2 \rbrace}  1.1^{\max \lbrace k_1 ; k_2 \rbrace} \min \big \lbrace 1.1^{-10 \max \lbrace k_1 ; k_2 \rbrace}; 1.1^{\max \lbrace k_1 ; k_2 \rbrace /3} \big \rbrace \varepsilon _1 ^2 ,
\end{eqnarray*}
which can be summed.
\\
\underline{Case 3: $-2m \leqslant k_1 , k_2 \leqslant m$} \\
First notice that the sum in $k_1, k_2$ has $O(m^2)$ terms. Therefore a factor $1.1^{-\alpha m}$ in the estimates (for some small positive real number $\alpha$) will be enough to ensure convergence of the sum in $k_1 ,k_2.$ \\
\\
When the gradient of the phase is not too small, we can integrate by parts in $\eta_1$ to gain decay in time. To quantify this more precisely, we split dyadically in the gradient of the phase, namely $\xi-2\eta_1.$ We denote $k_3$ the corresponding exponent.\\
\\
\underline{Case 3.1: $k_3 \leqslant -10 m$} \\
Using the Cauchy-Schwarz inequality we write that
\begin{align*}
& \Bigg \vert \int_{\mathbb{R}^3 } P_{k_3} (\xi - 2 \eta_1) s \eta_{1,l} e^{-is \vert \eta_1 \vert^2} \widehat{f_{k_1}}(s,\eta_1) e^{-is \vert \xi - \eta_1 \vert^2} \widehat{f_{k_2}} (s,\xi-\eta_1) d\eta_1 \Bigg \vert \\
& \leqslant  \Bigg( \int_{ \mathbb{R}^3} P_{k_3} (\xi - 2 \eta_1) \vert s \eta_{1,l} \widehat{f_{k_1}} (s,\eta_1) \vert^2 d \eta_1 \Bigg)^{1/2}  \Bigg( \int_{\mathbb{R}^3 } \vert \widehat{f_{k_2}} (s,\xi-\eta_1) \vert^2 d \eta_1 \Bigg)^{1/2}   .
\end{align*}
Therefore 
\begin{align*}
& \int_{\mathbb{R}^3} \Bigg \vert \int_{ \mathbb{R}^3} P_{k_3} (\xi - 2 \eta_1) s \eta_{1,l} \widehat{f_{k_1}}(s,\eta_1) \widehat{f_{k_2}} (s,\xi-\eta_1) d\eta_1 \Bigg \vert ^2 d\xi \\
& \leqslant  \Vert f_{k_2} \Vert_{L^2 _x} ^2 \int_{\mathbb{R}^3} \int_{ \mathbb{R}^3 } \vert s \eta_{1,l} \widehat{f_{k_1}} (s,\eta_1) \vert ^2 P_{k_3} (\xi -2\eta_1)  d\eta_1   d\xi \\
& \leqslant  \Vert f_{k_2} \Vert_{L^2 _x} ^2 \Vert s \nabla f_{k_1} \Vert_{L^2 _x} ^2 1.1^{3k_3} .
\end{align*}
As a result
\begin{eqnarray*}
&& \Bigg \Vert P_k (\xi) \int_{1.1^m} ^{1.1^{m+1}} \int_{\mathbb{R}^3} is \eta_{1,l} e^{is (\vert \xi \vert^2 - \vert \eta_1 \vert^2 - \vert \xi-\eta_1 \vert^2)} \widehat{f_{k_1}}(s,\eta_1) \widehat{f_{k_2}}(s,\xi-\eta_1) P_{k_3} (\xi-2\eta_1) d\eta_1 ds \Bigg \Vert_{L^{\infty}_t L^2 _x} \\
& \lesssim & 1.1^{2m} 1.1^{k_1} \Vert \widehat{f_{k_1}} \Vert_{L^{\infty}_t L^2 _x} \Vert f_{k_2} \Vert_{L^{\infty}_t L^2 _x} 1.1^{3 k_3 /2} \\
& \lesssim & 1.1^{2m} 1.1^{-14 m} 1.1^{0.1 k_3} 1.1^{k_1} \varepsilon_1 ^2 \\
& \lesssim & 1.1^{-11 m} 1.1^{0.1 k_3} \varepsilon_1 ^2 .
\end{eqnarray*}
We can sum over $k_3.$ Then since there are only $O(m^2)$ terms in the sums on $k_1 , k_2$ we can sum over $k_1 , k_2.$\\
\\
\underline{Case 3.2: $k_3 \geqslant k_1-10, k_3 \geqslant -10m,$ and $-2m \leqslant k_1,k_2 \leqslant m$}\\
In this case we do an integration by parts in $\eta$. We obtain: 
\begin{align}
\notag I_{m,k_1,k_2,k_3} &= \\
\label{space1}& -\int_{1.1^m} ^{1.1^{m+1}} \int_{\mathbb{R}^3} P_{k_3}(\xi-2\eta_1) \frac{\eta_{1,l} (\xi-2 \eta_1)_j}{2\vert \xi- 2 \eta_1 \vert ^2} e^{is (\vert \xi \vert^2 - \vert \eta_1 \vert^2 - \vert \xi-\eta_1 \vert^2)} \partial_{\eta_{1,j}} \widehat{f_{k_1}}(s,\eta_1) \widehat{f_{k_2}}(s,\xi-\eta_1) d\eta_1 ds \\
\label{space3} &- \int_{1.1^m} ^{1.1^{m+1}} \int_{\mathbb{R}^3} P_{k_3}(\xi-2\eta_1) \partial_{\eta_{1,j}}  \bigg( \frac{\eta_{1,l} (\xi-2 \eta_1)}{2\vert \xi- 2 \eta_1 \vert ^2} \bigg) e^{is (\vert \xi \vert^2 - \vert \eta_1 \vert^2 - \vert \xi-\eta_1 \vert^2)} \widehat{f_{k_1}}(s,\eta_1) \widehat{f_{k_2}}(s,\xi-\eta_1) d\eta_1 ds \\
\label{space4}&-\int_{1.1^m} ^{1.1^{m+1}} \int_{\mathbb{R}^3} P_{k_3}(\xi-2\eta_1) 1.1^{-k_1} \frac{\eta_{1,j}}{\vert \eta_1 \vert} \phi'(1.1^{-k_1} \vert \eta_1 \vert) \\
\notag& \times \frac{\eta_{1,l} (\xi-2 \eta_1)}{2\vert \xi- 2 \eta_1 \vert ^2} e^{is (\vert \xi \vert^2 - \vert \eta_1 \vert^2 - \vert \xi-\eta_1 \vert^2)} \widehat{f_{k_1}}(s,\eta_1) \widehat{f_{k_2}}(s,\xi-\eta_1) d\eta_1 ds \\
\notag & + \lbrace \textrm{similar terms} \rbrace ,
\end{align}
where similar terms refer to the case where the derivative hits the other profile, the localization at frequency $1.1^{k_3}$ (which is handled in the same way as \eqref{space3}), or the localization at frequency $1.1^{k_2}$ (which is handled as \eqref{space4}). \\
\\
For \eqref{space1}, we can use Strichartz estimates to write that
\begin{align*}
\Vert \eqref{space1} \Vert_{L^{\infty}_t L^2_x} &\lesssim  \Bigg \Vert \mathcal{F}^{-1}  \int_{ \mathbb{R}^3 } P_{k_3}(\xi-2\eta_1) \frac{\eta_{1,l} (\xi-2 \eta_1)_j}{\vert \xi- 2 \eta_1 \vert ^2} e^{is (- \vert \eta_1 \vert^2 - \vert \xi-\eta_1 \vert^2)} \partial_{\eta_1} \widehat{f_{k_1}}(s,\eta_1) \widehat{f_{k_2}}(s,\xi-\eta_1) d\eta_1 \Bigg \Vert_{L^{4/3} _t L^{3/2} _x} .
\end{align*}
Now we can use Lemma \ref{bilin} with multiplier 
\begin{align*}
m(\xi,\eta) = \frac{\eta_{1,l}(\xi-2\eta_1)_j P_{k_1}(\eta_1) P_{k_2}(\xi-\eta_1) P_{k_3}(\xi-2\eta_1)}{\vert \xi - 2 \eta_1 \vert^2} ,
\end{align*}
which by Lemma \ref{symbol} satisfies
\begin{align*}
\Vert \check{m} \Vert_{L^1} \lesssim 1.1^{k_1-k_3} .
\end{align*}
Now we can conclude that
\begin{eqnarray*}
\Vert \eqref{space1} \Vert_{L^{\infty}_t L^2_x} & \lesssim & 1.1^{k_1-k_3} \Vert \partial_{\eta_1} \widehat{f_{k_1}} \Vert_{L^{\infty}_t L^2 _x} \Vert e^{is \Delta} f_{k_2} \Vert_{L^{4/3} _t L^6 _x} \\
& \lesssim & 1.1^{k_1-k_3} \Vert \partial_{\eta_1} \widehat{f_{k_1}} \Vert_{L^{\infty} _t L^2 _x} (1.1^{m-4/3m})^{3/4} \varepsilon_1
\\
& \lesssim & 1.1^{k_1-k_3} 1.1^{-m/4} \varepsilon_1 ^2.
\end{eqnarray*}
Thanks to the factor in front and the condition $k_3 \geqslant k_1+10$ we can sum over $k_3.$ There remains the sums in $k_1,k_2$ but given that there are only $O(m^2)$ terms in $k_1,k_2,$ the factor $1.1^{-m/4} $ is enough to make the sums converge.  
\\
\\
Similarly for \eqref{space3}, we can write 
\begin{align*}
\Vert \eqref{space3} \Vert_{L^{\infty}_t L^2 _x} & \lesssim 1.1^{k_1-2k_3} \Vert \widehat{f_{k_1}} \Vert_{L^{\infty}_t L^2 _x} (1.1^{m-4/3m})^{3/4} \varepsilon_1
\end{align*}
and now we use Bernstein's inequality on the factor in $L^2.$ Since the $X-$norm of $f_{k_1}$ controls all its Lebesgue norms down to $6/5+$ we can write
\begin{align*}
\Vert \eqref{space3} \Vert_{L^{\infty}_t L^2 _x} & \lesssim 1.1^{2k_1-2k_3} 1.1^{-0.0001 k_1} 1.1^{-m/4} \varepsilon_1 ^2,
\end{align*}
but $k_1 \geqslant - 2 m $ therefore $- 0.0001 k_1 \leqslant 2 \times 0.0001 m$  and we conclude
\begin{align*}
\Vert \eqref{space3} \Vert_{L^{\infty}_t L^2 _x} & \lesssim 1.1^{2k_1-2k_3} 1.1^{-m/8} \varepsilon_1 ^2.
\end{align*}
Therefore we can sum over $k_3$ and then over $k_1, k_2$ as we did previously. 
\\
\\
Finally we can bound \eqref{space4} using a slight variation on this strategy:
\begin{align*}
\Vert \eqref{space4} \Vert_{L^{\infty}_t L^2 _x} & \lesssim 1.1^{k_1 - k_3} 1.1^{-k_1} \Vert \widehat{f_{k_1}} \Vert_{L^{\infty}_t L^2 _x} \Vert e^{it \Delta} f_{k_2} \Vert_{L^{4/3}_t L^6 _x} \\
& \lesssim 1.1^{k_1 - k_3} 1.1^{-0.0001 k_1} (1.1^{m-4/3 m})^{3/4} \varepsilon_1 ^2 ,
\end{align*}
and we conclude as we did before. \\
\\
\underline{Case 3.3: $-10 m \leqslant k_3 \leqslant k_1-10$ and $-2m \leqslant k_1,k_2 \leqslant m$} \\
Note that in this case the sums on $k_1,k_2,k_3$ have $O(m^3)$ terms therefore a term like $1.1^{-\alpha m}$ in the estimates (for $\alpha>0$) is enough to ensure that all the sums converge. \\
\\
Let's start with a further restriction: notice that $\xi - \eta = \xi- 2\eta + \eta$ therefore $ \vert \xi-\eta \vert \sim 1.1^{k_1} \sim 1.1^{k_2}.$ \\
Note that, using the bilinear estimates of Lemma \ref{bilin} and Bernstein's inequality:
\begin{eqnarray*}
\Vert I_{m,k_1,k_2,k_3} \Vert_{L^{\infty}_t L^2 _x} & \lesssim & 1.1^{2m} 1.1^{k_1} \Vert u_{k_1} \Vert_{L^{\infty} _t L^{\infty}_x} \Vert u_{k_1} \Vert_{L^{\infty} _t L^{2}_x}.
\end{eqnarray*}
But by Bernstein's inequality, the isometry property of the Schr\"{o}dinger group and Bernstein's inequality again (since the $X$ norm of $f_{k_1}$ controls its $L^p -$ norms for $6/5<p\leqslant 2$)
\begin{align*}
\Vert u_{k_1} \Vert_{L^{\infty}_x} & \lesssim 1.1^{3/2 k_1} \Vert u_{k_1} \Vert_{L^2 _x} \\
                                   &\lesssim 1.1^{3/2 k_1} \Vert f_{k_1} \Vert_{L^2 _x}\\
                                   & \lesssim 1.1^{2.49 k_1} \varepsilon_1 .
\end{align*}
Therefore
\begin{eqnarray*}
\Vert I_{m,k_1,k_2,k_3} \Vert_{L^{\infty}_t L^2 _x} & \lesssim & 1.1^{2m} 1.1^{k_1} 1.1^{2.49 k_1} 1.1^{0.99 k_1} \varepsilon_1 ^2 \\
                                                    & \lesssim & 1.1^{2m} 1.1^{4.48 k_1} \varepsilon_1 ^2 .
\end{eqnarray*}
If $k_1 \leqslant -101/224 m$ we can sum the expressions above. \\
As a result we can assume from now on that $k_1 >-101/224 m.$
\\
\\
Now we integrate by parts in time and obtain 
\begin{eqnarray*}
I_{m,k_1,k_2,k_3}&=&- \int_{1.1^m} ^{1.1^{m+1}} \int_{\mathbb{R}^3} P_{k_3} (\xi-2\eta_1) \frac{i s \eta_{1,l}}{\eta_1 \cdot (\xi-\eta_1)} e^{is (\vert \xi \vert^2- \vert \eta_1 \vert^2 - \vert \xi-\eta_1 \vert^2)} \partial_s \widehat{f_{k_1}}(s,\eta_1) \widehat{f_{k_2}}(s,\xi-\eta_1) d\eta_1 ds \\
&-& \int_{1.1^m} ^{1.1^{m+1}} \int_{\mathbb{R}^3} P_{k_3} (\xi-2\eta_1) \frac{i s \eta_{1,l}}{\eta_1 \cdot (\xi-\eta_1)} e^{is (\vert \xi \vert^2- \vert \eta_1 \vert^2 - \vert \xi-\eta_1 \vert^2)} \widehat{f_{k_1}}(s,\eta_1) \partial_s \widehat{f_{k_2}}(s,\xi-\eta_1) d\eta_1 ds \\
&-& \int_{1.1^m} ^{1.1^{m+1}} \int_{\mathbb{R}^3} P_{k_3} (\xi-2\eta_1) \frac{i \eta_{1,l}}{\eta_1 \cdot (\xi-\eta_1)} e^{is (\vert \xi \vert^2- \vert \eta_1 \vert^2 - \vert \xi-\eta_1 \vert^2)} \widehat{f_{k_1}}(s,\eta_1) \widehat{f_{k_2}}(s,\xi-\eta_1) d\eta_1 ds \\
 &-& \int_{\mathbb{R}^3} P_{k_3} (\xi-2\eta_1) \frac{i 1.1^m \eta_{1,l}}{\eta_1 \cdot (\xi-\eta_1)} e^{i 1.1^m  (\vert \xi \vert ^2 - \vert \eta_1 \vert ^2 - \vert \xi-\eta_1 \vert^2)} \widehat{f_{k_2}}(1.1^m,\xi- \eta_1) \widehat{f_{k_1}}(1.1^m, \eta_1) d\eta_1 \\
 &+& \int_{\mathbb{R}^3} P_{k_3} (\xi-2\eta_1) \frac{i 1.1^{m+1} \eta_{1,l}}{\eta_1 \cdot (\xi-\eta_1)} e^{i 1.1^{m+1}  (\vert \xi \vert ^2 - \vert \eta_1 \vert ^2 - \vert \xi-\eta_1 \vert^2)} \widehat{f_{k_2}}(1.1^{m+1},\xi- \eta_1) \widehat{f_{k_1}}(1.1^{m+1}, \eta_1) d\eta_1.
\end{eqnarray*}
Let's start with the easier boundary terms. They both have the form:
\begin{eqnarray*}
J_{m,k_1,k_2} = \int_{\eta_1} \frac{i t \eta_{1,l} }{\eta_1 \cdot (\xi - \eta_1)} e^{it(\vert \xi \vert^2 - \vert \eta_1 \vert ^2 - \vert \xi-\eta_1 \vert^2)} \widehat{f_{k_1}}(s,\eta_1) \widehat{f_{k_2}}(s, \xi-\eta_1) P_{k_3} (\xi-2\eta_1) d\eta_1 
\end{eqnarray*}
for $t \simeq 1.1^m.$ \\
Let 
\begin{align*}
m(\xi,\eta) = \frac{\eta_{1,l} P_{k_1}(\eta_1) P_{k_2}(\xi-\eta_1) P_{k_3}(\xi- 2\eta_1)}{\eta_1 \cdot (\xi - \eta_1)}.
\end{align*}
Using Lemma \ref{symbolbisbis}:
\begin{align*}
\Vert \check{m} \Vert_{L^1} \lesssim 1.1^{-k_1}.
\end{align*}
Now we can use $L^3 - L^6$ bilinear estimates from Lemma \ref{bilin} as well as dispersive estimates to write that
\begin{eqnarray*}
\Vert J_{m,k_1,k_2} \Vert_{L^{\infty}L^2} & \lesssim & 1.1^m 1.1^{-k_1} 1.1^{-m} 1.1^{-m/2} \varepsilon_1 ^2
\\
                                          & \lesssim & 1.1^m 1.1^{101/224 m} 1.1^{-3m/2} \varepsilon_1 ^2
\end{eqnarray*}
where for the last line we used the restriction above on $k_1: k_1>-101/224 \Rightarrow -k_1 < 101/224 \simeq 0.45.$ \\
Now notice that we have $O(m^3)$ terms in our sums over $k_1,k_2,k_3$. Therefore the bound above is enough to ensure convergence. \\
Let's look at the most threatening part: 
\begin{eqnarray*}
&& \int_{1.1^m}^{1.1^{m+1}} \int_{\mathbb{R}^3} \frac{i s \eta_{1,l}}{\eta_1 \cdot (\xi-\eta_1)} e^{is(\vert \xi \vert^2 - \vert \eta_1 \vert^2 - \vert \xi- \eta_1 \vert^2)} \partial_s \widehat{f_{k_1}} (s,\eta_1) \widehat{f}(s,\xi-\eta_1) d\eta_1 ds  \\ 
&=& \int_{1.1^m}^{1.1^{m+1}} \int_{ \mathbb{R}^3 } \frac{i s \eta_{1,l}}{\eta_1 \cdot (\xi-\eta_1)} e^{is(\vert \xi \vert^2  - \vert \xi- \eta_1 \vert^2)} \widehat{f}(s,\xi-\eta_1) P_{k_1}(\eta_1) \int_{ \mathbb{R}^3} \widehat{V}(s,\eta_1 - \eta_2) \widehat{u}(s,\eta_2) d\eta_2 d\eta_1 ds \\
&+& \int_{1.1^m}^{1.1^{m+1}} \int_{\mathbb{R}^3} \frac{i s \eta_{1,l}}{\eta_1 \cdot (\xi-\eta_1)} e^{is(\vert \xi \vert^2  - \vert \xi- \eta_1 \vert^2)} \widehat{f}(s,\xi-\eta_1) P_{k_1}(\eta_1) \int_{\mathbb{R}^3} \widehat{u}(\eta_1 - \eta_2) \widehat{u}(s,\eta_2) d\eta_2 d\eta_1 ds \\
&:=& I+II.
\end{eqnarray*}
\textit{Estimate on $I$:} \\
We use Strichartz estimates, Bernstein's inequality and Lemma \ref{symbolbisbis} to write that
\begin{eqnarray*} 
\Vert I \Vert_{L^{\infty}_t L^2_x} & \lesssim & \Bigg \Vert \mathcal{F}^{-1} \int_{ \mathbb{R}^3} P_{k_1} (\eta_1) \frac{is \eta_{1,l}}{\eta_1 \cdot (\xi-\eta_1)} \mathcal{F}(Vu)(s,\eta_1) \widehat{u_{k_2}}(s,\xi-\eta_1) d\eta_1 \Bigg \Vert_{L^{4/3}_t L^{3/2}_x} \\
& \lesssim & 1.1^{-k_1} \Vert t (Vu)_{k_1} \Vert_{L^{\infty}_t L^2 _x} \Vert u_{k_2} \Vert_{L^{4/3}_t L^{6}_x} \\
& \lesssim & \Vert t (Vu)_{k_1} \Vert_{L^{\infty}_t L^{6/5} _x}  \Vert u \Vert_{L^{4/3}_t L^{6}_x} \\
& \lesssim & \Vert t u \Vert_{L^{\infty}_t L^6 _x} \Vert V \Vert_{L^{\infty} _t B_x} 1.1^{-m/4} \varepsilon _1 .
\end{eqnarray*}
But now from Lemma \ref{decay} we have that
\begin{align*}
\Vert t u \Vert_{L^{\infty} _t L^6 _x} \lesssim 1.1^{0.01m} \varepsilon_1 
\end{align*}
therefore
\begin{align*}
\Vert I \Vert_{L^{\infty}_t L^2 _x} \lesssim 1.1^{-m/4} 1.1^{0.01 m} \delta \varepsilon_1 .
\end{align*}
As it was the case above, we have $O(m^3)$ terms in our sums over $k_1, k_2,k_3$. Therefore the bound above is enough to ensure convergence. 
\\
\\
\textit{Estimate on $II:$}
We use Strichartz's inequality as well as Bernstein's inequality, dispersive estimates and Lemma \ref{symbolbisbis} to write that
\begin{eqnarray*} 
\Vert II \Vert_{L^{\infty}_t L^2_x} & \lesssim & \Bigg \Vert \mathcal{F}^{-1} \int_{\mathbb{R}^3} \frac{is \eta_{1,l}P_{k_1}(\eta_1)}{\eta_1 \cdot (\xi-\eta_1)} \mathcal{F}(u^2)(s,\eta_1) \widehat{u_{k_2}}(s,\xi-\eta_1) d\eta_1 \Bigg \Vert_{L^{4/3} _t L^{3/2} _x} \\
& \lesssim & 1.1^m 1.1^{-k_1} \Vert u^2 \Vert_{L^{\infty}_t L^2 _x} \Vert u_{k_2} \Vert_{L^{4/3}_t L^{6} _x} \\
& \lesssim & 1.1^m 1.1^{-k_1} \Vert u \Vert_{L^{\infty}_t L^4 _x} ^2 \Vert u_{k_2} \Vert_{L^{4/3}_t L^{6} _x} .
\end{eqnarray*}
Recall that from Lemma \ref{decay} we have that
\begin{align*}
\Vert u \Vert_{L^4 _x} \lesssim 1.1^{-0.745m} \varepsilon_1
\end{align*}
therefore using again the fact that $k_1>-101/224,$ we have
\begin{align*}
\Vert II \Vert_{L^{\infty}_t L^2 _x} & \lesssim 1.1^m 1.1^{101/224 m} 1.1^{-1.49 m} 1.1^{-m/4} \varepsilon_1 ^2 .
\end{align*}
Since there are $O(m^3)$ terms in the sum over $k_1,k_2$ and $k_3$ we can sum and deduce the desired result in this case.
\end{proof}

\section{Some multilinear analysis} \label{multilinanalysis}
In Section \ref{series-repr}, we saw that the following kernels appear systematically in the iterates:
\begin{align*}
&\displaystyle \int_{\big(\mathbb{R}^3\big)^{n-2}} \prod_{l=1}^{n-1} \frac{\widehat{V}(s,\eta_{l-1} - \eta_{l}) P_{k_l}(\eta_l) P_{k}(\xi)}{\vert \xi \vert^2 - \vert \eta_l \vert^2} d\eta_{1} ... d\eta_{n-2} 
\\
&= \displaystyle \int_{\big(\mathbb{R}^3\big)^{n-2}} \prod_{l=1}^{n-1} \frac{\widehat{V'}(s,\eta_{l} - \eta_{l-1}) P_{k_l}(\eta_l) P_{k}(\xi)}{\vert \xi \vert^2 - \vert \eta_l \vert^2} d\eta_{1} ... d\eta_{n-2} ,
\end{align*}
where we used the convention $\eta_0 = \xi.$ We also denoted $\widehat{V'}(t,\xi)=\widehat{V}(t,-\xi).$ In what follows we will sometimes abuse notations and write $V$ instead of $V'.$ There is another slight abuse of notation here since when the $i-$th factor is such that $k_i - k < -1$ then the localization reads $P_{\leqslant k-2}(\eta_i)$ and not $P_{k_i}(\eta_i)$. In this whole section since the estimates are written for fixed times, we drop the time-dependence of the potential to improve legibility.

Now we relabel $\xi$ as $\eta$ and $\eta_n$ as $\xi_n + \eta$ we end up with the expression
\begin{align*}
 \displaystyle \int_{\big(\mathbb{R}^3\big)^{n-2}} \prod_{l=1}^{n-1} \frac{\widehat{V'}(\xi_{l} - \xi_{l-1}) P_{k_l}(\xi_l+\eta) P_{k}(\eta)}{\vert \xi_l + \eta \vert^2 - \vert \eta \vert^2} d\xi_{1} ... d\xi_{n-2} .
\end{align*}

The first goal in this section is to prove bounds on the operator norm of the associated linear operators, and in particular show their geometric dependence on the $B_x$ norm of the potential. 

The kernels in the present article are very close to expressions that appear in \cite{BSmain}, where their norms are shown to be geometric in a similar norm of the potential. Therefore we will see that we can adapt some results of this paper in our setting. Because  of the frequency localizations we cannot use the results of \cite{BSmain} directly. The first two subsections 6.1 and 6.2 are dedicated to adapting the results of this paper. More precisely, the idea is to find recursive relations between the kernels, and use them to deduce the desired estimates from the case where only one copy of the potential is present.

We will then be in a position to prove generalizations of Lemma \ref{bilin} in subsection 6.3. They are needed to bound the iterates of the remainder terms from the expansion: the added difficulty is that we must keep track of the dependence on $V$ in the estimates. 

\subsection{Notations and useful formulas}
\subsubsection{Notations}
In this section we introduce some notations from \cite{BSmain} and record some formulas that will be useful in subsequent sections.
We start by defining the following operation $\circledast$ for kernels $T_1$ and $T_2$ on $\big(\mathbb{R}^3 \big)^3$ (following Beceanu, \cite{B}):
\begin{eqnarray*}
&&\mathcal{F}^{-1}_{x_0} \mathcal{F}_{x_2,y} \big(T_1 \circledast T_2 \big)(\xi_0 , \xi_2,\eta) \\
&=& \int_{\mathbb{R}^3} \big(\mathcal{F}^{-1}_{x_0} \mathcal{F}_{x_1,y} T_1 \big)(\xi_0,\xi_1,\eta) \big( \mathcal{F}^{-1}_{x_1} \mathcal{F}_{x_2,y} T_2 \big) (\xi_1,\xi_2,\eta) d\xi_1.
\end{eqnarray*}
In physical space this becomes
\begin{align*}
(T_1 \circledast T_2)(x_0 , x_2 ,y) = \int_{\mathbb{R}^6} T_1 (x_0 , x_1 ,y_1) T_2 (x_1 , x_2 , y-y_1) dx_1 dy_1.
\end{align*}
To represent the kernels that appear in this work, we introduce a three building blocks: \\
Let $T_1(x_0,x,y)$ be the kernel defined as
\begin{align*}
\big(\mathcal{F}^{-1}_{x_0} \mathcal{F}_{x_1,y} T_1 \big)(\xi_0, \xi_1,\eta) := \frac{\widehat{V'}(\xi_1 - \xi_0)P_k(\eta) P_{k_1}(\xi_1+\eta) }{\vert \xi_1 + \eta \vert^2 - \vert \eta \vert ^2}.
\end{align*}
Let also $T_2 (x_0,x,y)$ be the kernel defined by
\begin{align*}
\big(\mathcal{F}^{-1}_{x_0} \mathcal{F}_{x_1,y} T_2 \big)(\xi_0, \xi_1,\eta) := \frac{\widehat{V'}(\xi_1 - \xi_0)P_k(\eta) P_{\leqslant k-2}(\xi_1+\eta) }{\vert \xi_1 + \eta \vert^2 - \vert \eta \vert ^2}.
\end{align*}
Recall also that we defined earlier in Lemma \ref{BSmain1} a kernel that appears in the work of Beceanu and Schlag:
\begin{align*}
\big(\mathcal{F}^{-1}_{x_0} \mathcal{F}_{x_1,y} T_3 \big)(\xi_0, \xi_1,\eta) := \frac{\widehat{V}(\xi_1 - \xi_0)}{\vert \xi_1 + \eta \vert^2 - \vert \eta \vert ^2}.
\end{align*}
Now with the notations introduced above we can define
\begin{align*}
G_{n-1}^{i_1,...,i_{n-1}} := T_{i_1} \circledast T_{i_2} \circledast ... \circledast T_{i_{n-1}} .
\end{align*}
\begin{remark}
Note that in what follows only $i_{n-1}$ will possibly be equal to 3. This means that $i_j \in \lbrace 1;2 \rbrace $ if $j \leqslant n-2$, and $i_{n-1} \in \lbrace 1;2;3 \rbrace. $ \\
\end{remark}
This definition implies that
\begin{align*}
G_{n-1} ^{i_1,...,i_{n-1}} (x_0,x,y) & = \mathcal{F}_{\xi_0} \mathcal{F}^{-1}_{\xi_{n-1},\eta} \int_{\big(\mathbb{R}^3 \big)^{n-2}} \prod_{\gamma=1}^{n-1} \frac{\widehat{V'}(\xi_{\gamma} - \xi_{\gamma-1}) P_{k_{\gamma}}(\xi_{\gamma}+\eta) P_{k}(\eta)}{\vert \xi_{\gamma} +\eta \vert^2 - \vert \eta \vert^2} d\xi_{1} ... d\xi_{n-2},
\end{align*}
with the abuses of notations mentioned at the beginning of this section, namely the $P_{k_{\gamma}}$ localizer may in fact be $P_{\leqslant k-2},$ and in the last kernel the localizers may be absent (case where $i_{n-1} = 3 $). \\
We are now in a position to define the kernels that appear in our work, namely
\begin{align*}
\mathcal{F}_{x,y} \big( K_{n-1}^{i_1,...,i_{n-1}} \big)(\xi_{n-1},\eta) = \mathcal{F}_{\xi_0}^{-1} \mathcal{F}_{x,y} \big(G_{n-1} ^{i_1,...,i_{n-1}}  \big) (0,\xi_{n-1},\eta).
\end{align*}
This means that in physical space
\begin{align*}
K_{n-1} ^{i_1,...,i_{n-1}} (x,y) &= \mathcal{F}^{-1}_{\xi_{n-1},\eta} \int_{\big(\mathbb{R}^{3}\big)^{n-2}} \prod_{\gamma=1}^{n-1} \frac{\widehat{V'}(\xi_{\gamma} - \xi_{\gamma-1}) P_{k_{\gamma}}(\xi_{\gamma}+\eta) P_{k}(\eta)}{\vert \xi_{\gamma} +\eta \vert^2 - \vert \eta \vert^2} d\xi_{1} ... d\xi_{n-2} 
\end{align*}
where by convention $\xi_0=0.$ This corresponds to the kernels that appear in the series.
\begin{remark}
In the sequel we will often drop the $i_1,...,i_{n-1}$ indices for better legibility. 
\end{remark}
\subsubsection{Three useful formulas}
We record three useful representations of the kernels. They give inductive relations between the kernels $K_n.$ This subsequently reduces the estimation of $K_n$ to the estimation of $K_1$ in the following subsection 6.2.
\begin{lemma} \label{kernel1}
We have 
\begin{align} 
K_{n-1}(x,y) = \int_{\mathbb{R}^6} K_{n-2}(x',y') T_{i_{n-1}} (x',x,y-y')~ dx' dy'.
\end{align}
\end{lemma}
\begin{proof}
By definition
\begin{eqnarray*}
\mathcal{F}_{x,y} K_{n-1} (\xi_{n-1},\eta) = \mathcal{F}^{-1}_{x_0} \mathcal{F}_{x,y} G_{n-1}^{i_1,...,i_{n-1}}  (0,\xi_{n-1},\eta).
\end{eqnarray*}
Therefore since 
\begin{align*}
G_{n-1} &= \bigg( T_{i_1} \circledast T_{i_2} \circledast ...\bigg) \circledast T_{i_{n-1}} \\
        &= G_{n-2} \circledast T_{i_{n-1}}
\end{align*}
we can write
\begin{align*}
\notag K_{n-1}(x,y) &= \int_{\mathbb{R}^3} G_{n-1}^{i_1,...,i_{n-1}} (x_0,x,y) dx_0 \\
\notag            &= \int_{\mathbb{R}^9} G_{n-2}^{i_1,...,i_{n-2}} (x_0,x',y') T_{i_{n-1}} (x',x,y-y')~ dx_0 dx' dy' \\
          &= \int_{\mathbb{R}^6} K_{n-2}(x',y') T_{i_{n-1}} (x',x,y-y')~ dx' dy'.
\end{align*}
\end{proof}
Now we prove a formula that is useful when proving bounds on the iterates:
\begin{lemma} \label{kernel2}
Let $f$ be a Schwartz function. \\
We have the following:
\begin{align*} 
\int_{\mathbb{R}^3} f(x_0) T_1(x_0,x,y) ~dx_0 &=\mathcal{F}^{-1}_{\xi_1,\eta} \frac{\widehat{fV}(\xi_1) P_k (\eta) P_{k_1} (\xi_1+\eta)}{\vert \xi_1 + \eta \vert^2 - \vert \eta \vert^2}, \\
\int_{\mathbb{R}^3} f(x_0) T_2(x_0,x,y) ~dx_0 &=\mathcal{F}^{-1}_{\xi_1,\eta} \frac{\widehat{fV}(\xi_1) P_k (\eta) P_{\leqslant k-2} (\xi_1+\eta)}{\vert \xi_1 + \eta \vert^2 - \vert \eta \vert^2}, \\
\int_{\mathbb{R}^3} f(x_0) T_3(x_0,x,y) ~dx_0 &=\mathcal{F}^{-1}_{\xi_1,\eta} \frac{\widehat{fV}(\xi_1)}{\vert \xi_1 + \eta \vert^2 - \vert \eta \vert^2}.
\end{align*}
\end{lemma}
\begin{proof}
Indeed we can write
\begin{align*}
\mathcal{F}_{x,y} \int_{\mathbb{R}^3} f(x_0) T_1(x_0,x,y) dx_0 &= \int_{\mathbb{R}^3} \widehat{f}(\xi_0) \mathcal{F}^{-1}_{x_0} \mathcal{F}_{x,y} T_1 (\xi_0,\xi_1,\eta) d\xi_0 \\
& = \int_{\mathbb{R}^3} \frac{\widehat{f}(\xi_0) \widehat{V}(\xi_1-\xi_0) P_k (\eta) P_{k_1} (\xi_1+\eta)}{\vert \xi_1+\eta \vert^2 - \vert \eta \vert^2} d\xi_0 \\
& = \frac{\widehat{fV}(\xi_1) P_k (\eta) P_{k_1} (\xi_1+\eta)}{\vert \xi_1 + \eta \vert^2 - \vert \eta \vert^2}
\end{align*}
and similarly for $T_2, T_3.$
\end{proof}
Finally as a consequence of the two previous formulas we have the following 
\begin{lemma} \label{kernel3}
Let $x,y',y'' \in \mathbb{R}^3.$ Assume that $x'' \mapsto K_{n-2} (x'',y'')$ is a Schwartz function in $x''.$ \\
 Then
\begin{align*} 
&\int_{\mathbb{R}^3} K_{n-2} (x'',y'') T_3(x'',x,y'-y'') dx'' = \\
\notag & \frac{1}{\vert y'-y'' \vert^2} L_{V(\cdot) K_{n-2}(\cdot,y'')} (\vert y'-y'' \vert - 2 \widehat{(y'-y'')} \cdot x, \widehat{(y'-y'')}) .
\end{align*}
If $ k_{n-1} - k >1$ then
\begin{align*} 
&\int_{\mathbb{R}^3} K_{n-2} (x'',y'') T_1(x'',x,y'-y'') dx'' = \\
\notag & \int_{ \mathbb{R}^3 } \big(V(\cdot) K_{n-2}(\cdot,y'')\big)_{\leqslant k_{n-1} +10} \check{m}(x-z,y) dz
\end{align*}
where we denoted
\begin{align*}
m(\xi_{n-1},\eta) = \frac{P_k (\eta) P_{k_{n-1}} (\xi_{n-1}+\eta)}{\vert \xi_{n-1}+\eta \vert^2 - \vert \eta \vert^2}.
\end{align*}
\end{lemma}
\begin{proof}
For the first formula we use Lemma \ref{kernel2} to write that
\begin{align*}
\int_{\mathbb{R}^3} K_{n-2} (x'',y'') T_3(x'',x,y'-y'') dx'' &= \mathcal{F}^{-1}_{\xi_{n-1},\eta} \frac{\mathcal{F} \big(V(\cdot) K_{n-2}(\cdot,y'')\big)(\xi_{n-1})}{\vert \xi_{n-1} + \eta \vert^2 - \vert \eta \vert^2}
\end{align*}
and then use Lemma \ref{BSmain1} to obtain the result.
\\
\\
For the second formula we start similarly and write that
\begin{align*}
&\int_{\mathbb{R}^3} K_{n-2} (x'',y'') T_3(x'',x,y'-y'') dx'' \\
&= \mathcal{F}^{-1}_{x,y} \frac{\mathcal{F} \big(V(\cdot) K_{n-2}(\cdot,y'')\big)(\xi_{n-1}) P_k (\eta) P_{k_{n-1}} (\xi_{n-1}+\eta)}{\vert \xi_{n-1} + \eta \vert^2 - \vert \eta \vert^2} \\
&= \mathcal{F}^{-1}_{x,y} \frac{\mathcal{F} \big(V(\cdot) K_{n-2}(\cdot,y'')\big)(\xi_{n-1}) P_{\leqslant k_{n-1} +10}(\xi_{n-1}) P_k (\eta) P_{k_{n-1}} (\xi_{n-1}+\eta)}{\vert \xi_{n-1} + \eta \vert^2 - \vert \eta \vert^2} \\
&= \int_{\mathbb{R}^3} \big(V(\cdot) K_{n-2}(\cdot,y'')\big)_{\leqslant k_{n-1} +10} \check{m}(x-z,y) dz .
\end{align*}
\end{proof}
\begin{remark}
In practice when writing our estimates we can remove the Schwartz assumption in the previous lemma by a standard approximation argument. 
\end{remark}
\subsection{Useful lemmas for the iteration}
The recursive formulas of the previous subsection reduced the problem of estimating the $n-$th kernel to estimating the first kernel. In this subsection we therefore prove bounds on $K_1,$ and deduce the corresponding estimates for the $n-$th kernel.

We start by recalling a useful result of Beceanu and Schlag: (here we actually state a weaker version of their result: we use our $B_x$ norm to write the estimates and not the $\dot{B}^{1/2}$ norm. See \cite{BSmain} for the original statement)
\begin{lemma}[see Lemma 6.2, \cite{BSmain}] \label{itgiven}
Let $v$ be a Schwartz function. \\
Then 
\begin{align*}
\Vert v(x) K_1 ^3 (x,y) \Vert_{L^1 _y B_x} & \lesssim \Vert \langle x \rangle v \Vert_{B_x} \Vert V \Vert_{B_x} .
\end{align*}
Let $f$ be  Schwartz function. Let 
\begin{align*}
\widetilde{K_1 ^3}(x,y) = \int_{\mathbb{R}^3} f(x_0) T_3 (x_0,x,y) dx_0 .
\end{align*}
Then 
\begin{align*}
\Vert \widetilde{K_1 ^3}(x,y) \Vert_{L^{\infty}_x L^1 _y } & \lesssim \Vert f V \Vert_{B_x} \\
\Vert v(x) \widetilde{K_1 ^3}(x,y) \Vert_{L^1 _y B_x } & \lesssim \Vert \langle x \rangle v(x) \Vert_{B_x} \Vert f V \Vert_{L^1 _y B_x} .
\end{align*}
\end{lemma}
We need to prove a similar statement for $K_1 ^1$ and $K_1 ^2 .$ We cannot directly use the result above due to the frequency localizations present for example in $T_1$. The remainder of this subsection is devoted to proving similar results for our kernels. \\
We begin with the easier case:
\begin{lemma} \label{itfacile}
Let $v$ be a Schwartz function. \\
With the notations introduced in the previous subsection, we have when $k_1 > k+1:$
\begin{align*}
\Vert v(x) K_1 ^1 (x,y) \Vert_{L^1 _y B_x} \lesssim 1.1^{\theta k_1} \Vert \langle x \rangle v \Vert_{B_x} \Vert V \Vert_{B_x} 
\end{align*}
for any $\theta \in [-1/2;1] .$ \\
We also have
\begin{align*}
\Vert v(x) K_1 ^2 (x,y) \Vert_{L^1 _y B_x} \lesssim \Vert \langle x \rangle v \Vert_{B_x} \Vert V \Vert_{B_x} .
\end{align*}
\end{lemma}
\begin{proof}
Let's treat the first case:
\\
Write that 
\begin{align*}
\Vert v(x) K_1 ^1 (x,y) \Vert_{L^1 _y B_x} & \leqslant \Vert v \Vert_{B_x} \Vert K_1  ^1 (x,y) \Vert_{L^1 _y L^{\infty} _x} 
\end{align*}
and by Lemma \ref{kernel3}
\begin{align*}
K_1 ^1 (x,y) &= \int_{\mathbb{R}^3} V_{\leqslant k_1 +10} (z) \check{m}(x-z,y) ~dz
\end{align*}
where
\begin{align*}
m(\xi,\eta) =  \frac{P_k(\eta) P_{k_1} (\xi+\eta)}{\vert \xi + \eta \vert^2 - \vert \eta \vert^2} .
\end{align*}
Using Lemma \ref{symbolbis} and Bernstein's inequality:
\begin{align*}
\Vert K_1 ^1 (x,y) \Vert_{L^1 _y L^{\infty}_x} & \lesssim 1.1^{-2k_1} \Vert V_{\leqslant k_1+10} \Vert_{L^{\infty}_x} \\
                                            & \lesssim 1.1^{\theta k_1} \Vert V \Vert_{B_x} .
\end{align*}
The second case is treated similarly, therefore we omit the proof.
\end{proof}
Now we treat the more difficult case, that is when the denominator is singular. The proof uses the same ideas as Lemma 6.2 of \cite{BSmain}, with some extra technicalities. As stated above, the main difference is the presence of frequency localizations, which prevent us from using the result directly.
\begin{lemma} \label{itdifficile}
Let $v$ be a Schwartz function. \\
We have, when $\vert k-k_1 \vert \leqslant 1:$
\begin{align*}
\Vert v(x) K_1 ^1(x,y) \Vert_{L^1 _y B_x} \lesssim \Vert \langle x \rangle v \Vert_{B_x} \Vert V \Vert_{B_x} .
\end{align*}
\end{lemma}
\begin{proof}
We have
\begin{align*}
K_1 ^1 (x,y) &= \mathcal{F}^{-1} _{\xi,\eta} \bigg( \frac{\widehat{V}(\xi) P_k(\eta) P_{k_1} (\xi+\eta)}{\vert \xi + \eta \vert^2 - \vert \eta \vert^2+i0} \bigg) \\
&=\mathcal{F}^{-1} _{\xi,\eta} \bigg( \frac{\widehat{V}(\xi) P_{\leqslant k+10} (\xi) P_k(\eta) P_{k_1} (\xi+\eta)}{\vert \xi + \eta \vert^2 - \vert \eta \vert^2+i0} \bigg) \\
& = \mathcal{F}^{-1} _{\xi,\eta} \bigg( \frac{\widehat{V_{\leqslant k+10}}(\xi)}{\vert \xi + \eta \vert^2 - \vert \eta \vert^2+i0} \bigg) \star \mathcal{F}^{-1} _{\xi,\eta} \big(P_k(\eta) P_{k_1}(\xi+\eta) \big) \\
& = \bigg(\frac{1}{\vert y \vert^2} L_{V_{\leqslant k+10}}(\vert y \vert - 2 \hat{y} \cdot x ,y) \bigg) \star \big( \phi_{k_1} (x) \phi_{k} (y-x) \big) \\
& = \int_{ \mathbb{R}^6 } \frac{1}{\vert z \vert^2} L_{V_{\leqslant k+10}}(\vert z \vert - 2 \hat{z} \cdot (x-u) ,z) \phi_{k_1} (u) \phi_k (y-z-u) ~dzdu
\end{align*}
where the first Fourier transform comes from Lemma \ref{BSmain1}. \\
In the second Fourier transform we denoted
\begin{align*}
\phi_k (x) &= \mathcal{F}^{-1} (P_k(\xi)) = 1.1^{3k} \check{\phi}(1.1^k x)
\end{align*}
where $\phi$ is the function introduced in the notation section of the paper. \\
\\
We must estimate
\begin{align*}
&\int_{\mathbb{R}^3}  \Bigg \Vert v(x) \int_{\mathbb{R}^6} \frac{1}{\vert z \vert^2} L_{V_{\leqslant k+10}}(\vert z \vert - 2 \hat{z} \cdot (x-u) ,z) \phi_{k} (u) \phi_{k_1} (y-z-u) ~dzdu \Bigg \Vert_{B_x} dy \\
& \lesssim \int_{\mathbb{R}^3} \Bigg \Vert v(x) \textbf{1}_{\vert x \vert \lesssim 1} \int_{\mathbb{R}^6 } \frac{1}{\vert z \vert^2} L_{V_{\leqslant k+10}}(\vert z \vert - 2 \hat{z} \cdot (x-u) ,\hat{z}) \phi_k (u) \phi_{k_1} (y-z-u) ~dzdu \Bigg \Vert_{L^2 _x}  dy \\
&+ \int_{\mathbb{R}^3} \sum_{p=1} ^{\infty} 1.1^{2p} \Bigg \Vert v(x) \textbf{1}_{\vert x \vert \sim 1.1^p} \int_{\mathbb{R}^6} \frac{1}{\vert z \vert^2} L_{V_{\leqslant k+10}}(\vert z \vert - 2 \hat{z} \cdot (x-u) ,z) \phi_k (u) \phi_{k_1} (y-z-u) ~dzdu \Bigg \Vert_{L^2 _x} dy \\
&:= I+II .
\end{align*}
First we consider $I$: using Minkowski's inequality and switching to polar coordinates in the $z$ variable, we write that
\begin{align*}
I \leqslant & \int_{\mathbb{R}^3} \Bigg \Vert v(x) \textbf{1}_{\vert x \vert \lesssim 1} \int_{\mathbb{R}^6 } \frac{1}{\vert z \vert^2} L_{V_{\leqslant k+10}}(\vert z \vert - 2 \hat{z} \cdot (x-u) ,\hat{z}) \phi_k (u) \phi_{k_1} (y-z-u) ~dzdu \Bigg \Vert_{L^2 _x}  dy  \\
\leqslant & \int_{\mathbb{R}^9} \Bigg \Vert v(x) \textbf{1}_{\vert x \vert \lesssim 1} \frac{1}{\vert z \vert^2} L_{V_{\leqslant k+10}}(\vert z \vert - 2 \hat{z} \cdot (x-u) ,\hat{z}) \phi_k (u) \phi_{k_1} (y-z-u) \Bigg \Vert_{L^2_x} dy dzdu \\
\leqslant & \int_{\mathbb{R}^3}  \vert \phi_{k_1} (y) \vert dy  \int_{ \mathbb{R}^3} \vert \phi_k (u) \vert \int_{\mathbb{S}^2} \int_{0}^{+\infty} \Vert v(x) \textbf{1}_{\vert x \vert \lesssim 1} L_{V_{\leqslant k+10}}(r-2 \omega \cdot (x-u),\omega) \Vert_{L^2 _x} dr d\omega du .
\end{align*}
Then we split this term depending on whether $r$ or $x-u$ dominates in the argument of $L,$ and use the Cauchy-Schwarz inequality:
\begin{align*}
I & \lesssim   \sum_{q=0} ^{\infty} \int_{\vert u \vert \sim 1.1^q} \vert \phi_k (u) \vert \sum_{l \geqslant q+10} \int_{\mathbb{S}^2} \int_{r \sim 1.1^l} \Vert v(x) \textbf{1}_{\vert x \vert \lesssim 1} L_{V_{\leqslant k+10}}(r-2 \omega \cdot (x-u),\omega) \Vert_{L^2 _x} dr d\omega du \\
&+  \sum_{q=0} ^{\infty} \int_{\vert u \vert \sim 1.1^q} \vert \phi_k (u) \vert \int_{\mathbb{S}^2} \int_{r \lesssim 1.1^q} \Vert v(x) \textbf{1}_{\vert x \vert \lesssim 1} L_{V_{\leqslant k+10}}(r-2 \omega \cdot (x-u),\omega) \Vert_{L^2 _x} dr d\omega du  \\
& \lesssim  \sum_{q=0} ^{\infty} \int_{\vert u \vert \sim 1.1^q} \vert \phi_k (u) \vert \sum_{l \geqslant q+10} 1.1^{l/2} \Bigg( \int_{\mathbb{S}^2} \int_{r \sim 1.1^l} \Vert v(x) \textbf{1}_{\vert x \vert \lesssim 1} L_{V_{\leqslant k+10}}(r-2 \omega \cdot (x-u),\omega) \Vert_{L^2 _x} ^2 dr d\omega \Bigg)^{1/2} du \\
& + \sum_{q=0} ^{\infty} \int_{\vert u \vert \sim 1.1^q} \vert \phi _k (u) \vert 1.1^{q/2} \Bigg( \int_{\mathbb{S}^2} \int_{r \lesssim 1.1^q} \Vert v(x) \textbf{1}_{\vert x \vert \lesssim 1} L_{V_{\leqslant k+10}}(r-2 \omega \cdot (x-u),\omega) \Vert_{L^2 _x} ^2 dr d\omega \Bigg)^{1/2} du \\
&  \lesssim  \Vert v(x) \textbf{1}_{\vert x \vert \lesssim 1} \Vert_{L^2_x} \int_{\mathbb{R}^3 } \vert \phi_k (u) \vert du \sum_{l \geqslant 0} 1.1^{l/2} \Bigg( \int_{\mathbb{S}^2} \int_{r \sim 1.1^l} \vert L_{V_{\leqslant k+10}}(r,\omega) \vert^2 dr d\omega \Bigg)^{1/2} \\
 &+ \Vert v(x) \textbf{1}_{\vert x \vert \lesssim 1} \Vert_{L^2_x} \Bigg( \int_{\mathbb{S}^2} \int_{0}^{\infty} \vert L_{V_{\leqslant k+10}}(r,\omega) \vert^2 dr d\omega \Bigg)^{1/2} \sum_{q=0}^{\infty} 1.1^{q/2} \int_{u \sim 1.1^q} \vert \phi _k (u) \vert du 
\end{align*}
Now notice that
\begin{align*}
\Vert V_{\leqslant k+10} \Vert_{\dot{B}^{1/2}} \leqslant \Vert V \Vert_{B_x} .
\end{align*}
Indeed using the Cauchy-Schwarz inequality in the sum that defines the $\dot{B}^{1/2}-$norm we have 
\begin{align*}
\Vert V_{\leqslant k+10} \Vert_{\dot{B}^{1/2}} & \lesssim \Vert \partial_{\xi_j} \big( \widehat{V}(\xi) P_{\leqslant k+10} \big) \Vert_{L^2} \\
& \lesssim \Vert \partial_{\xi_j} ( \widehat{V}) P_{\leqslant k+10} \Vert_{L^2} + 1.1^{-k} \Bigg \Vert \widehat{V}(\xi)P_{\leqslant k+50}(\xi) \frac{\xi_j}{\vert \xi \vert} \phi'(1.1^{-k} \vert \xi \vert) \Bigg \Vert_{L^2}.
\end{align*}
To bound the second term notice that $\vert \xi \vert \lesssim 1.1^{k} \Rightarrow 1.1^{-k} \lesssim \vert \xi \vert^{-1}.$ \\
Therefore
\begin{align*}
 1.1^{-k} \Bigg \Vert \widehat{V}(\xi)P_{\leqslant k+50}(\xi) \frac{\xi_j}{\vert \xi \vert} \phi'(1.1^{-k} \vert \xi \vert) \Bigg \Vert_{L^2} & \lesssim  \Bigg \Vert \frac{\widehat{V}(\xi)}{\vert \xi \vert}P_{\leqslant k+50}(\xi) \frac{\xi_j}{\vert \xi \vert} \phi'(1.1^{-k} \vert \xi \vert) \Bigg \Vert_{L^2} \\
 & \lesssim \bigg \Vert \frac{\widehat{V}(\xi)}{\vert \xi \vert } \Bigg \Vert_{L^2_x} \\
 & \lesssim \Vert V \Vert_{B_x}
\end{align*}
where to write the last line we used Hardy's inequality. \\
In conclusion
\begin{align*}
\Vert V_{\leqslant k+10} \Vert_{\dot{B}^{1/2}} \lesssim \Vert V \Vert_{B_x}.
\end{align*}
Now we can use this fact together with the two estimates from Lemma \ref{BSmain2} to write that
\begin{align*}
\sum_{l =0}^{\infty} 1.1^{l/2} \bigg( \int_{r \sim 1.1^l} \vert L_{V_{\leqslant k+10}}(r,\omega) \vert^2 dr d\omega \bigg)^{1/2} & \lesssim \Vert V \Vert_{B_x} 
\end{align*}
and we find that 
\begin{align*}
I \lesssim \Vert V \Vert_{B_x}  \Vert v(x) \textbf{1}_{\vert x \vert \lesssim 1} \Vert_{L^2_x}  + \Bigg( \sum_{q=0}^{\infty} 1.1^{q/2} \int_{\vert u \vert \sim 1.1^q} \phi_k (u) du \Bigg) \Vert V_{\leqslant k +10} \Vert_{L^2} \Vert v(x) \textbf{1}_{\vert x \vert \lesssim 1} \Vert_{L^2_x} .
\end{align*}
But since $\check{\phi}$ has fast decay, we have the following pointwise bound
\begin{align*}
\vert \phi_k (u) \vert \lesssim \frac{1.1^{3k}}{\langle 1.1^{k} \vert u \vert \rangle^7} .
\end{align*}
Therefore 
\begin{align*}
I \lesssim \Vert V \Vert_{B_x}  \Vert v(x) \textbf{1}_{\vert x \vert \lesssim 1} \Vert_{L^2_x}  + 1.1^{-k/2} \Vert V_{\leqslant k+10} \Vert_{L^2} \Vert v(x) \textbf{1}_{\vert x \vert \lesssim 1} \Vert_{L^2_x} .
\end{align*}
Now notice that 
\begin{eqnarray*}
1.1^{-k/2} \Vert \widehat{V} P_{\leqslant  k+10}  \Vert_{L^2} \lesssim \bigg \Vert \frac{\widehat{V}(\xi)}{\vert \xi \vert^{1/2}} \bigg \Vert_{L^2}
\end{eqnarray*}
since $\vert \xi \vert \leqslant 1.1^k \Rightarrow 1.1^{-k/2} \leqslant \vert \xi \vert^{-1/2}$,
\\
and now we can use Hardy's inequality to write that 
\begin{align*}
I \lesssim \Vert V \Vert_{B_x}  \Vert v(x) \textbf{1}_{\vert x \vert \lesssim 1} \Vert_{L^2_x}. 
\end{align*}
\bigskip
For $II$ we use a similar reasoning: \\
The terms in the sum over $p$ are bounded using Minkowski's inequality:
\begin{align*}
II &  \lesssim \sum_{p=1}^{\infty} \int_{ \mathbb{R}^9} 1.1^{2p} \Bigg \Vert v(x) \textbf{1}_{\vert x \vert \sim 1.1^p} \frac{1}{\vert z \vert^2} L_{V_{\leqslant k+10}}(\vert z \vert - 2 \hat{z} \cdot (x-u) ,\hat{z}) \phi_k (u) \phi_{k_1} (y-z-u) \Bigg \Vert_{L^2_x} dy dzdu \\
&= \sum_{p=1}^{\infty} \int_{ \mathbb{R}^3 } \vert \phi _k (y)  \vert \int_{ \mathbb{R}^3 } \vert \phi _k (u) \vert \int_{\mathbb{R}^3 } 1.1^{2p} \Bigg \Vert v(x) \textbf{1}_{\vert x \vert \sim 1.1^p} \frac{1}{\vert z \vert^2} L_{V_{\leqslant k+10}}(\vert z \vert - 2 \tilde{z} \cdot (x-u) ,\tilde{z}) \Bigg \Vert_{L^2_x} dz du dy \\
&= \sum_{p=1}^{\infty} \int_{\mathbb{R}^3 } \vert \phi_k (y) \vert dy \int_{ \mathbb{R}^3 } \vert \phi_k (u) \vert \int_{\mathbb{S}^2} \int_{ 0}^{ +\infty} 1.1^{2p} \Vert v(x) \textbf{1}_{\vert x \vert \sim 1.1^p} L_{V_{\leqslant k+10}}(r-2 \omega \cdot (x-u),\omega) \Vert_{L^2 _x} dr d\omega du.
\end{align*}
Now we split dyadically the variables $u$ and $r$ and use Cauchy-Schwarz:
\begin{align*}
&\sum_{q=1}^{\infty} \int_{ \vert u \vert \sim 1.1^q} \sum_{h \geqslant 10 + \max \lbrace p,q \rbrace}  \vert \phi_k (u) \vert \int_{ \mathbb{S}^2 } \int_{r \sim 1.1^h} 1.1^{2p} \Vert v(x) \textbf{1}_{\vert x \vert \sim 1.1^p} L_{V_{\leqslant k+10}}(r-2 \omega \cdot (x-u),\omega) \Vert_{L^2 _x} dr d\omega du \\
&+ \sum_{q=1}^{\infty} \int_{ \vert u \vert \sim 1.1^q} \vert \phi_k (u) \vert \int_{ \mathbb{S}^2 } \int_{r \lesssim 1.1^{\max \lbrace p,q \rbrace}} 1.1^{2p} \Vert v(x) \textbf{1}_{\vert x \vert \sim 1.1^p} L_{V_{\leqslant k+10}}(r-2 \omega \cdot (x-u),\omega) \Vert_{L^2 _x} dr d\omega du \\
 \leqslant & \sum_{q=1}^{\infty} \int_{ \vert u \vert \sim 1.1^q} \vert \phi_k (u) \vert \sum_{h \geqslant 10 + \max \lbrace p,q \rbrace} 1.1^{h/2} 1.1^{2p} \\
& \times \Bigg( \int_{\mathbb{S}^2 } \int_{r \sim 1.1^h} \bigg \Vert v(x) \textbf{1}_{\vert x \vert \sim 1.1^p} L_{V_{\leqslant k+10}}(r-2 \omega \cdot (x-u),\omega) \bigg \Vert ^2 _{L^2 _x} dr d\omega \Bigg)^{1/2} du \\
&+ \sum_{q=1}^{\infty} \int_{ \vert u \vert \sim 1.1^q} \vert \phi_k (u) \vert 1.1^{\max \lbrace p,q \rbrace /2 } 1.1^{2p} \\
& \times \Bigg( \int_{\mathbb{S}^2 } \int_{r \lesssim 1.1^{\max \lbrace p,q \rbrace}}  \bigg \Vert v(x) \textbf{1}_{\vert x \vert \sim 1.1^p} L_{V_{\leqslant k+10}}(r-2 \omega \cdot (x-u),\omega) \bigg \Vert ^2 _{L^2 _x} dr d\omega \Bigg)^{1/2} du \\
\leqslant &  1.1^{2p} \Vert v(x) \textbf{1}_{\vert x \vert \sim 1.1^p} \Vert_{L^2 _x} \sum_{q=1}^{\infty} \int_{ \vert u \vert \sim 1.1^q} \vert \phi_k (u) \vert \sum_{h \geqslant 10 + \max \lbrace p,q \rbrace} 1.1^{h/2} \Bigg( \int_{\mathbb{S}^2 } \int_{r \sim 1.1^h} \big \vert L_{V_{\leqslant k+10}}(r,\omega) \big \vert^2 dr d\omega \Bigg)^{1/2} du \\
&+ 1.1^{2p} \Vert v(x) \textbf{1}_{\vert x \vert \sim 1.1^p} \Vert_{L^2 _x} \sum_{q=1}^{\infty} \int_{ \vert u \vert \sim 1.1^q} \vert \phi_k (u) \vert 1.1^{\max \lbrace p,q \rbrace /2 } \Bigg( \int_{\mathbb{S}^2 } \int_{r \lesssim 1.1^{\max \lbrace p,q \rbrace}} \big \vert L_{V_{\leqslant k+10}}(r,\omega) \big \vert^2 dr d\omega \Bigg)^{1/2} du \\
:=&  III + IV .
\end{align*}
We have
\begin{align*}
III \lesssim 1.1^{2p} \Vert v(x) \textbf{1}_{\vert x \vert \sim 1.1^p} \Vert_{L^2 _x} \Vert V \Vert_{B_x} 
\end{align*}
using Lemma \ref{BSmain2} (second inequality). \\
For the second term we also use Lemma \ref{BSmain2} (first inequality). We obtain that 
\begin{align*}
IV & \lesssim  1.1^{2p} \Vert v(x) \textbf{1}_{\vert x \vert \sim 1.1^p} \Vert_{L^2 _x} \sum_{q=1}^{\infty} \int_{ \vert u \vert \sim 1.1^q} \vert \phi_k (u) \vert 1.1^{\max \lbrace p,q \rbrace /2 } \Bigg( \int_{\mathbb{S}^2 } \int_{r \lesssim 1.1^{\max \lbrace p,q \rbrace}} \big \vert L_{V_{\leqslant k+10}}(r,\omega) \big \vert^2 dr d\omega \Bigg)^{1/2} du \\
& \lesssim \Vert V_{\leqslant k +10} \Vert_{L^2} 1.1^{2p} \Vert v(x) \textbf{1}_{\vert x \vert \sim 1.1^p} \Vert_{L^2 _x} \sum_{q=1}^{\infty} \int_{ \vert u \vert \sim 1.1^q} \vert \phi_k (u) \vert 1.1^{\max \lbrace p,q \rbrace /2 } du.
\end{align*}
Now we can conclude as we did in the first case: we use the properties of the localization function and Hardy's inequality to write that the term is bounded by a constant times $\Vert \langle x \rangle v(x) \Vert_{B_x} \Vert V(x) \Vert_{B_x},$ which is the result we wanted.
\end{proof}
\begin{remark}
Note that the estimate written here is far from optimal: the exact same reasoning would be possible with far less spatial decay required of the potential. However we did not strive for the best assumptions possible on the potential here.
\end{remark}
Now we have the following corollary which we deduce from the last two lemmas.
\begin{corollary} \label{key}
Let $v$ be a Schwartz function. \\
Let $J(n) = \lbrace j \in \lbrace 1; ... ;n \rbrace ; k_j - k >1 \rbrace.$ \\
There exists a constant $C_0 >0$ independent of $n$ such that for every $\theta_j \in [-1/2;1],$ we have
\begin{align*}
\Vert v(x) K_n (x,y) \Vert_{L^1 _y B_x} \leqslant  C_0 ^n \bigg( \prod_{j \in J(n)} 1.1^{\theta_j k_j} \bigg) \Vert \langle x \rangle v \Vert_{B_x} \delta^{n} .
\end{align*}
\end{corollary}
\begin{proof}
Let $C_0$ denote the smallest of the implicit contants in Lemmas \ref{itgiven}, \ref{itfacile} and \ref{itdifficile}.
\\
We prove the result by induction on $n.$ The base case has already been treated in Lemmas \ref{itgiven}, \ref{itfacile} and \ref{itdifficile}. For the inductive step, we assume that $i_n=1$ and $k_1 - k >1$ (the other cases are treated similarly). 

Using Lemma \ref{kernel1} and Minkowski's inequality we can write that
\begin{align*}
\int_{\mathbb{R}^3 } \Vert v(x) K_n (x,y) \Vert_{B_x} dy & \leqslant \int_{\mathbb{R}^6} \bigg \Vert v(x) \int_{\mathbb{R}^3} K_{n-1} (x',y') T_1 (x',x,y-y') dx' \bigg \Vert_{B_x} dy dy' \\
& \leqslant \int_{\mathbb{R}^6} \bigg \Vert v(x) \int_{\mathbb{R}^3} K_{n-1} (x',y') T_1 (x',x,y) dx' \bigg \Vert_{B_x} dy dy'.
\end{align*} 
Now we use Lemma \ref{kernel2} to write that
\begin{align*}
&\int_{\mathbb{R}^6} \bigg \Vert v(x) \int_{\mathbb{R}^3} K_{n-1} (x',y') T_1 (x',x,y) dx' \bigg \Vert_{B_x} dy dy' 
\\
&= \int_{\mathbb{R}^3} \bigg \Vert v(x) \mathcal{F}^{-1}_{\xi_n,\eta} \bigg( \frac{\mathcal{F}\big(V(\cdot) K_{n-1}(\cdot,y') \big)(\xi_n) P_k(\eta) P_{k_n}(\xi_n+\eta)}{\vert \xi_n + \eta \vert^2 - \vert \eta \vert^2} \bigg) \bigg \Vert_{L^1 _y B_x}  dy',
\end{align*}
and we can apply Lemma \ref{itfacile} to write that 
\begin{align*}
&  \bigg \Vert v(x) \mathcal{F}^{-1}_{\xi_n,\eta} \bigg( \frac{\mathcal{F}\big(V(\cdot) K_{n-1}(\cdot,y') \big)(\xi_n) P_k(\eta) P_{k_n}(\xi_n+\eta)}{\vert \xi_n + \eta \vert^2 - \vert \eta \vert^2} \bigg) \bigg \Vert_{L^1 _y B_x}  
\\
&\leqslant  C_0 1.1^{\theta_n k_n} \Vert \langle x \rangle v \Vert_{B_x} \Vert V(x) K_{n-1}(x,y') \Vert_{B_x} .
\end{align*}
After integrating in $y'$ we can conclude using the induction hypothesis.
\end{proof}

\subsection{Some generalizations of Lemma \ref{bilin}}
Here we prove an analog of Lemma \ref{bilin} adapted to the $n-$th iterate terms. This will be useful since to prove Proposition \ref{stepn:estimates}, the strategy will be to follow the same path as in Section \ref{firstit}. The added difficulty is that we must keep track of the dependence on the $B_x$ norm of the potential. The lemma proved in this section will play the same role in the proof of Proposition \ref{stepn:estimates} as Lemma \ref{bilin} in Section \ref{firstit}. 

We start with the case of the kernel $K_1$ in the next lemma, and then generalize to the kernel $K_n$ using results from the previous subsections 6.1 and 6.2.
\\

For the first iterate we show:
\begin{lemma}
Let $m(\xi_1 ,\xi_2,\eta)$ be a real-valued multiplier such that $\check{m} \in L^1.$ \\
Then, if $k_1 >k+1$ we have the bound
\begin{align*}
\Bigg \Vert \mathcal{F}^{-1} _{\eta} \int_{\mathbb{R}^3} \widehat{K_1 ^1}(\xi_1, \eta) \int_{\mathbb{R}^3}   
\widehat{f}(\xi_1-\xi_2) m(\xi_1 ,\xi_2,\eta) \widehat{g}(\xi_2+\eta) d\xi_2 d\xi_1 \Bigg \Vert_{L^{\alpha}}
\lesssim 1.1^{\theta k_1} \Vert \check{m} \Vert_{L^1} \Vert V \Vert_{B_x} \Vert f \Vert_{L^p} \Vert g \Vert_{L^q}
\end{align*}
where $1/p + 1/q =1/\alpha$ and $\theta \in [-1/2;1].$ \\
If $i_1 =1$ and $\vert k_1 - k \vert \leqslant 1$ or if $i_1 \in \lbrace 2;3 \rbrace$ then 
\begin{align*}
\Bigg \Vert \mathcal{F}^{-1} _{\eta} \int_{\mathbb{R}^3} \widehat{K_1 ^{i_1}}(\xi_1, \eta) \int_{\mathbb{R}^3}   
\widehat{f}(\xi_1-\xi_2) m(\xi_1 ,\xi_2,\eta) \widehat{g}(\xi_2+\eta) d\xi_2 d\xi_1 \Bigg \Vert_{L^{\alpha}}
\lesssim \Vert \check{m} \Vert_{L^1} \Vert V \Vert_{B_x} \Vert f \Vert_{L^p} \Vert g \Vert_{L^q} .
\end{align*}
\end{lemma}
\begin{proof}
First we use Plancherel's theorem and then use the computation of the inverse Fourier transform carried out in the proof of Lemma \ref{bilin}, to write that
\begin{align*}
& \mathcal{F}^{-1} _{\eta} \int_{\mathbb{R}^3} \widehat{K_1 ^{i_1}}(\xi_1, \eta) \overline{ \int_{\mathbb{R}^3}   
\overline{\widehat{f}}(\xi_1-\xi_2) m(\xi_1 ,\xi_2,\eta) \overline{\widehat{g}}(\xi_2+\eta) d\xi_2 } d\xi_1 \\
& = (2 \pi)^{3} \mathcal{F}^{-1} _{\eta} \int_{\mathbb{R}^3} [\mathcal{F}_{y} K_1 ^{i_1} ](x, \eta) \overline{\mathcal{F}^{-1} _{\xi_1} \int_{\mathbb{R}^3}   
\overline{\widehat{f}}(\xi_1-\xi_2) m(\xi_1 ,\xi_2,\eta) \overline{\widehat{g}}(\xi_2+\eta) d\xi_2 } dx \\
& = (2 \pi)^{6}\mathcal{F}^{-1} _{\eta} \int_{\mathbb{R}^3} [\mathcal{F}_{y} K_1 ^{i_1} ](x, \eta) \int_{\mathbb{R}^6} f(r-x) g(t+r-x) \check{m}(-r,-t,\eta) e^{i \eta (x-r-t)} drdtdx \\
& =(2\pi)^{3} \int_{\mathbb{R}^9} f(r-x) g(t+r-x) \int_{\mathbb{R}^3} e^{i \eta (y+x-r-t)} [\mathcal{F}_{y} K_1 ^{i_1} ](x, \eta) \check{m}(-r,-t,\eta) d\eta dx dr dt \\
& =(2 \pi)^{3} \int_{\mathbb{R}^9} f(r-x) g(t+r-x) \int_{\mathbb{R}^3} K_1 ^{i_1}(x,y') \check{m}(-r,-t,y+x-r-t-y') dy' dx dr dt.
\end{align*}

To write the last line we used the formula for the inverse Fourier transform of a product. 

Now we estimate the $L^{\alpha}$ norm of the expression above using duality: let $h \in L^{\alpha'} _x, $ where $\alpha' = \frac{\alpha}{\alpha-1}$ We consider the pairing of the expression above with $\bar{h}$ and obtain
\begin{align*}
& \int_{\mathbb{R}^9}  f(r-x) g(t+r-x) \int_{\mathbb{R}^3} K_1 ^{i_1}(x,y') \int_{\mathbb{R}^3} \bar{h}(-y) \check{m}(r,t,x-r-t-y-y') dy dy' dx dr dt \\
=& \int_{\mathbb{R}^9} \int_{\mathbb{R}^3}  f(r-x) g(t+r-x)  K_1 ^{i_1}(x,y') \big( \bar{h}(-\cdot) \star \check{m}(-r,-t,\cdot) \big)(x-r-t-y') dx dy' dr dt.
\end{align*}

At this point we split the discussion into different cases: \\
\\
\underline{Case 1: $i_1 = 1, k_1 - k>1$}\\
We bound the expression above using H\"{o}lder's inequality in $x:$ 
\begin{align*}
& \int_{\mathbb{R}^9} \int_{\mathbb{R}^3} f(r-x) g(t+r-x) K_{1} ^{1}(x,y') \big( \bar{h}(-\cdot) \star \check{m}(-r,-t,\cdot) \big)(x-r-t+y') dx dy' dr dt \\
& \leqslant \Vert f \Vert_{L^p} \Vert g \Vert_{L^q} \Vert h \Vert_{L^{\alpha'}} \Vert K_1 ^1 \Vert_{L^1_y L^{\infty}_x} \int_{\mathbb{R}^6} \Vert \check{m}(-r,-t,\cdot) \Vert_{L^1} dr dt \\
& \leqslant \Vert f \Vert_{L^p} \Vert g \Vert_{L^q} \Vert h \Vert_{L^{\alpha'}} \Vert K_1 ^1 \Vert_{L^1_y L^{\infty}_x} \Vert \check{m} \Vert_{L^1} .
\end{align*}
Then we use Lemma \ref{kernel3} to write that
\begin{align*}
K_1 ^{1} (x,y') 
           &=\int_{\mathbb{R}^3} V_{\leqslant k_1+10} (x-r) \check{m'}(r,y') dr
\end{align*} 
with 
\begin{align*}
m'(\xi_1,\eta) = \frac{P_k(\eta) P_{k_1} (\xi_1+\eta)}{\vert \xi_1+\eta \vert^2 - \vert \eta \vert^2}.
\end{align*}
Using Bernstein's inequality and Lemma \ref{symbolbis} we get 
\begin{align*}
\Vert K_{1} ^{1} \Vert_{L^1_y L^{\infty} _x} \leqslant \Vert V_{\leqslant k_1+10} \Vert_{L^{\infty}_x} \Vert \check{m'} \Vert_{L^1}
& \lesssim 1.1^{\theta k_1} \Vert V \Vert_{B_x}  .
\end{align*}
\underline{Case 2: $i_1=2$} \\
This case is handled similarly to case 1 therefore the proof is omitted. \\
\underline{Case 3: $i_1=3$ or $i_1 =1 $ and $\vert k - k_1 \vert \leqslant 1:$} \\
We start with $i_1 =3.$ The other case is deduced from this by placing the localizations on $m:$ we replace $m$ by $P_k(\eta) P_{k_1} (\xi_1+\eta) m(\xi_1,\xi_2,\eta) $ and use Lemma \ref{symbol}. \\
In this case we use the expression of $K_1 ^3(x,y')$ from Lemma \ref{BSmain1} and switch to polar coordinates in the $y'$ variable to find that
\begin{align*}
& \int_{\mathbb{R}^9} \int_{\mathbb{R}^3}  f(r-x) g(t+r-x) K_1 ^3(x,y') \big( \bar{h}(-\cdot) \star \check{m}(-r,-t,\cdot) \big)(x-r-t-y') dx dy' dr dt \\
& = \int_{\mathbb{R}^9}  f(r-x) g(t+r-x) \int_{\mathbb{S}^2} \int_{0}^{\infty} L(\rho-2 \omega \cdot x, \omega) \big( \bar{h}(-\cdot) \star \check{m}(-r,-t,\cdot) \big)(x-r-t-\rho \omega) d\rho d\omega dx dr dt \\
&= \int_{\mathbb{R}^9}  f(r-x) g(t+r-x) \int_{\mathbb{S}^2} \int_{0}^{\infty} \textbf{1}_{\rho>-2 \omega \cdot x}
 L(\rho,\omega) \\ & \times \big( \bar{h}(-\cdot) \star \check{m}(-r,-t,\cdot) \big)(x-2( \omega \cdot x) \omega -r-t-\rho \omega) d\rho d\omega dx dr dt \\
& = \int_{\mathbb{R}^6} \int_{\mathbb{S}^2} \int_0 ^{\infty} L(\rho, \omega) \int_{\mathbb{R}^3} \textbf{1}_{\rho>-2 \omega \cdot x}  f(r-x) g(t+r-x) \\
& \times \big( h(\cdot) \star \check{m}(-r,-t,\cdot) \big)(x-2( \omega \cdot x) \omega -r-t-\rho \omega)  dx d\rho d\omega dr dt .
\end{align*}
Now we can conclude by applying H\"{o}lder's inequality in $x$ and using Lemma \ref{BSmain2} to bound the quantity above by
\begin{align*}
\bigg( \int_{\mathbb{S}^2} \int_{0}^{\infty} \vert L(\rho, \omega) \vert d\rho d\omega \bigg) \Vert f \Vert_{L^p} \Vert g \Vert_{L^q} \Vert h \Vert_{L^{\alpha'}} \Vert \check{m} \Vert_{L^1} \lesssim \Vert V \Vert_{B_x} \Vert f \Vert_{L^p} \Vert g \Vert_{L^q} \Vert h \Vert_{L^{\alpha'}} \Vert \check{m} \Vert_{L^1} 
\end{align*}
which yields the desired result.
\end{proof}
We also have the following $n-$th iterate generalization: 
\begin{lemma} \label{nbilin}
Let $m(\xi_{n-1} ,\xi_n,\eta)$ be a real-valued multiplier such that $\check{m} \in L^1.$ \\
Let $J(n-1) = \big \lbrace i \in \lbrace 1; ... ;n-1 \rbrace ; k_i>k+1 \big \rbrace. $ \\
Then for any $\theta_j \in [-1/2;1]$ we have
\begin{align*}
&\Bigg \Vert \mathcal{F}^{-1} _{\eta} \int_{\mathbb{R}^3} \widehat{K_{n-1}^{i_1,...,i_{n-1}}}(\xi_{n-1}, \eta) \int_{\mathbb{R}^3}   
\widehat{f}(\xi_{n-1}-\xi_n) m(\xi_{n-1} ,\xi_n,\eta) \widehat{g}(\xi_n+\eta) d\xi_n d\xi_{n-1} \Bigg \Vert_{L^{\alpha}} \\
& \lesssim C_0 ^{n-2}  \bigg( \prod_{j \in J(n-1)} 1.1^{\theta_j k_j} \bigg) \Vert \check{m} \Vert_{L^1} \delta ^{n-1} \Vert f \Vert_{L^p} \Vert g \Vert_{L^q}
\end{align*}
where $1/\alpha = 1/p + 1/q.$ \\
The constant $C_0$ has been defined in Corollary \ref{key}. \\
Finally the implicit constant does not depend on $n$ here. 
\end{lemma}
\begin{proof}
Again there are several cases to consider:
\\
\underline{Case 1: $i_{n-1} = 1 $ and $k_{n-1} - k >1$} \\
We drop the $i_1,..., i_{n-1}$ indices to improve legibility. \\
We start the proof in the same way as the previous lemma:
\begin{align*}
& \mathcal{F}^{-1} _{\eta} \int_{\mathbb{R}^3} \widehat{K_{n-1}}(\xi_{n-1}, \eta) \int_{\mathbb{R}^3}   
\widehat{f}(\xi_{n-1}-\xi_n) m(\xi_{n-1} ,\xi_n,\eta) \widehat{g}(\xi_1+\eta) d\xi_n d\xi_{n-1} \\
& =(2\pi)^{3} \int_{\mathbb{R}^9} f(r-x) g(t+r-x) \int_{\mathbb{R}^3} K_{n-1}(x,y') \check{m}(-r,-t,y+x-r-t-y') dy' dx dr dt .
\end{align*}
Now we estimate the $L^{\alpha}$ norm of the expression above using duality: let $h \in L^{\alpha'} _x.$ We consider the pairing of the expression above with $\bar{h}$ and obtain
\begin{align*}
\int_{\mathbb{R}^9} f(r-x) g(t+r-x) \int_{\mathbb{R}^3} K_{n-1}(x,y') \int_{\mathbb{R}^3} \bar{h}(y) \check{m}(-r,-t,y+x-r-t-y') dy dy' dx dr dt
\end{align*}
which we can bound as in the previous lemma using H\"{o}lder's inequality in $x$ by
\begin{align*}
\Vert f \Vert_{L^p} \Vert g \Vert_{L^q} \Vert K_{n-1} \Vert_{L^1 _y L^{\infty}_x} \Vert h \Vert_{L^{\alpha '}} \Vert m \Vert_{L^1} .
\end{align*}
Using Lemma \ref{kernel3} we can write that
\begin{align*}
K_{n-1}(x,y') &= \int_{\mathbb{R}^6} \big( K _{n-2}(x-r,y'') V(x-r) \big)_{\leqslant k_{n-1}+10} \check{m'}(r,y'-y'') dr dy'',
\end{align*}
where
\begin{align*}
m'(\xi_{n-1},\eta) = \frac{P_k(\eta) P_{k_{n-1}}(\xi_{n-1} + \eta)}{\vert \xi_{n-1} + \eta \vert^2 - \vert \eta \vert^2}.
\end{align*}
Therefore we can integrate in $x$ and use Bernstein's inequality together with Lemma \ref{symbolbis} to obtain: 
\begin{align*}
\Vert K_{n-1} (x,y) \Vert_{L^1 _y L^{\infty}_x} & \lesssim \Vert \check{m'} \Vert_{L^1} \Vert \big( V(x) K _{n-2}(x,y) \big)_{\leqslant k_{n-1}+10} \Vert_{L^1 _y L^{\infty}_x} \\
                                              & \lesssim 1.1^{-2k_{n-1}} \Vert \big( V(x) K _{n-2}(x,y) \big)_{\leqslant k_{n-1}+10} \Vert_{L^1 _y L^{\infty}_x}  \\
                                              & \lesssim 1.1^{\theta_{n-1} k_{n-1}} \Vert V(x) K _{n-2}(x,y) \Vert_{L^1 _y B_x}.
\end{align*}
Now we can conclude using Corollary \ref{key}. \\
\underline{Case 2: $i_{n-1} = 2 $} \\
Similar to the previous case.
\\
\underline{Case 3: $i_{n-1}=1 $ and $\vert k-k_{n-1} \vert \leqslant 1$ or $i_{n-1}=3.$} \\
As in the previous lemma, it is enough to treat the case where $i_{n-1}=3.$ \\
In this case we use the first equality in Lemma \ref{kernel3} to write that
\begin{align*}
K_{n-1} (x,y') &= \int_{\mathbb{R}^3} \frac{1}{\vert y'-y'' \vert^2} L_{V(\cdot) K_{n-2}(\cdot,y'')} (\vert y'-y'' \vert - 2 \widehat{(y'-y'')} \cdot x, \widehat{(y'-y'')}) dy'' .
\end{align*}
Then we can write, switching to polar coordinates in the $y'$ variable, that
\begin{align*}
&\int_{\mathbb{R}^6} \frac{1}{\vert y'-y'' \vert^2} L_{V(\cdot) K_{n-2}(\cdot,y'')} (\vert y'-y'' \vert - 2 \widehat{(y'-y'')} \cdot x, \widehat{(y'-y'')}) \int_{\mathbb{R}^3} \bar{h}(y) \check{m}(-r,-t,y+x-r-t-y') dy dy' dy'' \\
& = \int_{\mathbb{R}^6} \frac{1}{\vert y' \vert^2} L_{V(\cdot) K_{n-2}(\cdot,y'')} (\vert y' \vert - 2 \widehat{y'} \cdot x, \widehat{y'}) \int_{\mathbb{R}^3} \bar{h}(y) \check{m}(-r,-t,y+x-r-t-y'-y'') dy dy' dy'' \\
& = \int_{\mathbb{R}^3} \int_{\mathbb{S}^2} \int_0 ^{\infty} L_{V(\cdot) K_{n-2}(\cdot,y'')}(\rho-2 \omega \cdot x, \omega) \int_{\mathbb{R}^3} \bar{h}(y) \check{m}(-r,-t,y+x-r-t-\rho \omega-y'') dy d\omega d\rho dy'' \\
& =  \int_{\mathbb{R}^3} \int_{\mathbb{S}^2} \int_0 ^{\infty} \textbf{1}_{\rho>-2\omega \cdot x} L_{V(\cdot) K_{n-2}(\cdot,y'')}(\rho, \omega) \int_{\mathbb{R}^3} \bar{h}(y) \check{m}(r,t,y+x-2(\omega \cdot x) \omega-r-t-\rho\omega-y'') dy d\omega d\rho dy''.
\end{align*}
Now we integrate the previous integral against $f(r-x) g(r+t-x) dx dr dt$ to obtain
\begin{align*}
&\Bigg \vert \int_{\mathbb{R}^9} \int_{\mathbb{S}^2} \int_{0}^{\infty} L_{V(\cdot) K_{n-2}(\cdot,y'')}(\rho,\omega) \int_{\mathbb{R}^3} f(r-x) g(r+t-x) \textbf{1}_{\rho>-2\omega \cdot x} \\
& \times \int_{\mathbb{R}^3} \bar{h}(y) \check{m}(-r,-t,y+x-2(\omega \cdot x) \omega-r-t-\rho\omega-y'') dy dx d\rho d\omega dr dt dy'' \Bigg \vert.
\end{align*}
Next we use H\"{o}lder's inequality in $x$ as well as Lemma \ref{BSmain2} to bound this integral by
\begin{align*}
&\Vert f \Vert_{L^p} \Vert g \Vert_{L^q} \Vert h \Vert_{L^{\alpha '}} \Vert \check{m} \Vert_{L^1} \int_{\mathbb{R}^3} \Vert L_{V(\cdot) K_{n-2}(\cdot,y'')} \Vert_{L^1_{\rho,\omega}} dy'' \\
& \lesssim \Vert f \Vert_{L^p} \Vert g \Vert_{L^q} \Vert h \Vert_{L^{\alpha'}} \Vert \check{m} \Vert_{L^1} \Vert V(\cdot)  K_{n-2}(\cdot,y'') \Vert_{L^1_{y''} B_x},
\end{align*}
and we can conclude by induction using Corollary \ref{key}.
\end{proof}
Now we are done with these preparations, we can start bounding the $n-$th iterates.

\section{Multilinear terms in $V$} \label{mainproppf}
\noindent In this section we prove Proposition \ref{stepn:estimates}. We have six terms to bound, and we separate them depending on whether they are linear in $f$ (so-called potential terms) or bilinear (bilinear terms). This means that in the first subsection we treat $I_1 ^n f, I_2 ^n f, I_3 ^n f, I_4 ^n f$ and in the second $I_5 ^n f, I_6 ^n f.$ The proofs largely follow the strategy of Section \ref{firstit}, with the added difficulty that we keep track of the dependence on $V$ in the estimates through the use of lemmas proved in Section \ref{multilinanalysis}.  

\subsection{Estimating the $n-th$ iterates of potential terms}

We start with a direct application of Lemma \ref{nbilin}: \\
We estimate the $n-$th iterate of \eqref{R3}. Note from our discussion in Section 4.2 that these terms are such that $\vert k_n - k \vert \leqslant 1.$ \\
We prove the following: 
\begin{lemma}
Let $J(n-1):= \lbrace i \in \lbrace 1;...;n-1 \rbrace , k < k_i-1 \rbrace.$ \\
Recall that
\begin{align*}
\mathcal{F} I_1 ^n f(t,\xi) &:= \displaystyle \int_1 ^t \int_{\big( \mathbb{R}^3 \big)^{n-1}} \prod_{\gamma=1}^{n-1} \frac{\widehat{V}(s, \eta_{\gamma-1}-\eta_{\gamma}) P_{k_{\gamma}}(\eta_{\gamma}) P_{k}(\xi)}{\vert \xi \vert^2 - \vert \eta_{\gamma} \vert^2} d\eta_{1} ... d\eta_{n-2} \\
& \times \int_{\mathbb{R}^3}  \frac{\xi_l \eta_{n,j}}{\vert \eta _n \vert^2} \partial_{\eta_{n,j}} \widehat{V}(s, \eta_{n-1}-\eta_n) e^{is(\vert \xi \vert^2 - \vert \eta_n \vert ^2)} \widehat{f_{k_n}}(s,\eta_n) d\eta_n d\eta_{n-1} ds .
\end{align*}
We have
\begin{align*}
\Vert I_1 ^n f \Vert_{L^{\infty}_t L^2 _x} \lesssim C_0 ^{n-2} \bigg( \prod_{j \in J(n-1)} \min \lbrace 1.1^{-k_j/2} ; 1.1^{k_j/2} \rbrace \bigg) \varepsilon_1 \delta^n ,
\end{align*} 
where by convention $\prod_{j \in \emptyset} =1.$ \\
The constant $C_0$ has been defined in Corollary \ref{key}. \\
Finally the implicit constant does not depend on $n.$  
\end{lemma}
\begin{remark}
The corresponding estimate in Proposition \ref{stepn:estimates} is a straightforward consequence of this lemma due to the converging factor $\prod_{j \in J(n-1)} \min \lbrace 1.1^{-k_j/2} ; 1.1^{k_j/2} \rbrace .$
\end{remark}
\begin{proof}
Using Strichartz estimates, we write that
\begin{align*}
& \Vert I_1 ^n f \Vert_{L^{\infty}_t L^2 _x} \\
 \lesssim & \Bigg \Vert \mathcal{F}_{\xi}^{-1}  \int_{\big(\mathbb{R}^3 \big)^{n-1}} \prod_{\gamma=1}^{n-1} \frac{\widehat{V}(t, \eta_{\gamma-1}-\eta_{\gamma}) P_{k_{\gamma}}(\eta_{\gamma}) P_{k}(\xi)}{\vert \xi \vert^2 - \vert \eta_{\gamma} \vert^2} d\eta_{1} ... d\eta_{n-2} \\
& \times \int_{\mathbb{R}^3}  \frac{\xi_l \eta_{n,j}}{\vert \eta _n \vert^2} \partial_{n,j} \widehat{V}(t,\eta_{n-1}-\eta_{n}) e^{-it\vert \eta_n \vert ^2} \widehat{f_{k_n}}(t,\eta_n) d\eta_n d\eta_{n-1} \Bigg \Vert_{L^{4/3}_t L^{3/2}_x}
\\
& \lesssim \Bigg \Vert \mathcal{F}^{-1}_{\eta} \int_{\big(\mathbb{R}^3 \big)^{n-1}} \prod_{\gamma=1}^{n-1} \frac{\widehat{V'}(t,\xi_{\gamma}-\xi_{\gamma-1}) P_{k_{\gamma}}(\xi_{\gamma}+\eta) P_{k}(\eta)}{\vert \xi_{\gamma} +\eta \vert^2 - \vert \eta \vert^2} d\xi_{1} ... d\xi_{n-2} \\
& \times \int_{\mathbb{R}^3} \frac{\eta_l (\xi_n+\eta)_j}{\vert \xi_n + \eta \vert^2} \partial_{n,j} \widehat{V}(t,\xi_{n-1}- \xi_n) e^{-it\vert \xi_n+\eta \vert ^2} \widehat{f_{k_n}}(t,\xi_n+\eta) d\xi_n d\xi_{n-1} \Bigg \Vert_{L^{4/3}_t L^{3/2} _x}
\end{align*}
where to get the last line we performed the usual change of variables (labelling $\xi$ as $\eta$ and $\eta_{\gamma}$ as $\xi_{\gamma}+\eta$). 
\\
Now we can apply Lemma \ref{nbilin} with
\begin{align*}
m(\xi_n, \xi_{n-1}, \eta) &= \frac{P_k (\eta) P_{k_n}(\xi_n + \eta) P_{k_{n-1}}(\xi_{n-1}+\eta) \eta_l (\xi_n+\eta)_j}{ \vert \xi_n + \eta \vert^2} \\
p = 6, q&=2, \alpha=3/2
\end{align*}
and obtain that
\begin{align*}
\Vert I_1 ^n f \Vert_{L^{\infty}_t L^2 _x} & \lesssim C_0^{n-2} \prod_{j \in J(n-1)} \min \lbrace 1.1^{-k_j/2} ; 1.1^{k_j/2} \rbrace \delta ^n \Vert e^{it \Delta} f_{k_n}(t) \Vert_{L^{4/3}_t L^6 _x} \\
                                           & \lesssim C_0^{n-2} \prod_{j \in J(n-1)} \min \lbrace 1.1^{-k_j/2} ; 1.1^{k_j/2} \rbrace \delta ^n \varepsilon_1 .
\end{align*}
\end{proof}
Now we move on to the second type of terms listed in Proposition \ref{stepn:estimates}, namely $I_2^n f.$ 
Recall from our earlier discussion in Section \ref{series-repr} that these terms are such that $\vert k - k_{n-1} \vert \leqslant 1.$ This is the object of the following lemma:
\begin{lemma}
Let $J(n-1):= \lbrace i \in \lbrace 1;...;n-1 \rbrace , k < k_i-1 \rbrace.$ \\
Let
\begin{align*}
\mathcal{F} I_2 ^n f(t,\xi) & := \int_1 ^t \int_{\big(\mathbb{R}^3 \big)^{n-1}} \prod_{l=1}^{n-1} \frac{\widehat{V}(s,\eta_{l-1}-\eta_l) P_{k_l}(\eta_l) P_{k}(\xi)}{\vert \xi \vert^2 - \vert \eta_l \vert^2} d\eta_{1} ... d\eta_{n-2} \frac{\xi_l \eta_{n-1,j}}{\vert \eta_{n-1,j} \vert^2} \\
& \times \int_{\mathbb{R}^3}  \widehat{x_j V}(s,\eta_{n-1}-\eta_n) e^{is(\vert \xi \vert^2- \vert \eta_n \vert^2)} \widehat{f}(s,\eta_n) d\eta_n d\eta_{n-1} ds.
\end{align*}
Then 
\begin{align*}
\Vert I_2 ^n f \Vert_{L^{\infty}_t L^2 _x} \lesssim C_0^{n-2} \prod_{j \in J(n-1)} \min \lbrace 1.1^{-k_j/2} ; 1.1^{k_j/2} \rbrace \varepsilon_1 \delta^n .
\end{align*}
The implicit constant does not depend on $n.$ 
\end{lemma}
\begin{remark}
As for the above lemma, the corresponding estimate in Proposition \ref{stepn:estimates} is then obtained as a direct corollary.
\end{remark}
\begin{proof}
We start by splitting the $\eta_n$ frequency dyadically. We denote $k_n$ the corresponding exponent.
\begin{align*}
\Vert I_2 ^n f \Vert_{L^{\infty}_t L^2 _x} & \leqslant \sum_{k_n \in \mathbb{Z}} \Bigg \Vert \int_1 ^t e^{is \vert \eta \vert^2} \int_{\xi_{n-1}} \widehat{K_{n-1}}(\xi_{n-1},\eta) \\
                                           & \times \int_{\xi_n} \widehat{x_j V}(s,\xi_n - \xi_{n-1}) m(\xi_{n-1}, \xi_n,\eta) e^{-is\vert \eta_n \vert^2} \widehat{f_{k_n}}(s,\xi_n)  d\xi_n d\xi_{n-1} ds \Bigg \Vert_{L^{\infty}_t L^2_x}
\end{align*}
where
\begin{align*}
m(\xi_{n-1},\xi_n,\eta) &= \frac{\eta_l P_k(\eta) (\xi_{n-1} + \eta)_j P_{k_{n-1}}(\xi_{n-1} + \eta)P_{k_n}(\xi_n) }{\vert \xi_{n-1} + \eta \vert^2} .
\end{align*}
We use Strichartz estimates to write that
\begin{align*}
\Vert I_2 ^n f \Vert_{L^{\infty}_t L^2 _x} & \lesssim \sum_{k_n \in \mathbb{Z}} \Bigg \Vert \mathcal{F}^{-1} \int_{\mathbb{R}^3} \widehat{K_{n-1}}(\xi_{n-1},\eta) \int_{\mathbb{R}^3} \widehat{x_j V}(t,\xi_n - \xi_{n-1}) \\
& \times e^{-it \vert \xi_n + \eta \vert^2} m(\xi_{n-1}, \xi_n,\eta) \widehat{f_{k_n}}(t,\xi_n+\eta)  d\xi_n d\xi_{n-1} \Bigg \Vert_{L^{4/3}_t L^{3/2}_x} .
\end{align*}
Now we can use Lemma \ref{nbilin} to bound the above expression by 
\begin{align*}
\Vert I_2 ^n f \Vert_{L^{\infty}_t L^2 _x} & \lesssim C_0^{n-2} \sum_{k_n \in \mathbb{Z}} \prod_{j \in J(n-1)} \min \lbrace 1.1^{-k_j/2} ; 1.1^{k_j/2} \rbrace \delta^{n-1} \Vert e^{it \Delta} f_{k_n} \Vert_{L^{4/3}_t L^6 _x} \Vert x V \Vert_{L^{\infty}_t L^2 _x} \\
& \lesssim C_0 ^{n-2} \sum_{k_n \in \mathbb{Z}} \prod_{j \in J(n-1)} \min \lbrace 1.1^{-k_j/2} ; 1.1^{k_j/2} \rbrace \delta^n \min \lbrace 1.1^{-5 k_n/6}; 1.1^{k_n /8} \rbrace \varepsilon_1 
\end{align*}
where to write the last inequality we used Lemma \ref{disp2}. \\
The desired result follows.
\end{proof}
We move on to the third type of iterate, namely $I_3 ^n f. $ 
From our earlier discussion in Section \ref{series-repr}, we know that these iterates have the additional property that $\vert k_n - k \vert >1.$ Indeed they arose when doing integration by parts in time in that case. The desired estimate from Proposition \ref{stepn:estimates} is obtained as a direct consequence of the following lemma:
\begin{lemma} \label{nthmodel1}
Let $J(n):= \lbrace i \in \lbrace 1;...;n \rbrace , k_i > k+1 \rbrace.$  \\
Recall that
\begin{align*}
\mathcal{F} I_3 ^n f(t,\xi) &:= \displaystyle \int_{\big( \mathbb{R}^3 \big)^{n-1}} \prod_{\gamma=1}^{n-1} \frac{\widehat{V}(t, \eta_{\gamma-1}-\eta_{\gamma}) P_{k_{\gamma}}(\eta_{\gamma}) P_{k}(\xi)}{\vert \xi \vert^2 - \vert \eta_{\gamma} \vert^2} d\eta_{1} ... d\eta_{n-2} \\
& \times \int_{\mathbb{R}^3}  \xi_l \frac{\widehat{V}(t,\eta_{n-1} - \eta_n)}{\vert \xi \vert^2 - \vert \eta_n \vert^2} e^{it(\vert \xi \vert^2 - \vert \eta_n \vert ^2)} \widehat{t f_{k_n}}(t,\eta_n) d\eta_n d\eta_{n-1} .
\end{align*}
We have
\begin{align*}
\Vert I_3 ^n f \Vert_{L^{\infty}_t L^2 _x} \lesssim C_0 ^{n-2} \min \lbrace 1.1^{k-k_n} , 1 \rbrace \bigg( \prod_{j \in J(n)} \min \lbrace 1.1^{k_j/2}; 1.1^{-k_j/2} \rbrace \bigg) \varepsilon_1 \delta^n .
\end{align*} 
where the implicit constant does not depend on $n.$ 
\end{lemma}
\begin{proof}
Note that since $\vert k - k_{n} \vert >1,$ $\displaystyle \frac{1}{\vert \xi \vert^2 - \vert \eta_n \vert^2} $ is not singular. \\
Let $g \in L^2 _x.$ We prove the estimate by duality. We have, denoting as usual $\widehat{V'}(t,\xi) =\widehat{V}(t,-\xi),$
\begin{eqnarray*}
\vert \langle I_3 ^n f , g \rangle \vert &=& \frac{1}{(2\pi)^3} \Bigg \vert \displaystyle \int_{ \big(\mathbb{R}^3 \big)^{n}} \prod_{\gamma =1}^{n-1} \frac{\widehat{V'}(t,\eta_{\gamma} - \eta_{\gamma-1}) P_{k_{\gamma}}(\eta_{\gamma}) P_{k}(\xi)}{\vert \xi \vert^2 - \vert \eta_{\gamma} \vert^2} d\eta_{1} ... d\eta_{n-2} \\
& \times & \int_{\mathbb{R}^3} t \xi_l \frac{\widehat{V'}(t,\eta_{n-1} - \eta_n)}{\vert \xi \vert^2 - \vert \eta_n \vert^2} e^{-it\vert \eta_n \vert^2} \widehat{f_{k_n}}(t,\eta_n) d\eta_n d\eta_{n-1} e^{it\vert \xi \vert^2} \bar{\widehat{g}}(\xi) d\xi \Bigg \vert \\
&=& \frac{1}{(2\pi)^3} \Bigg \vert \displaystyle \int_{\mathbb{R}^3} \int_{ \big( \mathbb{R}^3 \big)^{n-1}} \prod_{\gamma=1}^{n-1} \frac{\widehat{V'}(t,\xi_{\gamma} - \xi_{\gamma-1}) P_{k_{\gamma}}(\xi_{\gamma}+\eta) P_{k}(\eta)}{\vert \xi_{\gamma} +\eta \vert^2 - \vert \eta \vert^2} d\xi_{1} ... d\xi_{n-2} \\
& \times & \int_{\mathbb{R}^3} \eta_l \frac{P_k(\eta) P_{k_n}(\xi_n+\eta)\widehat{V'}(t,\xi_n-\xi_{n-1})}{\vert \xi_n +\eta \vert^2 - \vert \eta \vert^2} \overline{e^{-it \vert \eta \vert^2}\widehat{g_k}}(\eta) t e^{-it\vert \xi_n+ \eta \vert^2}\widehat{f_{k_n}}(t,\xi_n+\eta) d\xi_n d\xi_{n-1} d\eta \Bigg \vert \\
\end{eqnarray*}
where we re-labelled $\xi = \eta$ and did the change of variables $\xi_{\gamma} = \eta_{\gamma}-\xi.$ We also used the convention $\xi_0 = 0.$ \\
Now we use Plancherel's theorem to write that
\begin{align*}
\vert \langle I_3 ^n f , g \rangle \vert = (2\pi)^3 \Bigg \vert \int_{\mathbb{R}^6} K_n^{i_1,...,i_n}(x,y) \overline{\tilde{g}}(t,y-x) (\overline{te^{it\Delta}f_{k_n}})(-x) dy dx \Bigg \vert
\end{align*}
where $\tilde{g}(t,x) = \mathcal{F}^{-1}_{\eta}(\eta_l e^{-it\vert \eta \vert^2} \widehat{g_k}(\eta)).$
\\
Note that since $\vert k_n - k \vert >1$ then either $i_n =1$ and $\vert k_n - k \vert >1$ or $i_n = 2.$ We detail the first case, the second one being similar. \\
Using Lemma \ref{kernel3} we get 
\begin{align*}
K_n ^{i_1,...,1}(x,y)= \int_{\mathbb{R}^6}  \big( K_{n-1}^{i_1,...,i_{n-1}} (x-r, y')V'(t,x-r) \big)_{\leqslant k_n+10} \check{m}(r,y-y') dr dy'  \\
\end{align*}
where
\begin{align*}
m(\xi_n,\eta) = \frac{P_k(\eta)P_{k_n}(\xi_n+\eta)}{\vert \xi_n + \eta \vert^2 - \vert \eta \vert^2}
\end{align*}
Now we can write, using H\"{o}lder's inequality, that
\begin{align*}
& \Bigg \vert \int_{\mathbb{R}^6} K_n^{i_1,...,i_n}(x,y) \overline{\tilde{g}}(t,y-x) (\overline{t e^{it\Delta}f_{k_n}})(-x) dy dx \Bigg \vert \\
& = \Bigg \vert \int_{\mathbb{R}^9} (\overline{t e^{it\Delta}f_{k_n}})(-x)  \big( K_{n-1}^{i_1,...,i_{n-1}} (\cdot, y')V'(t,\cdot) \big)_{\leqslant k_n+10} (x-r) \big(\check{m}(r,\cdot) \star \overline{\tilde{g}}(t,-\cdot) \big)(x-y') dx dy' dr \Bigg \vert \\
& \lesssim \int_{\mathbb{R}^6} \Vert t e^{it\Delta} f_{k_n} \Vert_{L^6 _x} \Vert  \big( K_{n-1}^{i_1,...,i_{n-1}} (\cdot, y')V'(t,\cdot) \big)_{\leqslant k_n+10} \Vert_{L^3 _x} \Vert \check{m}(r,\cdot) \star \overline{\tilde{g}}(t,-\cdot) \Vert_{L^2_x} dy' dr  \\
& \lesssim \Vert t e^{it\Delta} f_{k_n} \Vert_{L^6 _x}  \Vert  \big( K_{n-1}^{i_1,...,i_{n-1}} (x, y') V'(t,x) \big)_{\leqslant k_n+10} \Vert_{L^1 _{y'} L^3 _x} \Bigg(\int_{\mathbb{R}^3} \Vert \check{m}(r,\cdot) \star \overline{\tilde{g}}(t,-\cdot) \Vert_{L^2_x} dr \Bigg).
\end{align*}
Now we use Bernstein's inequality and Lemma \ref{symbolbis} to obtain
\begin{align*}
& \Bigg \vert \int_{\mathbb{R}^6} K_n^{i_1,...,i_n}(x,y) \overline{\tilde{g}}(t,y-x) (t e^{it\Delta}f_{k_n})(x) dy dx \Bigg \vert \\
& \lesssim 1.1^{k-k_n} 1.1^{-k_n} \Vert g \Vert_{L^2_x} \Vert \big( K_{n-1}^{i_1,...,i_{n-1}} (x, y')V'(t,x) \big)_{\leqslant k_n+10} \Vert_{L^1_{y'} L^3 _x} \varepsilon_1 \\
& \lesssim  1.1^{k-k_n} 1.1^{-k_n/2} \Vert g \Vert_{L^2_x} \Vert \big( K_{n-1}^{i_1,...,i_{n-1}} (x, y')V'(t,x) \big)_{\leqslant k_n+10} \Vert_{L^1_{y'} L^2 _x} \varepsilon_1 \\
&\lesssim  1.1^{k-k_n} \min \lbrace 1.1^{k_n/2}; 1.1^{k_n/2} \rbrace \Vert  K_{n-1}^{i_1,...,i_{n-1}} (x, y')V'(t,x)\Vert_{L^1_{y'} B_x} \varepsilon_1, 
\end{align*}
and we conclude using Corollary \ref{key} with $\theta_j = \pm 1/2$ for $j \in J(n-1).$ 
\end{proof}
For the last linear iterate, namely $I_4 ^n f$, we cannot rely on Lemma \ref{nbilin}, or a variation of it as before. The proof is closer in spirit to the proof of Beceanu and Schlag in \cite{BSmain}.\\
\\
First note that we know from Section \ref{series-repr} that these terms arise when $\vert k_n - k \vert \leqslant 1.$ We show
\begin{lemma}  \label{nthmodel2}
Let $J(n):= \lbrace i \in \lbrace 1;...;n \rbrace , k < k_i-1 \rbrace.$ \\
Let 
\begin{align*}
\mathcal{F} I_4 ^n f(t,\xi) &= \displaystyle \int_{\big(\mathbb{R}^3 \big)^{n-1}} \prod_{\gamma=1}^{n-1} \frac{\widehat{V}(t, \eta_{\gamma-1}-\eta_{\gamma}) P_{k_\gamma}(\eta_\gamma) P_{k}(\xi)}{\vert \xi \vert^2 - \vert \eta_{\gamma} \vert^2} d\eta_{1} ... d\eta_{n-2} \\
& \times  \int_{\mathbb{R}^3}  \frac{\xi_l \eta_{n,j}}{\vert \eta_n \vert^2} \frac{P_k (\xi) \widehat{V}(t,\eta_{n-1} - \eta_n)}{\vert \xi \vert^2 - \vert \eta_n \vert^2} e^{it(\vert \xi \vert^2 - \vert \eta_n \vert ^2)} \partial_{\eta_{n,j}} \widehat{f}(t,\eta_n)P_{k_n}(\eta_n) d\eta_n d\eta_{n-1} .
\end{align*}
We have
\begin{align*}
\Vert I_4 ^n f \Vert_{L^{\infty}_t L^2 _x} \lesssim C_0 ^{n-2} \bigg( \prod_{j \in J(n)} \min \lbrace 1.1^{k_j/2} ;1.1^{-k_j/2} \rbrace \bigg) \varepsilon_1 \delta^n
\end{align*} 
where by convention $\prod_{j \in \emptyset} =1.$ \\
The implicit constant does not depend on $n.$ 
\end{lemma}
\begin{remark}
As was the case for the other lemmas of this section, the corresponding bound in Proposition \ref{stepn:estimates} is obtained from the above estimate, due to the converging factor in front.
\end{remark}
\begin{proof}
For notational convenience we drop the time-dependence of the potential in this proof. \\
We prove the estimate by duality: let $g \in L^2 (\mathbb{R}^3).$ \\
We start with the same change of variables as usual to obtain 
\begin{align*}
\vert \langle I_4 ^n f , g \rangle \vert &= \frac{1}{(2\pi)^3} \Bigg \vert
\displaystyle \int_{\big( \mathbb{R}^3 \big)^{n+1}} \prod_{\gamma=1}^{n-1} \frac{\widehat{V'}(\xi_{\gamma} - \xi_{\gamma-1}) P_{k_{\gamma}}(\xi_{\gamma}+\eta) P_{k}(\eta)}{\vert \xi_{\gamma} +\eta \vert^2 - \vert \eta \vert^2} d\xi_{1} ... d\xi_{n-2} \\
& \times \frac{\widehat{V'}(\xi_{n} - \xi_{n-1})}{\vert \xi_n +\eta \vert^2 - \vert \eta \vert^2} \eta_l  \overline{e^{it \vert \eta \vert^2} \widehat{g_k}(\eta)} \frac{(\xi_n + \eta)_j}{\vert \xi_n+\eta \vert^2} e^{-it \vert \xi_n + \eta \vert^2} \partial_{\xi_n} \widehat{f}(t,\xi_n+\eta)P_{k_n}(\xi_n+\eta) d\xi_n d\xi_{n-1} d\eta \Bigg \vert .  \\
\end{align*}
Note that in the integral on $\xi_{n}$ and $\eta$ we removed the frequency localizations from the kernel and placed them on the functions. Therefore in this case the kernel is such that $i_n =3.$ 

Then we use Plancherel's theorem to write this expression as
\begin{align*}
\vert \langle I_4 ^n f , g \rangle \vert =(2 \pi)^3 \Bigg \vert \int_{\mathbb{R}^6} K_n ^{i_1, ..., 3} (x,y) \overline{\widetilde{g}}(t,y-x) \overline{\widetilde{f}} (t,-x) dxdy \Bigg \vert ,
\end{align*}
where
\begin{align*}
\widetilde{g}(t,x) &= \mathcal{F}^{-1}(\eta_l e^{it\vert \eta \vert^2} \widehat{g_k}(\eta))(x), \\
\widetilde{f}(t,y) &= \mathcal{F}^{-1} \bigg[ \frac{\xi_j}{\vert \xi \vert^2} e^{-it\vert \xi \vert^2} \partial_{\xi_n}  \widehat{f}(\xi)P_{k_n}(\xi) \bigg] .
\end{align*}
Now we use the expression for $K_n^{i_1,...,3}$ from Lemma \ref{kernel3} to obtain
\begin{align*}
&\int_{\mathbb{R}^6} K_n ^{i_1, ..., 3} (x,y) \overline{\widetilde{f}}(t,-x) \overline{\widetilde{g}} (t,y-x) dxdy \\ 
&= \int_{\mathbb{R}^6} \Bigg[ \int_{\mathbb{R}^3} \frac{1}{\vert y-y' \vert^2} L_{V(\cdot) K_{n-1}(\cdot,y')}(\vert y-y' \vert - 2 \widehat{(y-y')} \cdot x,\widehat{(y-y')} ) dy' \Bigg] \overline{\tilde{f}}(t,-x) \overline{\tilde{g}} (t,y-x) dx dy   \\
& =\int_{\mathbb{R}^6} \overline{\tilde{f}}(t,-x) \int_{\mathbb{R}^3} \frac{1}{\vert y \vert^2} L_{V(\cdot) K_{n-1}(\cdot,y')}(\vert y \vert - 2 \widehat{y} \cdot x,\widehat{y}) \overline{\widetilde{g}} (t,y+y'-x)  dy dx dy' \\
& = \int_{\mathbb{R}^6} \overline{\widetilde{f}}(t,-x) \int_{\mathbb{S}^2} \int_{0}^{\infty} L_{V(\cdot) K_{n-1}(\cdot,y')}(\rho - 2 \omega \cdot x,\omega) \overline{\tilde{g}} (t,\rho \omega+y'-x)  d\rho d\omega dx dy' \\
& = \int_{\mathbb{R}^3} \int_{\mathbb{S}^2} \int_{0}^{\infty} L_{V(\cdot) K_{n-1}(\cdot,y')}(\rho,\omega) \int_{\mathbb{R}^3} \overline{\widetilde{f}}(t,-x) \textbf{1}_{\rho>-2\omega \cdot x} \overline{\widetilde{g}}(t,y'+\rho \omega+2(\omega \cdot x) \omega-x) dx d\rho d\omega dy' .
\end{align*}
Next we use H\"{o}lder's inequality in $x$ as well as lemma \ref{BSmain2} and Lemma \ref{X'} to deduce from this equality that
\begin{align*}
\Bigg \vert \int_{\mathbb{R}^6} K_n ^{i_1, ..., 3} (x,y) \overline{\widetilde{f}}(t,-x) \overline{\tilde{g}} (t,y-x) dxdy \Bigg \vert & \leqslant \Vert \tilde{g} \Vert_{L^{\infty}_t L^2_x} \Vert \tilde{f} \Vert_{L^{\infty}_t L^2_x} \\
& \times  \int_{\mathbb{R}^3} \int_{\mathbb{S}^2} \int_{0}^{\infty} \vert L_{V(\cdot) K_{n-1}(\cdot,y')}(\rho,\omega) \vert d\rho d\omega dy' \\
& \lesssim  \Vert \tilde{g} \Vert_{L^{\infty}_t L^2_x} \Vert \tilde{f} \Vert_{L^{\infty}_t L^2_x} \Vert V(x) K_{n-1}(x,y') \Vert_{L^1_{y'} B_x} \\
& \lesssim \Vert g \Vert_{L^2_x} \Vert V(x) K_{n-1}(x,y') \Vert_{L^1_{y'} B_x} \varepsilon_1
\end{align*}
and we can conclude with Corollary \ref{key}.
\end{proof}

\subsection{Estimating the $n-$th iterate of bilinear terms} 
Now we can move on to the $n-$th iterates of bilinear terms present in Proposition \ref{stepn:estimates}, namely $I_5 ^n f$ and $I_6 ^n f.$ This proposition mirrors the second subsection of Section \ref{firstit}, where the first iterates of this type have been treated. As it was the case there, the most challenging term is $I_6 ^n f.$ It requires the use of the space-time resonance theory. Compared to the case of the first iterate, a few modifications are needed: notice for example that the phase in the $n-$th iterate depends on three variables ($\xi,\eta_{n-1}$ and $\eta_n$), and not two ($\xi$ and $\eta_1$) as in the first iterate case. In particular the sum of these three frequencies is only approximately 0. \\
We will also need a substitute to Lemma \ref{bilin} since we also have to keep track of the dependence on $\delta$ in the estimates. That is exactly the role that Lemma \ref{nbilin} will play. 

\subsubsection{Bounding $I_5 ^n f$}
We deal with the easier bilinear term in this subsection. 
\\
We prove the following estimate:
\begin{lemma}
Let $J(n-1) = \big \lbrace j \in \lbrace 1;...;n-1 \rbrace ; k_j > k+1 \big \rbrace. $
\\
Let
\begin{align*}
\mathcal{F} I_5 ^n f(t,\xi) & := \int_1 ^t \int_{ \big(\mathbb{R}^3 \big)^{n-1} } \prod_{l=1}^{n-1} \frac{\widehat{V}(s,\eta_{l-1}-\eta_l) P_{k_l}(\eta_l) P_{k}(\xi)}{\vert \xi \vert^2 - \vert \eta_l \vert^2} d\eta_{1} ... d\eta_{n-2} \frac{\xi_l \eta_{n-1,j}}{\vert \eta_{n-1} \vert^2} \\
& \times \int_{ \mathbb{R}^3} \partial_{\eta_{n,j}}  \widehat{f}(s,\eta_{n-1}-\eta_n) e^{is(\vert \xi \vert^2- \vert \eta_n - \eta_{n-1} \vert^2 - \vert \eta_n \vert^2)} \widehat{f}(s,\eta_n) d\eta_n d\eta_{n-1} ds.
\end{align*}
Then 
\begin{align*}
\Vert I_5 ^n f \Vert_{L^{\infty}_t L^2 _x} \lesssim C_0^{n-2} \prod_{j \in J(n-1)} \min \lbrace 1.1^{-k_j/2} ; 1.1^{k_j/2} \rbrace \varepsilon_1 ^2 \delta^{n-1}.
\end{align*}
The implicit constant does not depend on $n.$ 
\end{lemma}
\begin{proof}
We start by splitting dyadically in $\eta_n$ and denote $k_n$ the corresponding exponent. We have 
\begin{align*}
\Vert I_5 ^n f \Vert_{L^{\infty}_t L^2 _x} & \lesssim \sum_{k_n \in \mathbb{Z}} \Bigg \Vert \mathcal{F}^{-1}_{\eta} \int_{\xi_{n-1}} \widehat{K_{n-1}}(\xi_{n-1},\eta) \\
& \times \int_{\xi_n} e^{-is \vert \xi_n - \xi_{n-1} \vert^2} \partial_{n,j} \widehat{f}(t,\xi_n - \xi_{n-1})m(\xi_{n-1},\xi_n,\eta) e^{-is \vert \xi_n + \eta \vert^2} \widehat{f_{k_n}}(t,\xi_n+\eta)  d\xi_n d\xi_{n-1} \Bigg  \Vert_{L^{4/3}_t L^{3/2}_x} 
\end{align*}
where the multiplier $m$ is defined as
\begin{align*}
m(\xi_{n-1},\xi_n, \eta) = \frac{\eta_l (\xi_{n-1} + \eta)_j P_k (\eta) P_{k_{n-1}}(\xi_{n-1} + \eta)P_{k_n}(\xi_n+\eta)}{\vert \xi_{n-1} + \eta \vert^2}.
\end{align*}
Now we distinguish several cases: \\
\underline{Case 1: $k_n > k+1:$}
In that case we can add an additional localization in $\xi_n - \xi_{n-1}$ (since $\vert \xi_n - \xi_{n-1} \vert \sim 1.1^{k_n}$) and use Lemma \ref{nbilin}, Lemma \ref{X'} and Lemma \ref{disp2} to bound the above integral by
\begin{align*}
\Vert I_5 ^n f \Vert_{L^{\infty}_t L^2 _x} & \lesssim C_0 ^{n-2} \sum_{k_n \in \mathbb{Z}} \prod_{j \in J(n-1)} \min \lbrace 1.1^{-k_j/2} ; 1.1^{k_j/2} \rbrace \delta^{n-1} \Vert \partial_{\eta_{n,j}} \widehat{f} P_{k_n} \Vert_{L^{\infty}_t L^2 _x} \Vert e^{it \Delta} f_{k_n} \Vert_{L^{4/3}_t L^6 _x} \\ 
& \lesssim C_0^{n-2} \sum_{k_n \in \mathbb{Z}} \prod_{j \in J(n-1)} \min \lbrace 1.1^{-k_j/2} ; 1.1^{k_j/2} \rbrace \delta^{n-1} \min  \lbrace 1.1^{-5 k_n ^+ / 6} ; 1.1^{k_n /8} \rbrace \varepsilon_1 ^2 \\
& \lesssim C_0^{n-2} \bigg( \prod_{j \in J(n-1)} \min \lbrace 1.1^{-k_j/2} ; 1.1^{k_j/2} \rbrace \bigg) \delta^{n-1} \varepsilon_1 ^2 .
\end{align*}
\underline{Case 2: $k_n <k-1:$}
Similar to the previous case, with $\vert \xi_n - \xi_{n-1} \vert \sim 1.1^{k}.$
\\
\underline{Case 3: $\vert k_n - k \vert \leqslant 1:$}
\\
In that case we split dyadically in $\xi_n - \xi_{n-1}$ and denote $k_n '$ the corresponding exponent. \\ 
Note that $1.1^{k_n'} \sim \vert \xi_n - \xi_{n-1} \vert \leqslant \vert \xi_n + \eta \vert + \vert \xi_{n-1} + \eta \vert \leqslant 1.1^{k+3}+1.1^{k+2} \leqslant 1.1^{20 + k}$ (we used that $\vert k - k_{n-1} \vert \leqslant 1$ which is always the case for the bilinear terms). Therefore $k_n' \leqslant k+20 \leqslant k_n + 21.$
\\
Then we use Lemma \ref{nbilin} with $p=2.1 ,q=63/12=21/4 , r=3/2 $ as well as Lemma \ref{X'} to write that
\begin{align*}
\Vert I_5 ^n f \Vert_{L^{\infty}_t L^2 _x} &\lesssim C_0^{n-2} \sum_{k_n \in \mathbb{Z}} \sum_{k_n ' \leqslant k_n +21} \prod_{j \in J(n-1)} \min \lbrace 1.1^{-k_j/2} ; 1.1^{k_j/2} \rbrace \delta^{n-1} \\
& \times \Vert \mathcal{F}^{-1} \big( e^{-it \vert \xi_{n} \vert^2} \partial_{\xi_{n,j}} \widehat{f}(\xi_n) P_{k_n'}(\xi_n) \big) \Vert_{L^{\infty}_t L^{2.1}_x} \Vert e^{it \Delta} f_{k_n} \Vert_{L^{4/3}_t L^{21/4}_x} \\
& \lesssim C_0^{n-2} \sum_{k_n \in \mathbb{Z}} \sum_{k_n ' \leqslant k_n +21} \prod_{j \in J(n-1)} \min \lbrace 1.1^{-k_j/2} ; 1.1^{k_j/2} \rbrace \delta^{n-1} \\
& \times 1.1^{k_n ' /14} \Vert \mathcal{F}^{-1} \big( e^{it \vert \xi_{n} \vert^2} P_{k_n'}(\xi_n) \partial_{\xi_{n,j}} \widehat{f}(\xi_n) \big) \Vert_{L^{\infty}_t L^{2}_x} \Vert e^{it \Delta} f_{k_n} \Vert_{L^{4/3}_t L^{21/4}_x} .
\end{align*}
For the second line we used Bernstein's inequality. \\
Now we can write, using the energy bound and dispersive estimates, that if $k_n \geqslant 0,$
\begin{align*}
\Vert e^{it \Delta} f_{k_n} \Vert_{L^{21/4}_x} & \leqslant \Vert e^{it\Delta} f_{k_n} \Vert_{L^{\infty}_x}^{\frac{1}{20} \times \frac{4}{21}} \Vert e^{it \Delta} f_{k_n} \Vert_{L^{26/5}_x}^{\frac{104}{105}} \\
                                               & \lesssim \Vert e^{it\Delta} f_{k_n} \Vert_{H^{10}_x}^{\frac{1}{105}} \Vert e^{it \Delta} f_{k_n} \Vert_{L^{26/5}_x} ^{\frac{104}{105}} \\
                                               & \lesssim 1.1^{-\frac{k_n}{14}} 1.1^{-\frac{k_n}{42}} \frac{1}{t^{-32/35}} \varepsilon_1 ^2 
\end{align*}
and we can conclude that
\begin{align*}
\Vert I_5 ^n f \Vert_{L^{\infty}_t L^2 _x} &\lesssim C_0^{n-2} \sum_{k_n \geqslant 0} \sum_{k_{n}' \leqslant k_n +21} \prod_{j \in J(n-1)} \min \lbrace 1.1^{-k_j/2} ; 1.1^{k_j/2} \rbrace \delta^{n-1} \\
& \times 1.1^{1/14(k_n ' - k_n)} 1.1^{-k_n/42} \Bigg( \int_1 ^t s^{-\frac{32}{35} \times \frac{4}{3}} ds \Bigg)^{\frac{3}{4}} \varepsilon_1 ^2 .
\end{align*}
Now we can sum over $k_n'$ first, and then over $k_n$ to get the desired result.
\\
The proof in the case $k_n \leqslant 0$ is treated similarly by writting instead
\begin{align*}
\Vert e^{it \Delta} f_{k_n} \Vert_{L^{21/4}_x} & \leqslant \Vert e^{it\Delta} f_{k_n} \Vert_{L^{\infty}_x}^{\frac{1}{40} \times \frac{4}{21}} \Vert e^{it \Delta} f_{k_n} \Vert_{L^{209/40}_x}^{\frac{209}{210}}
\end{align*}
and carrying out the same reasoning. Given the similarity with the previous case the details are omitted. 
\end{proof}

\subsubsection{Bounding $I_6 ^n f$}
We now turn to the most involved bilinear iterate, namely \eqref{R9}. For this term we will need to use the theory of space-time resonances, as was the case for the first iterate of this term.\\
Let's also recall that these iterates have the property that $\vert k - k_{n-1} \vert \leqslant 1.$ \\
We prove that 
\begin{lemma}\label{estimateR9}
Let $J(n-1):= \lbrace i \in \lbrace 1;...;n-1 \rbrace , k < k_i-1 \rbrace .$ \\
Let
\begin{align*}
\mathcal{F} I_{6} ^n f(t,\xi) & := \int_{1} ^{t} \int_{\big(\mathbb{R}^3 \big)^{n-1}} \prod_{l=1}^{n-1} \frac{\widehat{V}(s,\eta_{l-1}-\eta_l) P_{k_l}(\eta_l) P_{k}(\xi)}{\vert \xi \vert^2 - \vert \eta_l \vert^2} d\eta_{1} ... d\eta_{n-2} \frac{\xi_l \eta_{n-1,j}}{\vert \eta_{n-1} \vert^2} \\
& \times \int_{\eta_n} i s \eta_{n,j} \widehat{f}(s,\eta_{n-1}-\eta_n) e^{is(\vert \xi \vert^2 - \vert \eta_{n-1} - \eta_n \vert ^2 - \vert \eta_n \vert^2)} \widehat{f}(s,\eta_n) d\eta_n  d\eta_{n-1}.
\end{align*}
Then
\begin{align*}
\Vert I_{6}  ^n f \Vert_{L^{\infty}_t L^2 _x} \lesssim C_0^{n-2} \bigg( \prod_{j \in J(n-1)} \min \lbrace 1.1^{-k_j/2} ; 1.1^{k_j/2} \rbrace \bigg)  \varepsilon_1 ^2 \delta^{n-1}.
\end{align*} 
where by convention $\prod_{j \in \emptyset} =1.$ \\
The implicit constant does not depend on $n.$ 
\end{lemma}
\begin{proof}
We start by splitting dyadically in $\eta_n$ (the corresponding exponent is $k_n$), in $\eta_n - \eta_{n-1}$ (the corresponding exponent is $k_{n+1}$) and in time ($m$ denotes the exponent). Then we estimate the $L^{\infty}_t L^2 _x$ norm of 
\begin{align*}
\mathcal{F} I_{6,m,k_n,k_{n+1}} ^n f & := \int_{1.1^m} ^{1.1^{m+1}} \int_{\big(\mathbb{R}^3 \big)^{n-1}} \prod_{l=1}^{n-1} \frac{\widehat{V}(s,\eta_{l-1}-\eta_l) P_{k_l}(\eta_l) P_{k}(\xi)}{\vert \xi \vert^2 - \vert \eta_l \vert^2} d\eta_{1} ... d\eta_{n-2} \frac{\xi_l \eta_{n-1,j}}{\vert \eta_{n-1} \vert^2} \\
& \times \int_{\mathbb{R}^3} i s \eta_{n} \widehat{f_{k_{n+1}}}(s,\eta_{n-1}-\eta_n) e^{is(\vert \xi \vert^2 - \vert \eta_{n-1} - \eta_n \vert ^2 - \vert \eta_n \vert^2)} \widehat{f_{k_{n}}}(s,\eta_n) d\eta_n d\eta_{n-1} ds.
\end{align*}  
Now we distinguish different cases. Note that for cases 1 and 2 the discussion is the same as in Lemma \ref{bilinhard} but with Lemma \ref{nbilin} replacing Lemma \ref{bilin}.
\\
\\
\underline{Case 1: $\max \lbrace k_n , k_{n+1} \rbrace \geqslant m$}
\\
Then to estimate the integral above we write that, doing the usual change of variables:
\begin{align*}
&\big \Vert I_{6,m,k_n,k_{n+1}} ^n f \big \Vert_{L^{\infty}_t L^2 _x}  \\
&= \Bigg \Vert \int_{1.1^m} ^{1.1^{m+1}}\int_{\big(\mathbb{R}^3 \big)^{n-1}} \prod_{l=1}^{n-1} \frac{\widehat{V'}(s,\xi_l - \xi_{l-1}) P_{k_l}(\xi_l+\eta) P_{k}(\eta)}{\vert \xi \vert^2 - \vert \eta_l \vert^2} d\xi_{1} ... d\xi_{n-2} \int_{\mathbb{R}^3} i s \widehat{f_{k_{n}}}(s, \xi_{n-1}-\xi_n) \\
& \times m(\xi_{n-1},\xi_n,\eta) e^{is(\vert \eta \vert^2 - \vert \xi_{n-1} - \xi_n \vert ^2 - \vert \xi_n +\eta \vert^2)} \widehat{f_{k_{n}}}(s,\xi_n+\eta) d\xi_n d\xi_{n-1} ds \Bigg \Vert_{L^{\infty}_t L^2 _x} 
\end{align*} 
with the multiplier 
\begin{align*}
m(\xi_{n-1}, \xi_n , \eta) &= \frac{P_k(\eta) P_{k_{n-1}} (\xi_{n-1}+\eta) P_{k_{n+1}}(\xi_n + \eta) \eta_l (\xi_{n-1}+ \eta)_j (\xi_n + \eta)_j}{\vert \xi_{n-1}+ \eta \vert^2} .
\end{align*} 
Now we use Lemma \ref{nbilin} to write that
\begin{align*}
\big \Vert I_{6,m,k_n,k_{n+1}} ^n f \big \Vert_{L^{\infty}_t L^2 _x} & \leqslant \Bigg \Vert \int_{1.1^m} ^{1.1^{m+1}} e^{-i s \Delta} \Bigg( \mathcal{F}^{-1}_{\eta} \int_{\mathbb{R}^3} \widehat{K}_{n-1}(\xi_{n-1},\eta) \int_{\mathbb{R}^3} is \widehat{f_{k_{n+1}}}(s, \xi_{n-1}-\xi_n) \\
& \times m(\xi_{n-1},\xi_n,\eta) e^{is(\vert \eta \vert^2 - \vert \xi_{n-1} - \xi_n \vert ^2 - \vert \xi_n +\eta \vert^2)} \widehat{f_{k_{n}}}(s,\xi_n+\eta) d\xi_n d\xi_{n-1} \Bigg) ds \Bigg \Vert_{L^{\infty}_t L^2_x} \\
& \leqslant 1.1^m  \Bigg \Vert \mathcal{F}^{-1}_{\eta} \int_{\mathbb{R}^3} \widehat{K}_{n-1}(\xi_{n-1},\eta) \int_{\mathbb{R}^3} it \widehat{f_{k_{n+1}}}(t, \xi_{n-1}-\xi_n) \\
& \times m(\xi_{n-1},\xi_n,\eta) e^{it(\vert \eta \vert^2 - \vert \xi_{n-1} - \xi_n \vert ^2 - \vert \xi_n +\eta \vert^2)} \widehat{f_{k_{n}}}(t,\xi_n+\eta) d\xi_n d\xi_{n-1}\Bigg \Vert_{L^{\infty}_t L^2_x} \\
& \lesssim C_0 ^{n-2} 1.1^{2m} 1.1^{\max \lbrace k_n ; k_{n+1} \rbrace} 1.1^{-10 \max \lbrace k_n ; k_{n+1} \rbrace} \\ &\times \min \big \lbrace 1.1^{-10 \min \lbrace k_n ; k_{n+1} \rbrace}; 1.1^{3 \min \lbrace k_n ; k_{n+1} \rbrace /2} \big \rbrace 
 \prod_{j \in J(n-1)} \min \lbrace 1.1^{-k_j/2} ; 1.1^{k_j/2} \rbrace   \varepsilon_1 ^2 \delta^{n-1} 
\end{align*}
which can be summed. \\
\underline{Case 2: $\min \lbrace k_n, k_{n+1} \rbrace \leqslant -2m: $}
\\
This is similar to case 1: we use Lemma \ref{nbilin} with the same multiplier, and we put the low frequency term in $L^{\infty}$ and the high frequency one in $L^{2}.$ As a result we get the bound
\begin{align*}
\Vert I_{6,m,k_n,k_{n+1}} ^n f \Vert_{L^{\infty}_t L^2_x} & \lesssim C_0 ^{n-2} 1.1^{2m} 1.1^{\max \lbrace k_n ; k_{n+1} \rbrace} 1.1^{3 \min \lbrace k_n ; k_{n+1} \rbrace /2} \\ &\times \min \big \lbrace 1.1^{-10 \max \lbrace k_n ; k_{n+1} \rbrace}; 1.1^{\max \lbrace k_n ; k_{n+1} \rbrace /3} \big \rbrace 
 \prod_{j \in J(n-1)} \min \lbrace 1.1^{-k_j/2} ; 1.1^{k_j/2} \rbrace   \varepsilon_1 ^2 \delta^{n-1} \\
& \lesssim C_0^{n-2} 1.1^{-0.5 m} 1.1^{0.25 \min \lbrace k_n ; k_{n+1} \rbrace}  1.1^{\max \lbrace k_n ; k_{n+1} \rbrace} \\
& \times \min \big \lbrace 1.1^{-10 \max \lbrace k_n ; k_{n+1} \rbrace}; 1.1^{\max \lbrace k_n ; k_{n+1} \rbrace /3} \big \rbrace \\
&\times \prod_{j \in J(n-1)} \min \lbrace 1.1^{-k_j/2} ; 1.1^{k_j/2} \rbrace   \varepsilon_1 ^2 \delta^{n-1} 
\end{align*}
which can be summed. \\
\underline{Case 3: $-2m \leqslant k_n,k_{n+1} \leqslant m$} \\
In this case there are $O(m^2)$ terms in the sum on $k_1,k_2.$ \\
Mimicking the proof of Lemma \ref{bilinhard}, we introduce a localization in $\frac{1}{2}\nabla_{\eta_n} (-\vert \xi \vert^2 + \vert \eta_n - \eta_{n-1} \vert^2 + \vert \eta_n \vert^2) =2 \eta_n - \eta_{n-1}.$ \\ Let's denote $k'_n$ the corresponding exponent of localization. \\
\underline{Case 3.1: $k'_n \leqslant -10 m.$} \\
Let
\begin{align*}
F(\eta_{n-1} ) = \int_{\mathbb{R}^3} is \eta_{n} P_{k_n '} (2\eta_n -\eta_{n-1}) \widehat{f_{k_n}}(s,\eta_{n-1} - \eta_n) e^{is(- \vert \eta_{n-1} - \eta_n \vert ^2 - \vert \eta_n \vert^2)} \widehat{f_{k_{n+1}}}(t,\eta_n) d\eta_n .
\end{align*}
Similarly to what has been done in Lemma \ref{bilinhard} (replace $\xi$ by $\eta_{n-1}$ and $\eta$ by $\eta_n$) we have
\begin{align*}
\Vert F \Vert_{L^2 _{\eta_{n-1}}} \lesssim 1.1^{0.1 k_n'} 1.1^{-13 m} \varepsilon_1 ^2.
\end{align*}
Using this fact we write that
\begin{align*}
\Vert I_{6,m,k_n,k_{n+1}} ^n f \Vert_{L^{\infty}_t L^2_x} &= \Bigg \Vert \int_{1.1^m} ^{1.1^{m+1}}  \int_{\big(\mathbb{R}^3 \big)^{n-1}} \prod_{l=1}^{n-1} \frac{\widehat{V'}(s,\xi_l - \xi_{l-1}) P_{k_l}(\xi_l+\eta) P_{k}(\eta)}{\vert \xi \vert^2 - \vert \eta_l \vert^2} \\
& \times \frac{\eta_l(\xi_{n-1}+\eta)_j}{\vert \xi_{n-1}+\eta \vert^2} F(\xi_{n-1}+\eta) e^{is \vert \eta \vert^2} d\xi_{1} ... d\xi_{n-1} ds \Bigg \Vert_{L^{\infty}_t L^2_x} \\
&= \Bigg \Vert \int_{1.1^m} ^{1.1^{m+1}} e^{-is\Delta} \bigg( \mathcal{F}^{-1} \int_{\mathbb{R}^3} \widehat{K_{n-1}} (\xi_{n-1},\eta) \\
& \times \frac{\eta_l (\xi_{n-1}+\eta)_j}{\vert \xi_{n-1}+\eta \vert^2} F_{k_{n-1}}(\xi_{n-1}+\eta)  d\xi_{n-1}  \bigg) ds \Bigg  \Vert_{L^{\infty}_t L^2_x} .
\end{align*}
Now we can repeat the proof of Lemma \ref{nthmodel2} (use Plancherel's theorem, then Corollary \ref{key}) to obtain
\begin{align*}
\Vert I_{6,m,k_n,k_{n+1}} ^n f \Vert_{L^{\infty}_t L^2_x} & \lesssim C_0^{n-2} 1.1^m \prod_{j \in J(n-1)} \min \lbrace 1.1^{-k_j/2} ; 1.1^{k_j/2} \rbrace \delta^{n-1} 1.1^{k_n} \Vert F \Vert_{L^2} \\
& \lesssim C_0^{n-2} 1.1^{0.1 k_n '} 1.1^{-12 m} \prod_{j \in J(n-1)} \min \lbrace 1.1^{-k_j/2} ; 1.1^{k_j/2} \rbrace \delta^{n-1} 1.1^{k_n}  \varepsilon_1 ^2,
\end{align*}  
which can be summed over $k_n'$ then $k_n, k_{n+1}$ (only $O(m^2)$ terms) and $m.$ \\
\underline{Case 3.2: $k'_n \geqslant k_n -50, k_n' \geqslant -10 m$} \\
This is, again, the analog of the same case in Lemma \ref{bilinhard}. \\
We do an integration by parts in $\eta_n$ and obtain that 
\begin{align} 
&\notag \mathcal{F} I_{6,m,k_n,k_{n+1}} ^n f  = \\
\label{space1bis}&\frac{1}{2} \int_{1.1^m} ^{1.1^{m+1}} \int_{\big(\mathbb{R}^3\big)^{n-1}} \prod_{l=1}^{n-1} \frac{\widehat{V}(s,\eta_l - \eta_{l-1}) P_{k_l}(\eta_l) P_{k}(\xi)}{\vert \xi \vert^2 - \vert \eta_l \vert^2} d\eta_{1} ... d\eta_{n-2} \frac{\xi_l \eta_{n-1,j}}{\vert \eta_{n-1} \vert^2} \int_{\mathbb{R}^3} \frac{\eta_{n}(2\eta_n - \eta_{n-1})_j}{\vert 2 \eta_n - 2\eta_{n-1} \vert^2} \\
&\notag \times P_{k_n '} (2\eta_n -\eta_{n-1})  \partial_{\eta_{n,j}} \widehat{f_{k_{n+1}}}(s,\eta_n - \eta_{n-1}) e^{is(\vert \xi \vert^2 - \vert \eta_{n-1} - \eta_n \vert ^2 - \vert \eta_n \vert^2)} \widehat{f_{k_{n}}}(s,\eta_n) d\eta_n d\eta_{n-1} ds \\
\notag &  + \lbrace \textrm{similar terms} \rbrace.
\end{align}
By similar terms we mean the $n-$th iterates of the terms obtained in Lemma \ref{bilinhard}. \\
We saw that the three terms were treated following roughly the same strategy. Therefore here we focus on \eqref{space1bis} only. \\
We do the usual change of variables to obtain
\begin{align*}
\Vert \eqref{space1bis} \Vert_{L^{\infty}_t L^2 _x}  &=  \Bigg \Vert \int_{1.1^m} ^{1.1^{m+1}} \int_{\big(\mathbb{R}^3 \big)^{n-1}} \prod_{l=1}^{n-1} \frac{\widehat{V'}(s,\xi_l - \xi_{l-1}) P_{k_l}(\xi_l+\eta) P_{k}(\eta)}{\vert \eta \vert^2 - \vert \xi_l+\eta \vert^2} d\xi_{1} ... d\xi_{n-2} \\
& \times \frac{\eta_l (\xi_{n-1}+\eta)_j}{\vert \xi_{n-1}+\eta \vert^2} \int_{\mathbb{R}^3} 
\widehat{f_{k_{n}}}(s,\xi_n+\eta) \frac{P_{k_n'}(2\xi_n - \xi_{n-1}+\eta)(\xi_{n}+\eta)(2\xi_n - \xi_{n-1}+\eta)_j}{\vert 2\xi_n - \eta_{n-1}+\eta \vert^2} \\
& \times  e^{is(\vert \eta \vert^2 - \vert \xi_{n-1} - \xi_n \vert ^2 - \vert \xi_n+\eta \vert^2)} \partial_{\xi_{n,j}} \widehat{f_{k_{n+1}}}(s,\xi_n - \xi_{n-1}) d\xi_n d\xi_{n-1} ds \Bigg \Vert_{L^{\infty}_t L^2 _x}.
\end{align*}
Using Strichartz estimates we can write:
\begin{align*}
\Vert \eqref{space1bis} \Vert_{L^{\infty}_t L^2 _x}  & \lesssim \Bigg \Vert \int_{1.1^m} ^{1.1^{m+1}} e^{is \Delta} \Bigg(\mathcal{F}^{-1}_{\eta} \int_{\mathbb{R}^3} \widehat{K_{n-1}}(\xi_{n-1},\eta) \int_{\mathbb{R}^3}  \partial_{\xi_{n,j}} \widehat{f_{k_{n+1}}}(s,\xi_n-\xi_{n-1}) \\
& \times m(\xi_{n-1},\xi_n,\eta) \widehat{f_{k_{n}}}(s,\xi_n+\eta) d\xi_n d\xi_{n-1} \Bigg) ds \Bigg \Vert_{L^{\infty}_t L^2 _x}   \\
& \lesssim \Bigg \Vert \mathcal{F}^{-1}_{\eta} \int_{\mathbb{R}^3} \widehat{K_{n-1}}(\xi_{n-1},\eta) \int_{\mathbb{R}^3} \partial_{\xi_{n,j}}\widehat{f_{k_{n+1}}}(s,\xi_n-\xi_{n-1}) \\
& \times m(\xi_{n-1},\xi_n,\eta) \widehat{f_{k_{n}}}(s,\xi_n+\eta) d\xi_n d\xi_{n-1} \Bigg \Vert_{L^{4/3}_t L^{3/2}_x}
\end{align*}
with
\begin{align*}
&m(\xi_{n-1}, \xi_n , \eta) \\
 =& \frac{P_k(\eta) (2 \xi_n - \xi_{n-1} + \eta)_{j} P_{k_n '} (2 \xi_n - \xi_{n-1} + \eta) P_{k_{n-1}} (\xi_{n-1}+\eta) P_{k_{n+1}}(\xi_n + \eta) \eta_l (\xi_{n-1}+ \eta)_j (\xi_n + \eta)_j}{\vert 2 \xi_n - \xi_{n-1} + \eta \vert^2 \vert \xi_{n-1}+ \eta \vert^2}.
\end{align*} 
Now we can apply Lemma \ref{nbilin} as well as Lemma \ref{X'} to write that the term in question is bounded by 
\begin{align*}
\Vert \eqref{space1bis} \Vert_{L^{\infty}_t L^2 _x} \lesssim C_0^{n-2} \bigg( \prod_{j \in I(n)} \min \lbrace 1.1^{-k_j/2} ; 1.1^{k_j/2} \rbrace \bigg) \delta^{n-1} 1.1^{k_n -k_n '}  1.1^{-m/4} \varepsilon_1 ^2
\end{align*}
Then we can sum on $k_n' \geqslant k_n -50, $ then sum on $k_n, k_{n+1}$ (since there are $O(m^2)$ terms in the sum) and finally  on $m.$ 
\\
\\
\underline{Case 3.3: $-10 m \leqslant k_n ' \leqslant k_n -50.$} \\
In this case, we can start, as in the proof of Lemma \ref{bilinhard}, by noticing that $k_n \sim k_{n+1}.$ \\
Using the same reasoning as in case 1 and 2, we can adapt the proof of the reduction to $k_n > -\frac{101}{224}m$ from Lemma \ref{bilinhard}. \\
\\
Now notice a significant difference with Lemma \ref{bilinhard}, namely the fact that the phase in the inner integral depends on three variables ($\xi,$ $\eta_{n-1}$ and $\eta_{n}$) instead of simply $\xi$ and $\eta.$ However remember that terms of this type only appear when $\vert k - k_{n-1} \vert \leqslant 1$ implying $\vert \xi \vert \simeq \vert \eta_{n-1} \vert.$ Therefore the usual strategy of integrating by parts in time will still work. \\
\\
More precisely, let's start by integrating by parts in time: 
\begin{align*}
&\mathcal{F} I_{6,m,k_n,k_{n-1},k_n '} ^n f  = \\
&-\int_{1.1^m} ^{1.1^{m+1}} \int_{\big(\mathbb{R}^3 \big)^{n-1}} \frac{\partial_s\widehat{V}(s,\xi - \eta_1) P_{k_1}(\eta_1) P_{k}(\xi)}{\vert \xi \vert^2 - \vert \eta_1 \vert^2} \prod_{l=2}^{n-1} \frac{\widehat{V}(s,\eta_{l-1}-\eta_l) P_{k_l}(\eta_l) P_{k}(\xi)}{\vert \xi \vert^2 - \vert \eta_l \vert^2} d\eta_{1} ... d\eta_{n-2} \frac{\xi_l \eta_{n-1,j}}{\vert \eta_{n-1} \vert^2} \\
&\times \int_{\mathbb{R}^3} \frac{s \eta_n \widehat{f_{k_{n+1}}}(s,\eta_{n-1}-\eta_n)P_{k_n'}(2\eta_n - \eta_{n-1}) }{\vert \xi \vert^2 - \vert \eta_n - \eta_{n-1} \vert^2 - \vert \eta_n \vert^2}  e^{is(\vert \xi \vert^2 - \vert \eta_{n-1} - \eta_n \vert ^2 - \vert \eta_n \vert^2)}  \widehat{f_{k_{n}}}(s,\eta_n) d\eta_n d\eta_{n-1} ds \\
& - \int_{1.1^m} ^{1.1^{m+1}} \int_{\big(\mathbb{R}^3\big)^{n-1}} \prod_{l=1}^{n-1} \frac{\widehat{V}(s,\eta_{l-1}-\eta_l) P_{k_l}(\eta_l) P_{k}(\xi)}{\vert \xi \vert^2 - \vert \eta_l \vert^2} d\eta_{1} ... d\eta_{n-2} \frac{\xi_l \eta_{n-1,j}}{\vert \eta_{n-1} \vert^2} \\
&\times \int_{\mathbb{R}^3} e^{is(\vert \xi \vert^2 - \vert \eta_{n-1} - \eta_n \vert ^2 - \vert \eta_n \vert^2)} \frac{P_{k_n'}(2\eta_n - \eta_{n-1}) s \partial_s \widehat{f_{k_{n+1}}}(s,\eta_{n-1}-\eta_n) \eta_n}{\vert \xi \vert^2 - \vert \eta_n - \eta_{n-1} \vert^2 - \vert \eta_n \vert^2} \widehat{f_{k_{n}}}(s,\eta_n) d\eta_n d\eta_{n-1} ds \\
& + \int_{\big(\mathbb{R}^3 \big)^{n-1}} \prod_{l=1}^{n-1} \frac{\widehat{V}(1.1^m ,\eta_{l-1}-\eta_l) P_{k_l}(\eta_l) P_{k}(\xi)}{\vert \xi \vert^2 - \vert \eta_l \vert^2} d\eta_{1} ... d\eta_{n-2} \frac{\xi_l \eta_{n-1,j}}{\vert \eta_{n-1} \vert^2} \\
&\times \int_{\mathbb{R}^3} e^{i 1.1^m (\vert \xi \vert^2 - \vert \eta_{n-1} - \eta_n \vert ^2 - \vert \eta_n \vert^2)} \frac{1.1^m P_{k_n'}(2\eta_n - \eta_{n-1}) \widehat{f_{k_{n+1}}}(1.1^m ,\eta_{n-1}-\eta_n) \eta_n}{\vert \xi \vert^2 - \vert \eta_n - \eta_{n-1} \vert^2 - \vert \eta_n \vert^2} \widehat{f_{k_{n}}}(1.1^m,\eta_n) d\eta_n d\eta_{n-1} \\
&\notag + \lbrace \textrm{similar terms} \rbrace
\end{align*}
and by similar terms we refer to the case where the partial derivative in $s$ hits the other $f$, or the other $V$'s (not the first one). It also includes the second boundary term, which has the exact same form as the boundary term that has been explicitely written. \\
Given the similarities, we will only estimate the terms written explicitely here. 
\\
\\
\textit{A useful symbol bound:} \\
For the expressions above to have good estimates, we must prove a lower bound on the denominator: \\
We have, given our definition of Littlewood-Paley projections (see notations section in the introduction) as well as the fact that $\vert k-k_{n-1} \vert \leqslant 1:$
\begin{align*}
\vert \xi \vert \geqslant \frac{1.1^k}{1.04} \geqslant 1.1^{k_{n-1}-1}\frac{1}{1.04} = 1.1^{k_{n-1}} \times 1.04 \times \frac{1}{1.1 \times(1.04)^2} \geqslant \frac{\vert \eta_{n-1} \vert}{1.1 \times 1.04 ^2}. 
\end{align*}
Using this fact together with $\vert 2 \eta_n - \eta_{n-1} \vert \sim 1.1^{k_n '}$ and $k_n' \leqslant k_n -50,$ we can write
\begin{align*}
\vert \xi \vert^2 - \vert \eta_n - \eta_{n-1} \vert^2 - \vert \eta_n \vert^2 & \geqslant \frac{\vert \eta_{n-1} \vert^2}{1.1^2 (1.04)^4} - 2 \vert \eta_n \vert^2 - \vert 2 \eta_n - \eta_{n-1} \vert^2 + (2 \eta_n - \eta_{n-1}) \cdot \eta_n \\
& \geqslant \frac{\vert \eta_{n-1} \vert^2}{1.42} - 2 \vert \eta_n \vert^2 - (1.1)^{-100} \vert \eta_n  \vert^2 - 1.1^{-50} \vert \eta _n \vert^2 \\
& \geqslant \frac{4 \vert \eta_n \vert^2}{1.1^2} + \frac{\vert \eta_{n-1} - 2\eta_n \vert^2}{1.1^2} + \frac{(\eta_{n-1}-2\eta_n) \cdot (2\eta_n)}{1.1^2} - 2.001   \vert \eta_n \vert^2 \\
& \geqslant 2.8 \vert \eta_n \vert^2 - \frac{2 (1.1)^{-50} \vert \eta_n \vert^2}{1.1} - 2.0012   \vert \eta_n \vert^2 \\
& \geqslant 0.7 \vert \eta_n \vert^2.
\end{align*}
This is to carry out this computation that we required localizations of frequencies at $1.1^k$ and not $2^k.$ \\
In particular this shows that denominator that appears above is not singular. \\ 
\\
\textit{Estimating $I_{6,m,k_n,k_{n+1},k_n '} ^n$}
We wish to estimate these terms in $L^2_x$ therefore we do the usual change of variables. In the end we must estimate the following three terms in $L^{\infty}_t L^2 _x:$
\begin{align}
&\label{dspot} \int_{1.1^m} ^{1.1^{m+1}} \int_{\big(\mathbb{R}^3 \big)^{n-1}} \frac{\partial_s \widehat{V'}(s,\xi_1) P_{k_l}(\xi_1+\eta) P_{k}(\eta)}{\vert \xi_1+\eta \vert^2-\vert \eta \vert^2} \prod_{l=2}^{n-1} \frac{\widehat{V'}(s,\xi_l - \xi_{l-1}) P_{k_l}(\xi_l+\eta) P_{k}(\eta)}{\vert \xi_l+\eta \vert^2 - \vert \eta \vert^2} d\xi_{1} ... d\xi_{n-2} \\
&\notag \times \int_{\mathbb{R}^3} s m(\xi_{n-1},\xi_n,\eta) \widehat{f_{k_{n+1}}}(s,\xi_n - \xi_{n-1}) e^{is(\vert \eta \vert^2 - \vert \xi_{n-1} - \xi_n \vert ^2 - \vert \xi_n+\eta \vert^2)}  \widehat{f_{k_{n}}}(s,\xi_n+\eta) d\xi_n d\xi_{n-1} ds \\
&\label{dsf} \int_{1.1^m} ^{1.1^{m+1}} \int_{\big(\mathbb{R}^3 \big)^{n-1}} \prod_{l=1}^{n-1} \frac{\widehat{V'}(s,\xi_l - \xi_{l-1}) P_{k_l}(\xi_l+\eta) P_{k}(\eta)}{\vert \xi_l+\eta \vert^2 - \vert \eta \vert^2} d\xi_{1} ... d\xi_{n-2} \\
&\notag\times \int_{\mathbb{R}^3} s m(\xi_{n-1},\xi_n,\eta) \partial_s \widehat{f_{k_{n+1}}}(s,\xi_{n-1}-\xi_n) e^{is(\vert \eta \vert^2 - \vert \xi_{n-1} - \xi_n \vert ^2 - \vert \xi_n+\eta \vert^2)}  \widehat{f_{k_{n}}}(s,\xi_n+\eta) d\xi_n d\xi_{n-1} ds \\
&\label{dsboundary}\int_{\big(\mathbb{R}^3 \big)^{n-1}} \prod_{l=1}^{n-1} \frac{\widehat{V'}(1.1^m,\xi_l - \xi_{l-1}) P_{k_l}(\xi_l+\eta) P_{k}(\eta)}{\vert \xi_l+\eta \vert^2 - \vert \eta \vert^2} d\xi_{1} ... d\xi_{n-2} \\
&\notag \times \int_{\mathbb{R}^3} 1.1^m m(\xi_{n-1},\xi_n,\eta) \widehat{f_{k_{n+1}}}(1.1^m,\xi_{n-1}-\xi_n) e^{i 1.1^m(\vert \eta \vert^2 - \vert \xi_{n-1} - \xi_n \vert ^2 - \vert \xi_n+\eta \vert^2)}  \widehat{f_{k_{n}}}(1.1^m,\xi_n+\eta) d\xi_n d\xi_{n-1}
\end{align}
with
\begin{align*}
m(\xi_{n-1}, \xi_n , \eta) = \frac{P_k(\eta) P_{k_n}'(2\xi_n - \xi_{n-1} + \eta) P_{k_{n-1}} (\xi_{n-1}+\eta) P_{k_{n+1}}(\xi_n + \eta) \eta_n (\xi_{n-1}+ \eta)_j (\xi_n + \eta)_j}{\vert \xi_{n-1}+ \eta \vert^2 (\vert \eta \vert^2 - \vert \xi_n - \xi_{n-1} \vert^2 - \vert \xi_n +\eta \vert^2)}.
\end{align*}
From the computation carried out above as well as the restriction $k_n >-101/224$ and Lemma \ref{symbol}, we have that
\begin{align*}
\Vert \check{m} \Vert_{L^1} \lesssim 1.1^{-k_n} \lesssim 1.1^{101/224m}.
\end{align*}
Now we can apply Lemma \ref{nbilin} as well as dispersive estimates to obtain that
\begin{align*}
\Vert (\ref{dspot}) \Vert_{L^{\infty}_t L^2 _x} & \lesssim C_0 ^{n-2} \bigg( \prod_{j \in J(n-1)} \min \lbrace 1.1^{k_j/2} ; 1.1^{-k_j/2} \rbrace \bigg) \Vert  \langle x \rangle \partial_t V \Vert_{L^1 _t B_x} \Vert \langle x \rangle V \Vert_{L^{\infty}_t B_x}^{n-2} \\
& \times 1.1^{-k_n} \Vert t e^{it\Delta} f_{k_n} \Vert_{L^{\infty}_t L^6 _x} \Vert e^{it\Delta} f_{k_{n+1}} \Vert_{L^{\infty}_t L^3_x} \\
                                                & \lesssim C_0^{n-2} \bigg( \prod_{j \in J(n-1)}  \min \lbrace 1.1^{k_j/2} ; 1.1^{-k_j/2} \rbrace \bigg) \delta^{n-1} 1.1^{-k_n} 1.1^{-m/2} \varepsilon_1 ^2  \\
                                                & \lesssim C_0^{n-2} \bigg( \prod_{j \in J(n-1)}  \min \lbrace 1.1^{k_j/2} ; 1.1^{-k_j/2} \rbrace \bigg)  \delta^{n-1} 1.1^{101/224 m} 1.1^{-m/2} \varepsilon_1 ^2 
\end{align*}
which can be summed on $k_j \in J(n-1),$ then on $k_n, k_n'$ and $k_{n+1}$ (there are $O(m^3)$ such terms) and finally on $m.$\\
Similarly we can straightforwardly adapt the proof from Lemma \ref{bilinhard} for \eqref{dsboundary} using Lemma \ref{nbilin} to obtain:
\begin{align*}
\Vert (\ref{dsboundary}) \Vert_{L^{\infty}_t L^2 _x} & \lesssim C_0^{n-2} \bigg( \prod_{j \in J(n-1)}  \min \lbrace 1.1^{k_j/2} ; 1.1^{-k_j/2} \rbrace \bigg) \delta^{n-1} 1.1^{101/224 m} 1.1^{-m/2} \varepsilon_1 ^2
\end{align*}
which can be summed.
\\
Now we have to deal with (\ref{dsf}). Using equation \eqref{NLSV}, we find that $\partial_s \widehat{f_{k_n}}(s,\cdot) = e^{is \vert \cdot \vert^2} (\widehat{(u^2)_{k_n}} + \widehat{(Vu)_{k_n}})(\cdot).$ \\
Therefore there are two pieces to estimate:
\begin{align}
\notag \eqref{dsf} &= \\ 
&\label{dsfb} \int_{1.1^m} ^{1.1^{m+1}} \int_{\big(\mathbb{R}^3 \big)^{n-1}} \prod_{l=1}^{n-1} \frac{\widehat{V'}(s,\xi_l - \xi_{l-1}) P_{k_l}(\xi_l+\eta) P_{k}(\eta)}{\vert \xi_l+\eta \vert^2 - \vert \eta \vert^2} d\xi_{1} ... d\xi_{n-2} \\
&\notag \times \int_{\mathbb{R}^3} s m(\xi_{n-1},\xi_n,\eta) \widehat{(u^2)_{k_{n+1}}} (\xi_{n-1}-\xi_n) e^{-is \vert \xi_n+\eta \vert^2}  \widehat{f_{k_{n}}}(s,\xi_n+\eta) d\xi_n d\xi_{n-1} ds \\
&\label{dsfv} + \int_{1.1^m} ^{1.1^{m+1}} \int_{\big(\mathbb{R}^3 \big)^{n-1}} \prod_{l=1}^{n-1} \frac{\widehat{V'}(s,\xi_l - \xi_{l-1}) P_{k_l}(\xi_l+\eta) P_{k}(\eta)}{\vert \xi_l+\eta \vert^2 - \vert \eta \vert^2} d\xi_{1} ... d\xi_{n-2} \\
&\notag \times \int_{\mathbb{R}^3} s m(\xi_{n-1},\xi_n,\eta) \widehat{(Vu)_{k_{n+1}}} (\xi_n - \xi_{n-1}) e^{-is\vert \xi_n+\eta \vert^2}  \widehat{f_{k_{n}}}(s,\xi_n+\eta) d\xi_n d\xi_{n-1} ds .
\end{align}
\textit{Estimating \eqref{dsfv}:} This is the analog of the term $I$ in Lemma \ref{bilinhard}. Given the similarity we only sketch the proof. \\
We use Strichartz estimates, Lemma \ref{nbilin} with multiplier $m$ and $p=2, q=6, r=3/2$ to bound that term by
\begin{align*}
\Vert \eqref{dsfv} \Vert_{L^{\infty}_t L^2_x} & \lesssim C_0^{n-2} \bigg( \prod_{j \in J(n-1)}  \min \lbrace 1.1^{k_j/2} ; 1.1^{-k_j/2} \rbrace \bigg) \delta^{n-1} 1.1^{-k_n} \Vert t (Vu) \Vert_{L^{\infty}_t L^2_x} \Vert u \Vert_{L^{4/3} _t L^6 _x} \\
& \lesssim C_0^{n-2} \bigg( \prod_{j \in J(n-1)}  \min \lbrace 1.1^{k_j/2} ; 1.1^{-k_j/2} \rbrace \bigg) \delta^{n-1} 1.1^{-m/4} 1.1^{0.01 m} \delta \varepsilon_1 ^2
\end{align*}
which yields the desired result.
\\
\textit{Estimating \eqref{dsfb}:} This is the analog of the term $II$ in Lemma \ref{bilinhard}. Given the similarity we only sketch the proof. \\
We use Strichartz estimates, Lemma \ref{nbilin} with multiplier $m$ and $p=2, q=6, r=3/2$ to bound that term by
\begin{align*}
\Vert \eqref{dsfb} \Vert_{L^{\infty}_t L^2 _x}& \lesssim C_0^{n-2} \bigg( \prod_{j \in J(n-1)}  \min \lbrace 1.1^{k_j/2} ; 1.1^{-k_j/2} \rbrace \bigg) \delta^{n-1} 1.1^m 1.1^{-k_n} \Vert u^2 \Vert_{L^{\infty}_t L^2_x} \Vert u \Vert_{L^{4/3} _t L^6 _x} \\
& \lesssim C_0^{n-2} \bigg( \prod_{j \in J(n-1)}  \min \lbrace 1.1^{k_j/2} ; 1.1^{-k_j/2} \rbrace \bigg) \delta^{n-1} 1.1^m 1.1^{99/224 m} 1.1^{-1.49 m} 1.1^{-m/4} \varepsilon_1 ^3
\end{align*}
which yields the result.
\end{proof}

\section{End of proofs of theorems \ref{mainthm} and \ref{mainresultscattering}} \label{end}
In this section we finish the proofs of the main Theorems. The steps that are left are simpler than above, since the use of the space-time resonance theory is not needed. First we finish the proof of Theorem \ref{mainthm} by concluding the bootstrap argument initiated in Subsection \ref{boot}. 

Then we prove the scattering statement of Theorem \ref{mainresultscattering} in a second subsection. This is essentially a corollary of previous expansions and estimates. 

\subsection{Bounding the $H^{10}$ norm} \label{energy}
This subsection is dedicated to the proof of \eqref{bootstrapconcl1}.
\subsubsection{Set-up}
The proof will be similar in spirit to the control of the $X-$norm: we write the solution as a convergent series whose terms are bounded in $H^{10}.$ There is an additional $\delta$ factor that guarantees that the series converges. However the reasoning is simpler, in the sense that we do not need the theory of space-time resonances.
\begin{remark}[Notations]
Throughout this section, we will estimate $H^{10}$ norms and write systematically
\begin{align*}
\Vert f \Vert_{H^{10}} \sim \sum_{k = 0}^{+\infty} 1.1^{10k} \Vert P_k (\xi) \mathcal{F} f \Vert_{L^2}.
\end{align*}
Note that there is a slight abuse of notations here since we write $P_0(\xi)$ for $P_{\leqslant 0}(\xi).$ 
\end{remark}
We can now start the proof of \eqref{bootstrapconcl1}. We start with Duhamel's formula:
\begin{align*}
\widehat{f}(t,\xi) &= e^{i \vert \xi \vert^2} \widehat{u_1}(\xi) -\frac{i}{ (2\pi)^3} \int_1 ^t e^{is \vert \xi \vert ^2} \int_{\mathbb{R}^3} e^{-is \vert \xi-\eta_1 \vert^2} e^{-i s \vert \eta_1 \vert^2} \widehat{f}(s,\eta_1) \widehat{f}(s,\xi-\eta_1) d\eta_1 ds \\
\notag&-\frac{i}{(2\pi)^3 } \int_1 ^t e^{is \vert \xi \vert^2} \int_{\mathbb{R}^3} \widehat{V}(\xi - \eta_1) e^{-is \vert \eta_1 \vert ^2} \widehat{f}(s,\eta_1) d\eta_1 ds .
\end{align*}
\subsubsection{Bound on first bilinear iterate}
We start by estimating the $H^{10}$ norm of the bilinear term in the next lemma. We start by localizing it in frequency (denoting $k$ the corresponding exponent), and discretizing it in the time variable (as usual we denote $m$ the exponent). Therefore we introduce the following terms:
\begin{align*}
J_{k,m} f := \mathcal{F}^{-1}_{\xi} P_k(\xi) \int_{1.1^m} ^{1.1^{m+1}} \int_{\mathbb{R}^3} e^{is(\vert \xi \vert^2 - \vert \eta_1 \vert^2 - \vert \xi - \eta_1 \vert^2)} \widehat{f}(s,\eta_1) \widehat{f}(s,\xi-\eta_1) d\eta_1 ds.
\end{align*}
The desired control over the bilinear term above will follow from the lemma:
\begin{lemma}\label{H10bilin}
There exists $a>0$ such that for all $m,$ we have the bound
\begin{align*}
\sum_{k = 0}^{+\infty} 1.1^{10 k} \Vert \mathcal{F} J_{k, m} f \Vert_{L^{\infty}_t L^2 _x} \lesssim 1.1^{-am} \varepsilon_1 ^2 .
\end{align*}
\end{lemma}
\begin{proof}
We split frequencies $\eta_1$ and $\xi-\eta_1$  dyadically and denote $k_1, k_2$ the corresponding exponents. Now notice that 
\begin{align*}
1.1^{k-1} \leqslant \vert \xi \vert \leqslant \vert \xi -\eta_1 \vert + \vert \eta_1 \vert \leqslant 1.1^{\max \lbrace k_1,k_2 \rbrace+10}
\end{align*}
Therefore $k \leqslant 11 + \max \lbrace k_1 ,k_2 \rbrace.$
\\
In  what follows we denote $k_{min} := \min \lbrace k_1 ,k_2 \rbrace$ and $k_{max} := \max \lbrace k_1 ,k_2 \rbrace. $ \\
Now we can bound the term above using Strichartz estimates, Lemma \ref{bilin}, H\"{o}lder's inequality , Bernstein's inequality and the energy bound:
\begin{align*}
&1.1^{10 k} \Bigg \Vert P_k (\xi) \int_{1.1^m} ^{1.1^{m+1}} \int_{\mathbb{R}^3} e^{is(\vert \xi \vert^2 - \vert \eta_1 \vert^2 - \vert \xi - \eta_1 \vert^2)} \widehat{f_{k_1}}(s,\eta_1) \widehat{f_{k_2}}(s,\xi-\eta_1) d\eta_1 \Bigg \Vert_{L^{\infty}_t L^2 _x} \\
& \lesssim 1.1^{10k} \Vert \textbf{1}_{[1.1^m;1.1^{m+1}]}(t) \big( e^{it\Delta} f_{k_1} \big) \big(e^{it\Delta} f_{k_2} \big) \Vert_{L^{4/3}_t L^{3/2}_x} \\
& \lesssim 1.1^{10k} \Vert \textbf{1}_{[1.1^m;1.1^{m+1}]}(t) e^{it\Delta} f_{k_{min}} \Vert_{L^{4/3}_t L^6 _x} \Vert f_{k_{max}} \Vert_{L^{\infty}_t L^2 _x} \\
& \lesssim 1.1^{10(k-k_{max})} \bigg \Vert \textbf{1}_{[1.1^m;1.1^{m+1}]}(t) \Vert \vert e^{it\Delta} f_{k_{min}} \vert ^{9/10} \Vert_{L^{60/7}_x} \Vert \vert e^{it\Delta} f_{k_{min}} \vert ^{1/10} \Vert_{L^{20} _x} \bigg \Vert_{L^{4/3}_t} 1.1^{10 k_{max}} \Vert f_{k_{max}} \Vert_{L^{\infty}_t L^2 _x} \\
& \lesssim  1.1^{10(k-k_{max})} \big \Vert \textbf{1}_{[1.1^m;1.1^{m+1}]}(t) \Vert e^{it\Delta} f_{k_{min}} \Vert_{L^{54/7}_x}^{9/10} \big \Vert_{L^{4/3}_t} \Vert f_{k_{min}} \Vert_{L^{\infty}_t L^2 _x}^{1/10}  1.1^{10 k_{max}} \Vert f_{k_{max}} \Vert_{L^{\infty}_t L^2 _x}  \\
& \lesssim 1.1^{10(k-k_{max})} \bigg(\int_{1.1^m} ^{1.1^{m+1}} \frac{1}{s^{4/3}} ds \bigg)^{3/4} \min \lbrace 1.1^{-k_{min}} ; 1.1^{\frac{k_{min}}{20}} \rbrace  1.1^{10 k_{max}} \Vert f_{k_{max}} \Vert_{L^{\infty}_t L^2 _x} \varepsilon_1 .
\end{align*}
Now we can sum over $k,$ and then $k_{max}$ and $k_{min}$ to get the desired result. 
\end{proof}
\subsubsection{Series expansion}
Now we move on to the potential part. As mentioned above, we will a derive a series representation of this term. First we multiply by $P_k (\xi)$ and split frequencies dyadically in $\eta_1.$ Then we integrate by parts in time to obtain
\begin{align}
\notag& \frac{-i}{(2\pi)^3} P_k (\xi)\int_1 ^t \int_{\mathbb{R}^3} e^{is(\vert \xi \vert^2 - \vert \eta_1 \vert^2)} \widehat{f_{k_1}}(s,\eta_1) \widehat{V}(s,\xi-\eta_1) d\eta_1 ds \\
&\label{Rb1}=  \frac{-i}{(2\pi)^3} P_k(\xi) \int_{\mathbb{R}^3} e^{it(\vert \xi \vert^2 - \vert \eta_1 \vert^2)} \frac{\widehat{V}(t,\xi-\eta_1)}{\vert \xi \vert^2 - \vert \eta_1 \vert^2} \widehat{f_{k_1}}(t,\eta_1) d\eta_1 \\
& \label{Rb1prime} + \frac{i}{(2\pi)^3} P_k(\xi) \int_{\mathbb{R}^3} e^{i(\vert \xi \vert^2 - \vert \eta_1 \vert^2)} \frac{\widehat{V}(1,\xi-\eta_1)}{\vert \xi \vert^2 - \vert \eta_1 \vert^2} \widehat{f_{k_1}}(1,\eta_1) d\eta_1 \\
&\label{Rbdt} + \frac{i}{(2\pi)^3} \int_1 ^t P_k(\xi) \int_{\mathbb{R}^3} e^{is(\vert \xi \vert^2-\vert \eta_1 \vert^2)} \frac{\partial_s \widehat{V}(s,\xi-\eta_1)}{\vert \xi \vert^2 - \vert \eta_1 \vert^2} \widehat{f_{k_1}}(s,\eta_1) d\eta_1 ds \\ 
&\label{Bb1}+\frac{1}{(2\pi)^6} \int_1 ^t P_k(\xi) \int_{\mathbb{R}^3}  \frac{\widehat{V}(s,\xi-\eta_1)}{\vert \xi \vert^2 - \vert \eta_1 \vert^2} P_{k_1} (\eta_1) \\
\notag& \times \int_{\mathbb{R}^3} \widehat{f}(s,\eta_1-\eta_2) e^{is(\vert \xi \vert^2- \vert \eta_1 - \eta_2 \vert^2- \vert \eta_2 \vert^2)} \widehat{f}(s,\eta_2) d\eta_2 d\eta_1 ds \\
&\notag +\frac{1}{(2\pi)^6} \int_1 ^t P_k(\xi) \int_{\mathbb{R}^3} \frac{\widehat{V}(s,\xi-\eta_1)}{\vert \xi \vert^2 - \vert \eta_1 \vert^2} P_{k_1} (\eta_1) \\
& \notag \times \underbrace{ \int_{\mathbb{R}^3} \widehat{V}(s,\eta_1-\eta_2) e^{is(\vert \xi \vert^2- \vert \eta_2 \vert^2)} \widehat{f}(s,\eta_2) d\eta_2 d\eta_1 ds.}_{:=\mathcal{T}_{k_1}^1}
\end{align}
Then we repeat the process by integrating by parts in time in $\mathcal{T}_{k_1}^1$, similarly to what has been done to bound the $X$ norm above. \\
At the $n-$th step of the iteration we will obtain the following terms:
\\
The $n-$th iterate of \eqref{Rb1}
\begin{align*}
\mathcal{F} I_{7} ^n f(t,\xi) & := -\frac{i^n}{(2\pi)^{3n}} \int_{\big( \mathbb{R}^3 \big)^{n-1}} \prod_{\gamma=1}^{n-1} \frac{\widehat{V}(t,\eta_{\gamma} - \eta_{\gamma-1})P_k(\xi) P_{k_{\gamma}}(\eta_{\gamma})}{\vert \xi \vert^2 - \vert \eta_{\gamma} \vert^2} d\eta_{1} ... d\eta_{n-2} \\
& \times \int_{\mathbb{R}^3} \frac{\widehat{V}(t,\eta_n - \eta_{n-1})}{\vert \xi \vert^2 - \vert \eta_n \vert^2} e^{-it(\vert \xi \vert^2-  \vert \eta_n \vert^2)} \widehat{f} (t,\eta_n) d\eta_n d\eta_{n-1}.
\end{align*} 
There is also the $n-$th iterate of \eqref{Bb1}:
\begin{align*}
\mathcal{F} I_{8} ^n f(t,\xi) & := - \frac{i^{n+1}}{(2\pi)^{3(n+1)}} \int_1 ^t  \int_{\big(\mathbb{R}^{3} \big)^{n-1}} \prod_{\gamma=1}^{n-1} \frac{\widehat{V}(s,\eta_{\gamma-1}-\eta_{\gamma}) P_{k_{\gamma}}(\eta_{\gamma})P_k(\xi)}{\vert \xi \vert^2 - \vert \eta_{\gamma} \vert^2} d\eta_{1} ... d\eta_{n-2} \\
& \times \int_{\mathbb{R}^3} \widehat{f}(s, \eta_{n-1}-\eta_n) e^{-is(\vert \xi \vert^2- \vert \eta_n - \eta_{n-1} \vert^2 - \vert \eta_n \vert^2)} \widehat{f} (s,\eta_n) d\eta_n d\eta_{n-1} ds.
\end{align*}
Finally we introduce the $n-$th iterates of \eqref{Rbdt}, for $l \in \lbrace 1;...;n-1 \rbrace$:
\begin{align*}
\mathcal{F} I_9 ^{n,l} f(t,\xi) & :=  \frac{i^n}{(2\pi)^{3n}} \int_1 ^t \int_{\big(\mathbb{R}^3 \big)^n} \prod_{\gamma=1, \gamma \neq l}^{n-1} \frac{\widehat{V}(s,\eta_{\gamma} - \eta_{\gamma-1})P_k(\xi) P_{k_{\gamma}}(\eta_{\gamma})}{\vert \xi \vert^2 - \vert \eta_{\gamma} \vert^2} \\
& \times \frac{\partial _s \widehat{V}(s,\eta_{l} - \eta_{l-1})P_k(\xi) P_{k_{l}}(\eta_{\gamma})}{\vert \xi \vert^2 - \vert \eta_{l} \vert^2} d\eta_{1} ... d\eta_{n-1}  e^{-is(\vert \xi \vert^2-  \vert \eta_n \vert^2)} \widehat{f} (s,\eta_n) d\eta_n  ds.
\end{align*}
\subsubsection{Estimates on iterates} 
In this subsection we estimate the above iterates. We start with $I_7 ^n f$ in the next lemma:
\begin{lemma} \label{H10pot}
Let $J(n):= \lbrace i \in \lbrace 1;...;n \rbrace , k < k_i-1 \rbrace .$ \\  
Then there exists a constant $C_1>0$ independent of $n$ such that
\begin{align*}
\sum_{k =0}^{+\infty} 1.1^{10k} \Vert I_7 ^n f \Vert_{L^{\infty}_t L^2_x } & \lesssim C_1 ^{n-2} \prod_{j \in J(n-1)} \min \lbrace 1.1^{k_j /2} ; 1.1^{-k_j /2} \rbrace \varepsilon_1 \delta^n . 
\end{align*}
Moreover for any $l \in \lbrace 1;...;n \rbrace,$ 
\begin{align*}
\sum_{k =0}^{+\infty} 1.1^{10k} \Vert I_9 ^{n,l} f \Vert_{L^{\infty}_t L^2_x } & \lesssim C_1 ^{n-2} \prod_{j \in J(n-1)} \min \lbrace 1.1^{k_j /2} ; 1.1^{-k_j /2} \rbrace \varepsilon_1 \delta^{n-1} \int_1 ^t \Vert \partial_s V \Vert_{B_x '} ds \\
&  \lesssim  C_1 ^{n-2} \prod_{j \in J(n-1)} \min \lbrace 1.1^{k_j /2} ; 1.1^{-k_j /2} \rbrace \varepsilon_1 \delta^{n}.
\end{align*}
The implicit constant here is also independent of $n$ (it is the same implicit constant in both inequalities). 
\end{lemma} 
\begin{proof}
We use the notation $\widetilde{C_0} := \frac{C_0}{(2 \pi)^3}.$ \\
Let's start by decomposing dyadically on $\eta_n$. Let's denote $k_n$ the corresponding exponent. Then there are several cases:
\\
\underline{Case 1: $\vert k_n - k \vert \leqslant 1 $} \\
In this case we can essentially repeat the proof of Lemma \ref{nthmodel2} and bound the above term by
\begin{align*}
1.1^{10k} \Vert I_{7} ^n f \Vert_{L^2 _x} & \lesssim \widetilde{C_0}^{n-2} \prod_{j \in J(n)} \min \lbrace 1.1^{k_j /2} ; 1.1^{-k_j /2} \rbrace \delta^n 1.1^{10k} \Vert \widehat{f_{k_n}} \Vert_{L^2} \\
& \lesssim \widetilde{C_0}^{n-2} \prod_{j \in J(n)} \min \lbrace 1.1^{k_j /2} ; 1.1^{-k_j /2} \rbrace \delta^n 1.1^{10(k-k_n)} 1.1^{10k_n} \Vert \widehat{f_{k_n}} \Vert_{L^2}.
\end{align*}
We can sum that last term in $k$ and then in $k_n.$ The desired bound follows. \\
\underline{Case 2: $k_n > k+1$} In this case we repeat the proof of Lemma \ref{nthmodel1} and write:
\begin{align*}
1.1^{10k} \Vert I_{7} ^n f \Vert_{L^2 _x} 
& \lesssim \widetilde{C_0}^{n-2} \prod_{j \in J(n)} \min \lbrace 1.1^{k_j /2} ; 1.1^{-k_j /2} \rbrace \delta^n 1.1^{10(k-k_n)} 1.1^{10k_n} \Vert \widehat{f_{k_n}} \Vert_{L^2} 
\end{align*}
and the factor $1.1^{10(k-k_n)}$ allows us to sum over $k.$ The bound follows. \\
\underline{Case 3: $k_n < k-1:$} \\
\underline{Case 3.1: $\forall j \in \lbrace 1;...; n \rbrace, k_j < k-1$ } \\
Then the first term in the product is $\widehat{V}(t,\xi-\eta_1)$ which is localized at frequency $1.1^k.$
Therefore we can repeat the proof of Lemma \ref{nthmodel1} to find that
\begin{align*}
1.1^{10 k} \Vert I_7 ^n f \Vert_{L^2 _x} & \lesssim \widetilde{C_0}^{n-2} \prod_{j \in J(n), j \neq 1} \min \lbrace 1.1^{k_j /2} ; 1.1^{-k_j /2} \rbrace 1.1^{10 k} \Vert \langle x \rangle V_{k} \Vert_{L^{\infty}_t B_x} \delta^{n-1} \varepsilon_1,
\end{align*}
where that last term on $k$ can be summed given our assumptions that $V \in L^{\infty}_t B'_x$ and we get the desired result. \\
\underline{Case 3.2: $\exists j \in \lbrace 1;...; n-1 \rbrace, k_j \geqslant k-1$ } \\
Let's consider $j' = \max \big \lbrace j \in \lbrace 1; ... ; n-1 \rbrace ; k_j \geqslant k-1 \big \rbrace .$  \\
If $k_{j'} >k+1$ then $\widehat{V}(t,\eta_{j'+1} - \eta_{j'})$ is localized at frequency $ 1.1^{k_{j'}}.$   
\\
Therefore repeating again the proof of Lemma \ref{nthmodel1} we have
\begin{align*}
1.1^{10 k} \Vert I_7 ^n f \Vert_{L^2 _x} & \lesssim \widetilde{C_0} ^{n-2} 1.1^{10 k}  \prod_{j \in J(n), j \neq j'} \min \lbrace 1.1^{k_j /2} ; 1.1^{-k_j /2} \rbrace \Vert \langle x \rangle V_{k_{j'}} \Vert_{B_x} \delta^{n-1} \varepsilon_1 \\
& \lesssim \widetilde{C_0}^{n-2} 1.1^{10 k - 10 k_{j'}}  \prod_{j \in J(n), j \neq j'} \min \lbrace 1.1^{k_j /2} ; 1.1^{-k_j /2} \rbrace \\
& \times 1.1^{10 k_{j'}} \Vert \langle x \rangle V_{k_{j'}} \Vert_{L^{\infty}_t B_x} \delta^{n-1} \varepsilon_1
\end{align*}
which can be summed.  
\\
Now if $\vert k_{j'} - k \vert \leqslant 1.$ Then either there exists $j'' \in \lbrace j'+1 ; n \rbrace$ such that $\vert k_{j''} - k_{j''-1} \vert >1$ and then the factor $\widehat{V}(t,\eta_{k_{j''}} - \eta_{k_{j''}-1})$ is localized at frequency $1.1^{k_{j''}}. $ Moreover since there are $n$ terms in the product, then $k_{j''} \geqslant k-n-1.$ Then similarly to what has been done above, we can write that
\begin{align*}
&1.1^{10 k} \Vert I_7 ^n f \Vert_{L^2 _x}  \\
& \lesssim \widetilde{C_0}^{n-2} 1.1^{10 n} 1.1^{10 k - 10 (k_{j''}+n)}  \prod_{j \in J(n), j \neq j''} \min \lbrace 1.1^{k_j /2} ; 1.1^{-k_j /2} \rbrace 1.1^{10 k_{j''}^{+}} \Vert \langle x \rangle V_{k_{j''}} \Vert_{B_x} \delta^{n-1} \varepsilon_1
\end{align*}
which yields the result. \\
Finally if for every $ j'' \in \lbrace j'+1;...;n \rbrace,$ we have $\vert k_{j''} - k_{j''-1} \vert \leqslant 1,$ then $k_n \geqslant k-n.$ Then we can write that
\begin{align*}
1.1^{10 k} \Vert I_7 ^n f \Vert_{L^2 _x} & \lesssim \widetilde{C_0}^{n-2} 1.1^{10n} 1.1^{10 k^{+} - 10 (k_{n}+n)} \\
& \times \prod_{j \in J(n)} \min \lbrace 1.1^{k_j /2} ; 1.1^{-k_j /2} \rbrace \delta^{n} 1.1^{10 k_{n}} \Vert f_{k_n} \Vert_{L^2 _x}.
\end{align*}
The second inequality is proved in the exact same way, except that one of $V$ factors is replaced by a $\partial_t V.$ 
\end{proof}
Finally we can similarly deal with $I_8 ^n f$. We start by discretizing in time. We denote $m$ the corresponding exponent, such that $1.1^{m+1} \leqslant T.$ Then we define
\begin{align*}
\mathcal{F} I_{8,m} ^n f & := \int_{1.1^m} ^{1.1^{m+1}}  \int \prod_{\gamma=1}^{n-1} \frac{\widehat{V}(s,\eta_{\gamma-1}-\eta_{\gamma}) P_{k_{\gamma}}(\eta_{\gamma})P_k(\xi)}{\vert \xi \vert^2 - \vert \eta_{\gamma} \vert^2} d\eta_{1} ... d\eta_{n-2} \\
& \times \int_{\eta_n} \widehat{f}(s, \eta_{n-1}-\eta_n) e^{-it(\vert \xi \vert^2- \vert \eta_n - \eta_{n-1} \vert^2 - \vert \eta_n \vert^2)} \widehat{f} (s,\eta_n) d\eta_n d\eta_{n-1} ds.
\end{align*}
\begin{lemma} \label{H10bilinit}
Let $J(n):= \lbrace i \in \lbrace 1;...;n \rbrace , k \leqslant k_i-2 \rbrace.$ \\ 
Then there exists a constant $a>0$ such that for all $n$ and $m$ we have:
\begin{align*}
\sum_{k=0} ^{+\infty} 1.1^{10k} \Vert I_{8,m} ^n f \Vert_{L^{\infty}_t L^2_x } \lesssim 1.1^{-am} C_1^{n-2} \prod_{j \in J(n-1)} \min \lbrace 1.1^{k_j /2} ; 1.1^{-k_j /2} \rbrace \varepsilon_1 ^2 \delta^{n-1} .
\end{align*}
The implicit constant does not depend on $n.$ The constant $C_1$ is the same as in the previous Lemma \ref{H10pot}.
\end{lemma} 
\begin{proof}
The proof is similar to the above Lemmas \ref{H10pot}, therefore we only sketch the proof.

We start by discretizing in $\eta_n,$ and denote $k_n$ the corresponding exponent. We also denote $I_{8,m,k_n}^n f$ the discretized version of $I_{8,m} ^n f.$ (that is with the additional Littlewood-Paley projector on $\eta_n$). \\
\underline{Case 1: $k_n \geqslant k-1$} \\
Using Strichartz estimates and Lemma \ref{key}, we obtain
\begin{align*}
1.1^{10k} \Vert I_{8,m} ^n f \Vert_{L^{\infty}_t L^2 _x} & \lesssim \widetilde{C_0}^{n-2} \prod_{j \in J(n-1)} \min \lbrace 1.1^{k_j/2} ; 1.1^{-k_j/2} \rbrace \delta^{n-1}  1.1^{(k - k_n )} \\
& \times 1.1 ^{k_n } \Vert f_{k_n} \Vert_{L^{\infty}_t L^2 _x} \Vert \textbf{1}_{[1.1^m;1.1^{m+1}]}(t) e^{it\Delta} f \Vert_{L^{4/3}_t L^{6}_x}.
\end{align*}
We can conclude with Lemma \ref{decay}.
\\
\underline{Case 2: $k_n < k-2$}  \\
As in Lemma \ref{H10pot}, there are several subscases to consider. They essentially depend on which factor the $\nabla^{10}$ derivative falls on. \\
\underline{Subcase 2.1: $\forall j \in \lbrace 1; ...; n-1 \rbrace, k_j < k-1$} \\
Then as we did above, using Strichartz estimates and Lemma \ref{key} and the fact that the first $V$ factor is localized at frequency $1.1^{k},$ we obtain
\begin{align*}
1.1^{10 k} \Vert  I_{8,m} ^n f \Vert_{L^{\infty}_t L^2 _x} & \lesssim \widetilde{C_0} ^{n-2} \prod_{j \in J(n-1), j \neq 1} \min \lbrace 1.1^{k_j/2} ; 1.1^{-k_j/2} \rbrace \delta^{n-2} 1.1 ^{10k} \Vert \langle x \rangle V_{k} \Vert_{L^{\infty}_t B_x} \\
& \times  \Vert f_{k_n} \Vert_{L^{\infty}_t L^2 _x} \Vert \textbf{1}_{[1.1^m;1.1^{m+1}]}(t) e^{it\Delta} f \Vert_{L^{4/3}_t L^{6}_x}.
\end{align*}
We conclude with Lemma \ref{decay}. \\
\underline{Subcase 2.2: $\exists j \in \lbrace 1;...;n-1 \rbrace, k_j \geqslant k-1$} \\
We consider $j'$ defined as in Lemma \ref{H10pot}. If $k_{j'}>k+1$ then we write using Strichartz estimates and Lemma \ref{key} that
\begin{align*}
1.1^{10 k} \Vert I_{8,m} ^n f \Vert_{L^{\infty}_t L^2 _x} & \lesssim \widetilde{C_0} ^{n-2} \prod_{j \in J(n-1), j \neq j'} \min \lbrace 1.1^{k_j/2} ; 1.1^{-k_j/2} \rbrace \delta^{n-2} 1.1 ^{10 k -10k_{j'}} \\
& \times 1.1^{10 k_{j'}} \Vert \langle x \rangle V_{k_{j'}} \Vert_{L^{\infty}_t B_x} \Vert f_{k_n} \Vert_{L^{\infty}_t L^2 _x} \Vert \textbf{1}_{[1.1^m;1.1^{m+1}]}(t) e^{it\Delta} f \Vert_{L^{4/3}_t L^{6}_x}.
\end{align*}
We conclude with Lemma \ref{decay}. \\
Now if $ \vert k_{j'} - k \vert \leqslant 1,$ we proceed as in Lemma \ref{H10pot} and distinguish whether there exists $j'' \in \lbrace j'+1 ;...; n \rbrace$ such that $\vert k_{j''} - k_{j''-1} \vert >1$ and then the factor $\widehat{V}(t,\eta_{k_{j''}} - \eta_{k_{j'' -1}})$ is localized at frequency $1.1^{k_{j''}}.$ Then we write as we did above in Lemma \ref{H10pot}
\begin{align*}
1.1^{10 k} \Vert I_{8,m} ^n f \Vert_{L^{\infty}_t L^2 _x} & \lesssim \widetilde{C_0} ^{n-2} 1.1^{10 n} \prod_{j \in J(n-1), j \neq j''} \min \lbrace 1.1^{k_j/2} ; 1.1^{-k_j/2} \rbrace \delta^{n-2} 1.1 ^{10k -10(k_{j''} +n)} \\
& \times 1.1^{10 k_{j''}} \Vert \langle x \rangle V_{k_{j'}} \Vert_{L^{\infty}_t B_x}   \Vert f_{k_n} \Vert_{L^{\infty}_t L^2 _x} \Vert \textbf{1}_{[1.1^m;1.1^{m+1}]}(t) e^{it\Delta} f \Vert_{L^{4/3}_t L^{6}_x}.
\end{align*} 
Finally if for every $j'' \in \lbrace j' + 1;...;n \rbrace,$ we have $\vert k_{j''} - k_{j''-1} \vert \leqslant 1$ and $k_n \geqslant n.$ Then we can conclude with Strichartz estimates and Lemma \ref{key}
\begin{align*}
1.1^{10 k} \Vert  I_{8,m} ^n f \Vert_{L^{\infty}_t L^2 _x} & \lesssim \widetilde{C_0}^{n-2} 1.1^{10 n} 1.1^{10 k - 10(k_n + n)} \prod_{j \in J(n-1)} \min \lbrace 1.1^{k_j/2};1.1^{-k_j/2} \rbrace  \\
& \times 1.1^{10 k_n } \Vert f_{k_n} \Vert_{L^{\infty}_t L^2 _x} \Vert e^{it\Delta} f \Vert_{L^{4/3}_t L^6 _x} \delta^{n-1}.
\end{align*}
We conclude with Lemma \ref{decay}.
\end{proof}
\subsubsection{Conclusion}
Now we explain how the above lemmas imply the other bootstap conclusion \ref{bootstrapconcl1} on the $H^{10}$ norm of the solution.
\\
\\
First note that we obtain the following expansion:
\begin{align} \label{H10expansion}
\widehat{f}(t,\xi) &= e^{i \vert \xi \vert^2} \widehat{u_1}(\xi) -\frac{i}{(2\pi)^3} \int_1 ^t e^{is \vert \xi \vert ^2} \int_{\mathbb{R}^3} e^{-is \vert \xi-\eta_1 \vert^2} e^{-i s \vert \eta_1 \vert^2} \widehat{f}(s,\eta_1) \widehat{f}(s,\xi-\eta_1) d\eta_1 ds \\
\notag&-\frac{i}{(2\pi)^3} \int_1 ^t e^{is \vert \xi \vert^2} \int_{\mathbb{R}^3} \widehat{V}(\xi - \eta_1) e^{-is \vert \eta_1 \vert ^2} \widehat{f}(s,\eta_1) d\eta_1 ds \\
& \notag + \sum_{k = 0}^{+\infty} \sum_{n = 2}^{+\infty} \sum_{k_1,...,k_{n-1} \in J(n-1)} \Bigg(  \mathcal{F} I_7 ^n f(t) - \mathcal{F}  I_7 ^n f(1) + \mathcal{F}  I_8 ^n f (t) + \sum_{l=1}^n \mathcal{F}  I_{9} ^{n,l} f(t) \Bigg).
\end{align}
Now we use this representation and the lemmas above to show \eqref{bootstrapconcl1}.
\begin{proof}[Proof of \eqref{bootstrapconcl1}]
Let $D$ denote the largest implicit constant in Lemmas \ref{H10bilin}, \ref{H10pot} and \ref{H10bilinit}. Then we write, using these three lemmas together with the representation \eqref{H10expansion}:
\begin{align*}
\Vert f(t) \Vert_{L^{\infty}_t H^{10}} & \leqslant \varepsilon_0 + D \varepsilon_1^2 + \sum_{n=2}^{+\infty}  D C_1^{n-2} \delta^n ((n+1) \varepsilon_1 + \varepsilon_0 + \varepsilon_1^2) \\
                                       & \leqslant \frac{\varepsilon_1}{2},
\end{align*} 
provided $\delta$ is small enough. 
\end{proof}
\subsection{Scattering} \label{scattering}
In this section we explain how the scattering statement from Theorem \ref{mainresultscattering} follows from the estimates proved above in Proposition \ref{stepn:estimates}. 

To explain the idea, we start with the case where $V=0.$ Then we prove the following usual scattering statement:
\begin{theorem} \label{usualscattering}
Let $u$ denote the solution constructed in Theorem \ref{mainthm} when $V=0.$ It scatters in $H^{10}$ in the sense that there exists $f_{\infty} \in H^{10}$ such that 
\begin{align*}
\lim_{t \to + \infty} \Vert e^{-it\Delta} u(t) - f_{\infty} \Vert_{H^{10}} = 0.
\end{align*}
\end{theorem}
\begin{proof}
The theorem will follow from proving that $(f(t))_t$ is Cauchy in $H^{10}.$ Let $t>\tau>0.$ We start from Duhamel's formula
\begin{align*}
\widehat{f}(t,\xi) &= e^{i \vert \xi \vert^2} \widehat{u_1}(\xi) -i \int_1 ^t e^{is \vert \xi \vert ^2} \int_{\mathbb{R}^3} e^{-is \vert \xi-\eta_1 \vert^2} e^{-i s \vert \eta_1 \vert^2} \widehat{f}(s,\eta_1) \widehat{f}(s,\xi-\eta_1) d\eta_1 ds .
\end{align*}
which yields that
\begin{align*}
\widehat{f}(t,\xi) - \widehat{f}(\tau,\xi) = -i \int_{\tau} ^t e^{is \vert \xi \vert ^2} \int_{\mathbb{R}^3} e^{-is \vert \xi-\eta_1 \vert^2} e^{-i s \vert \eta_1 \vert^2} \widehat{f}(s,\eta_1) \widehat{f}(s,\xi-\eta_1) d\eta_1 ds .
\end{align*}
Then using Lemma \ref{H10bilin}, we can conclude that $\Vert f(t) - f(\tau) \Vert_{H^{10}}$ converges to $0$ as $t,\tau$ tend to $+\infty.$ 
\end{proof}
\begin{remark}
The proof gives a rate of convergence for $\Vert f(t) - f(\tau) \Vert_{H^{10}_x}$ of $\tau^{-1/4}$ as $t>\tau, \tau \to +\infty.$
\end{remark}
Now we see how to adapt this idea when the potential is present. The main obstruction to proving a similar result in this case are the boundary terms $I_7^n f$ in the expansion above. Indeed they are not integrable in time as the integrands of the other two terms $I_8 ^n f$ and $I_9 ^n f$ are. They will form the correction $W_V$ from Theorem \ref{mainresultscattering}. 
\begin{proof}[Proof of Theorem \ref{mainresultscattering}]
Let $t>\tau>0.$ By taking the difference $f(t)-f(\tau)$ of the two expansions \eqref{H10expansion} at times $t$ and $\tau$, we obtain
\begin{align*}
& \Bigg \Vert \Big( f(t) - \sum_{k=0}^{+\infty} \sum_{n=2}^{+\infty} \sum_{k_1,...,k_{n-1} \in J(n-1)} I_7 ^n f(t) \Big) - \Big(f(\tau) - \sum_{k = 0}^{+\infty} \sum_{n=2}^{+\infty} \sum_{k_1,...,k_{n-1} \in J(n-1)} I_7 ^n f(\tau) \Big) \Bigg \Vert_{H^{10}} \\
\leqslant &  \Bigg \Vert \mathcal{F}_{\xi}^{-1} \int_{\tau} ^t \int_{\mathbb{R}^3} e^{is(\vert \xi \vert^2 - \vert \eta_1 \vert^2 - \vert \xi-\eta_1 \vert^2} \widehat{f}(s,\eta_1) \widehat{f}(s,\xi-\eta_1) d\eta_1 ds \Bigg \Vert_{H^{10}_x}    \\
& + \Bigg \Vert \sum_{k=0}^{+\infty} \sum_{n=2}^{+\infty} \sum_{k_1,...,k_{n-1} \in J(n-1)} I_8 ^n f(t) - I_8 ^n f(\tau) \Bigg \Vert_{H^{10}} + \Bigg \Vert \sum_{k=0}^{+\infty} \sum_{n=2}^{+\infty} \sum_{k_1,...,k_{n-1} \in J(n-1)} \sum_{l=1}^n I_9 ^{n,l} f(t) - I_9 ^{n,l} f(\tau) \Bigg \Vert_{H^{10}} \\
\leqslant & \tau^{-a} D \varepsilon_1^2 + \tau^{-a} D \sum_{n=2}^{+\infty} C_2^{n-1} \delta^{n} \varepsilon_1^2 + D \Big( \int_{\tau} ^t \Vert \partial_s V \Vert_{B_x '} ds \Big) \sum_{n=2}^{+\infty} C_2^{n-1} \delta^{n-1} \varepsilon_1.
\end{align*}
We used Lemmas \ref{H10bilin}, \ref{H10pot} and \ref{H10bilinit} to write the last line. As above, $D$ denotes the largest implicit constant in these lemmas, and $a$ the smallest value from Lemmas \ref{H10bilin} and \ref{H10bilinit}. We now define
\begin{align*}
&W_V(t) u(t) = \\
&u(t) - \sum_{k=0}^{+\infty} \sum_{n=2}^{+\infty} \sum_{k_1,...,k_{n-1} \in J(n-1)} \frac{i^n}{(2\pi)^{3n}} \int_{\big(\mathbb{R}^3 \big)^{n}} \prod_{\gamma=1}^{n} \frac{\widehat{V}(t,\eta_{\gamma} - \eta_{\gamma-1})P_k(\xi) P_{k_{\gamma}}(\eta_{\gamma})}{\vert \xi \vert^2 - \vert \eta_{\gamma} \vert^2}  \widehat{u} (t,\eta_n) d\eta_{1} ... d\eta_{n}.
\end{align*} 
The boundedness on $H^{10}$ of this operator follows from Lemma \ref{H10pot}. \\
The above estimate shows that $\big(e^{-it\Delta}W_V(t) u(t) \big)$ is Cauchy in $H^{10},$ hence the desired result.
\end{proof}

\newpage

\appendix

\section{Explicit expressions of iterates}
In this short appendix we give the explicit expressions of terms that appeared in the expansion in Section 4.2. \\
First in the case $\vert k-k_2 \vert >1:$
\begin{align}
& \notag \lbrace \textrm{similar terms 1} \rbrace = \\
-&\label{2R2prime} \frac{2}{(2\pi)^6} \int_{\mathbb{R}^3} \frac{P_k (\xi) P_{k_1} (\eta_1) \widehat{V}(1,\xi-\eta_1)}{\vert \xi \vert^2 - \vert \eta_1 \vert^2} \\
\notag & \times \int_{\mathbb{R}^3 } i \xi_l e^{it(\vert \xi \vert^2 - \vert \eta_2 \vert^2)} \frac{\widehat{V}(1,\eta_1-\eta_2)}{\vert \xi \vert^2 - \vert \eta_2 \vert^2} \widehat{f_{k_2}}(1,\eta_2) d\eta_2 d\eta_1 \\
+&\label{2R2bis} \frac{2}{(2\pi)^6} \int_1^t \int_{ \mathbb{R}^3} \frac{P_k (\xi) P_{k_1} (\eta_1) \widehat{V}(s,\xi-\eta_1)}{\vert \xi \vert^2 - \vert \eta_1 \vert^2} \\
\notag & \times \int_{\mathbb{R}^3 } is \xi_l e^{it(\vert \xi \vert^2 - \vert \eta_2 \vert^2)} \frac{\partial_s \widehat{V}(s,\eta_1-\eta_2)}{\vert \xi \vert^2 - \vert \eta_2 \vert^2} \widehat{f_{k_2}}(t,\eta_2) d\eta_2 d\eta_1 ds \\
+&\label{2R2bis'} \frac{2}{(2\pi)^6} \int_1^t \int_{\mathbb{R}^3} \frac{P_k (\xi) P_{k_1} (\eta_1) \partial_{s} \widehat{V}(s,\xi-\eta_1)}{\vert \xi \vert^2 - \vert \eta_1 \vert^2} \\
\notag & \times \int_{\mathbb{R}^3} is \xi_l e^{it(\vert \xi \vert^2 - \vert \eta_2 \vert^2)} \frac{\widehat{V}(s,\eta_1-\eta_2)}{\vert \xi \vert^2 - \vert \eta_2 \vert^2} \widehat{f_{k_2}}(t,\eta_2) d\eta_2 d\eta_1 ds \\
+&\label{2R6prime} \frac{2}{(2\pi)^6} \int_1^t \int_{ \mathbb{R}^3} \frac{P_k (\xi) P_{k_1} (\eta_1) \widehat{V}(s,\xi-\eta_1)}{\vert \xi \vert^2 - \vert \eta_1 \vert^2} \\
\notag & \times \int_{ \mathbb{R}^3 } i \xi_l e^{it(\vert \xi \vert^2 - \vert \eta_2 \vert^2)} \frac{\widehat{V}(s,\eta_1-\eta_2)}{\vert \xi \vert^2 - \vert \eta_2 \vert^2} \widehat{f_{k_2}}(t,\eta_2) d\eta_2 d\eta_1 ds \\
-&\label{2B1} \frac{2i}{(2\pi)^9} \int_1 ^t \int_{ \mathbb{R}^3 } \frac{P_k (\xi) P_{k_1} (\eta_1) \widehat{V}(s,\xi-\eta_1)}{\vert \xi \vert^2 - \vert \eta_1 \vert^2} \int_{\mathbb{R}^3 } is \eta_{1,l} e^{is(\vert \xi \vert^2 - \vert \eta_2 \vert^2)} P_{k_2}(\eta_2)  e^{is\vert \eta_2 \vert^2} \\
& \notag \times \frac{\widehat{V}(s,\eta_1-\eta_2)}{\vert \xi \vert^2 - \vert \eta_2 \vert^2}  \int_{ \mathbb{R}^3} e^{-is \vert \eta_3 \vert^2} \widehat{f}(s,\eta_3) e^{-is \vert \eta_2 - \eta_3 \vert^2} \widehat{f}(s,\eta_2-\eta_3) d\eta_3 d\eta_2 d\eta_1 ds .
\end{align}
Then for the same term, but in the case $\vert k-k_2 \vert \leqslant 1:$
\begin{align}
&\notag \lbrace \textrm{similar terms 2} \rbrace = \\
&\label{2R4} \frac{i}{(2\pi)^6} \int_1 ^t \int_{\mathbb{R}^3} \frac{P_k (\xi) P_{k_1} (\eta_1) \widehat{V}(s,\xi-\eta_1)}{\vert \xi \vert^2 - \vert \eta_1 \vert^2} \\
\notag & \times \int_{ \mathbb{R}^3} e^{is(\vert \xi \vert^2 - \vert \eta_2 \vert^2)} \partial_{{\eta}_{2,j}} \bigg( \frac{\xi_l \eta_{2,j}}{\vert \eta_2 \vert^2} \bigg) \widehat{V}(s,\eta_1-\eta_2) \widehat{f_{k_2}}(s,\eta_2) d\eta_2 d\eta_1 ds \\
&\label{autre}\frac{i}{(2\pi)^6} \int_1 ^t \int_{ \mathbb{R}^3} \frac{P_k (\xi) P_{k_1} (\eta_1) \widehat{V}(s,\xi-\eta_1)}{\vert \xi \vert^2 - \vert \eta_1 \vert^2} \int_{\mathbb{R}^3 } e^{is(\vert \xi \vert^2 - \vert \eta_2 \vert^2)} 1.1^{-k_1} \\
& \notag \times \phi'(1.1^{-k_1} \vert \xi \vert) \frac{\xi_j}{\vert \xi \vert} \frac{\xi_l \eta_{2,j}}{\vert \eta_2 \vert^2} \widehat{V}(s,\eta_1-\eta_2) \widehat{f_{k_2}}(s,\eta_2) d\eta_2 d\eta_1 ds .
\end{align}
And finally:
\begin{align}
\notag \lbrace \textrm{similar terms 3} \rbrace &= \\ 
 &- \frac{1}{(2\pi)^6} \label{2R5prime} \int_{\mathbb{R}^3} \frac{P_k (\xi) P_{k_1} (\eta_1) \widehat{V}(1,\xi-\eta_1)}{\vert \xi \vert^2 - \vert \eta_1 \vert^2} \\
                           & \notag \times \int_{\mathbb{R}^3} \frac{\xi_l \eta_{2,j}}{\vert \eta_2 \vert ^2} \frac{P_k (\xi) P_{k_2} (\eta_2) \widehat{V}(1,\eta_1-\eta_2)}{\vert \xi \vert^2 - \vert \eta_2 \vert^2} e^{i(\vert \xi \vert^2 - \vert \eta_2 \vert^2)} \partial_{\eta_{2,j}} \widehat{f}(1,\eta_2)P_{k_2}(\eta_2) d\eta_2 d\eta_1\\
                            &- \frac{1}{(2\pi)^6} \label{2R5bis} \int_1 ^t  \int_{\mathbb{R}^3 } \frac{P_k (\xi) P_{k_1} (\eta_1) \widehat{V}(s,\xi-\eta_1)}{\vert \xi \vert^2 - \vert \eta_1 \vert^2} \\
                            & \notag \times \int_{\mathbb{R}^3 } \frac{\xi_l \eta_{2,j}}{\vert \eta_2 \vert ^2} \frac{\partial_s \widehat{V}(s,\eta_1-\eta_2)}{\vert \xi \vert^2 - \vert \eta_2 \vert^2 } e^{i s( \vert \xi \vert^2 - \vert \eta_2 \vert^2)} \partial_{\eta_{2,j}} \widehat{f}(s,\eta_2) P_{k_2}(\eta_2) d\eta_2 d\eta_1 ds .
\end{align}

\end{document}